\documentclass[10pt,reqno]{amsart}

\usepackage{amsmath,amscd,amssymb,amsthm,mathtools,thmtools,ulem}
\usepackage{tikz-cd}
\usepackage{thm-restate}

\usepackage{graphicx,adjustbox}
\usepackage{comment}

\usepackage[pagebackref,hyperfootnotes]{hyperref}
\usepackage{color,units}
\hypersetup{
	colorlinks,
	linkcolor={red!50!black},
	citecolor={green!50!black},
	urlcolor={blue!80!black}
}
\usepackage[capitalize,nameinlink,noabbrev]{cleveref}

\makeatletter
\def\th@plain{%
  \thm@notefont{}
  \itshape 
}
\def\th@definition{%
  \thm@notefont{}
  \normalfont 
}
\makeatother

\usepackage{enumitem}
\setlist[enumerate]{label=(\roman*),leftmargin=0.8cm}
\usepackage{todonotes}

\newtheorem{proposition}{Proposition}[section]
\newtheorem{lemma}[proposition]{Lemma}
\newtheorem{theorem}[proposition]{Theorem}
\newtheorem{corollary}[proposition]{Corollary}
\newtheorem{conjecture}[proposition]{Conjecture}

\theoremstyle{definition}
\newtheorem{remark}[proposition]{Remark}
\newtheorem{definition}[proposition]{Definition}

\newtheorem{example}[proposition]{Example}

\numberwithin{equation}{section} 

\DeclareMathOperator{\End}{End}
\DeclareMathOperator{\Id}{Id}
\DeclareMathOperator{\tr}{tr}
\DeclareMathOperator{\rk}{rk}
\DeclareMathOperator{\ch}{ch}
\DeclareMathOperator{\Euc}{Euc}

\DeclareMathOperator{\GL}{GL}
\DeclareMathOperator{\Uni}{U}
\DeclareMathOperator{\Aut}{Aut}

\DeclareMathOperator{\Ima}{Im}

\DeclareMathOperator{\Lie}{Lie}

\DeclareMathOperator{\FS}{FS}

\renewcommand{\phi}{\varphi}

\newcommand{\caT}{\mathcal{T}}

\DeclareMathOperator{\Hom}{Hom}

\renewcommand{\Re}{\operatorname{Re}}
\renewcommand{\Im}{\operatorname{Im}}

\newcommand{\R}{\mathbb{R}}
\newcommand{\C}{\mathbb{C}}

\newcommand{\Z}{\mathbb{Z}}

\newcommand{\G}{\mathcal{G}}
\newcommand{\pr}{\mathbb{P}}
\let\oldepsilon\epsilon
\renewcommand{\epsilon}{\varepsilon}
\renewcommand{\H}{\mathcal{H}}
\newcommand{\HH}{\mathbb{H}}

\newcommand{\D}{\mathcal{D}}

\newcommand{\scA}{\mathcal{A}}

\renewcommand{\deg}{\operatorname{deg}}

\newcommand{\deriv}[2]{\frac{d #1}{d #2}} 
\newcommand{\pderiv}[2]{\frac{\partial #1}{\partial #2}} 

\newcommand{\im}{\operatorname{im}}
\newcommand{\contr}{\Lambda}

\newcommand{\ddb}{i\partial \bar\partial}
\newcommand{\scL}{\mathcal{L}}

\newcommand{\mfg}{\mathfrak{g}}

\newcommand{\F}{\mathcal{F}}

\newcommand{\mfk}{\mathfrak{k}}

\newcommand{\scD}{\mathcal{D}}

\newcommand{\db}{\overline\partial}

\newcommand{\mfu}{\mathfrak{u}}

\newcommand{\ddbar}{\partial\overline{\partial}}
\newcommand{\dbard}{\overline{\partial} \partial}
\newcommand{\del}{\partial}
\DeclareMathOperator{\Coh}{Coh}
\DeclareMathOperator{\Td}{Td}
\newcommand{\Lap}{\Delta}
\DeclareMathOperator{\aut}{\mathfrak{aut}}

\newcommand{\Gr}{\operatorname{Gr}}
\newcommand{\delbar}{\overline{\partial}}

\let\oldchi\chi
\renewcommand{\chi}{\protect\raisebox{\depth}{$\oldchi$}}

\makeatletter
\renewcommand*{\eqref}[1]{%
	\hyperref[{#1}]{\textup{\tagform@{\ref*{#1}}}}%
}
\makeatother

\let\oldtocsection=\tocsection
\let\oldtocsubsection=\tocsubsection
\let\oldtocsubsubsection=\tocsubsubsection
\renewcommand{\tocsection}[2]{\hspace{0em}\oldtocsection{#1}{#2}}
\renewcommand{\tocsubsection}[2]{\hspace{2em}\oldtocsubsection{#1}{#2}}
\renewcommand{\tocsubsubsection}[2]{\hspace{4.5em}\oldtocsubsubsection{#1}{#2}}

\title[$Z$-critical connections and Bridgeland stability conditions]{$Z$-critical connections and Bridgeland stability conditions}

\author[Ruadha\'i Dervan, John Benjamin McCarthy and Lars Martin Sektnan]{Ruadha\'i Dervan, John Benjamin McCarthy and Lars Martin Sektnan}
\address{Ruadha\'i Dervan, DPMMS, Centre for Mathematical Sciences, Wilberforce Road, Cambridge CB3 0WB, United Kingdom}
\address{School of Mathematics and Statistics, University of Glasgow, University Place, Glasgow G12 8QQ, United Kingdom}\email{ruadhai.dervan@glasgow.ac.uk}

\address{John Benjamin McCarthy, Department of Mathematics, South Kensington Campus
London, SW7 2AZ, United Kingdom}
\email{j.mccarthy18@imperial.ac.uk}

\address{Lars Martin Sektnan, Institut for Matematik, Aarhus University, 8000, Aarhus C, Denmark}
\address{Department of Mathematical Sciences, University of Gothenburg, 412 96 Gothenburg, Sweden}
\email{sektnan@chalmers.se}

\begin{document}
\normalem

\begin{abstract} We associate geometric partial differential equations on holomorphic vector bundles to Bridgeland stability conditions. We call solutions to these equations $Z$-critical connections, with $Z$ a central charge. Deformed Hermitian Yang--Mills connections are a special case. We explain how our equations arise naturally through infinite dimensional moment maps. 

Our main result shows that in the large volume limit, a sufficiently smooth holomorphic vector bundle admits a $Z$-critical connection if and only if it is asymptotically $Z$-stable. Even for the deformed Hermitian Yang--Mills equation, this provides the first examples of solutions in higher rank. 

\end{abstract}

\maketitle

\setcounter{tocdepth}{2}
\tableofcontents

\section{Introduction}

The Hitchin--Kobayashi correspondence, established in the 1980s \cite{donaldson1985anti, uhlenbeck1986existence}, states that a simple holomorphic vector bundle admits a Hermite--Einstein metric if and only if it is slope stable. This links a purely algebro-geometric notion, namely slope stability, with the existence of solutions to a geometric partial differential equation, namely the Hermite--Einstein equation. This correspondence has had many spectacular applications, for example to the moduli theory of holomorphic vector bundles, and to four-dimensional topology.

A much more recent notion of stability arises from Bridgeland's stability conditions \cite{bridgeland2007stability}. There are many aspects to Bridgeland's notion of a stability condition, but for us the most important point will be that for a given stability condition, one can still ask for a vector bundle to be stable. Bridgeland stability conditions have themselves had many spectacular applications in algebraic geometry, and these applications have emphasised that there is typically no canonical choice of stability condition \cite{bayer-survey}. Indeed, many of these applications arise from the very fact that one can vary the stability condition. 

Motivated by the Hitchin--Kobayashi correspondence, it is natural to ask whether or not there is a differential-geometric counterpart of Bridgeland stability conditions. This is the question we answer in the present work. To a given stability condition, we associate a geometric partial differential equation in such a way that existence of solutions to these equations on a given holomorphic vector bundle should be equivalent to stability of the bundle. In fact, our recipe requires fewer hypotheses than the full set required in Bridgeland's definition of a stability condition, and so our work is completely unreliant on the difficult problem of actually \emph{constructing} Bridgeland stability conditions.

We briefly describe the recipe, more precise details are provided in \cref{sec:prelims}. A primary input into a Bridgeland stability condition is a \emph{central charge}; this is a function which assigns to each coherent sheaf a complex number. The central charges that we consider are \emph{polynomial central charges}, which were introduced in the work of Bayer \cite{bayer}. These take the form $$Z_{\Omega,k}(E) := \int_X \sum_{d=0}^n \rho_d k^d [\omega]^d \cdot \ch(E) \cdot U,$$ with $[\omega]$  a K\"ahler class on the $n$-dimensional compact K\"ahler manifold $X$, $\ch(E)$ the total Chern character of the sheaf $E$, $U$ a unipotent operator, $\rho_d$ a sequence of non-zero complex numbers with $\Im(\rho_n) > 0$,  and $k$ a positive integer. Assuming $E$ is a holomorphic vector bundle, taking Chern--Weil representatives of the bundle $E$ defined using a Hermitian metric with Chern connection $A$, an arbitrary representative $\tilde U$ of $U$ and a K\"ahler metric $\omega \in [\omega]$, one obtains a complex $(n,n)$-form $\tilde Z_{\Omega,k}(A)$ representing $Z_{\Omega,k}(E).$ The equation we associate takes the form $$\Im \left( e^{-i\varphi(E)} \tilde{Z}_{\Omega,k}(A) \right) = 0,$$ which can be viewed as a fully non-linear partial differential equation on the space of  Hermitian metrics on $E$ (or equivalently, on the space of integrable connections compatible with a fixed Hermitian metric). The topological constant $\phi$ is defined so that the \emph{integral} of the equation vanishes, a necessary condition for the existence of a solution. We call a connection solving our equation a $Z$-\emph{critical connection}, emphasising the role played by the choice of central charge; we then call the  equation the $Z$-\emph{critical equation}.

Before turning to our main results, which justify our recipe, we mention that there is one central charge for which there is a well-studied associated partial differential equation. This is the central charge $$Z(E) = -\int_X e^{-i[\omega]}\ch(E),$$ to which is associated the \emph{deformed Hermitian Yang--Mills equation}  $$   \Im \left(e^{-i\phi(E)} \left( \omega \otimes \Id_E - \frac{F_A}{2\pi}\right)^n\right) = 0,$$ where $F_A$ is the curvature of the connection $A$. This equation was originally derived in the physics literature for the case of line bundles or Abelian gauge group \cite{LYZ,marino2000nonlinear}, where it is related to problems in Type IIB string theory. Mathematically, the equation arises through SYZ mirror symmetry as the dual to the special Lagrangian equation on a Calabi--Yau manifold. In higher rank, the equation given above was suggested by Collins--Yau \cite[\S 8.1]{collins2018moment}. There are by now many important results known about deformed Hermitian Yang--Mills connections on holomorphic line bundles which have largely motivated our work, beginning with the foundational results of Jacob--Yau \cite{jacob2017special}, and including for example a moment map interpretation of the equation due to Collins--Yau \cite[\S 2]{collins2018moment}, a complete analytic solution of the existence problem, due to Collins--Jacob--Yau \cite{collins20151}, and a complete algebro-geometric solution, due to Chen \cite{chen2019j,chen2020supercritical}. All of these results concern the equation on holomorphic line bundles, and until the present work nothing has been known in higher rank.

We now turn to our main results. Note that a polynomial central charge $Z_{\Omega,k}(E)$ generates a sequence of central charges by varying $k$. Our first main result concerns the \emph{large volume limit}, when $k \gg 0$ is taken to be large. This is the most important regime in string theory; the equations are not expected to be physically relevant away from this limit. This is also the regime considered by Bayer in his work on polynomial stability conditions \cite{bayer}.

\begin{restatable}{theorem}{intromainthm}\label{thm:intromainthm} A simple, sufficiently smooth holomorphic vector bundle admits a uniformly bounded family of $Z_k$-critical connections for all $k \gg 0$ if and only if it is asymptotically $Z$-stable with respect to subbundles. \end{restatable}

This solves the existence problem for sufficiently smooth bundles in the large volume regime, and provides a direct analogue of the Hitchin--Kobayashi correspondence for $Z$-critical connections and Bridgeland stability conditions. We briefly explain the terminology, and refer to \cref{sec:perturbation} for fully precise definitions and statements. The notion of asymptotic $Z$-stability is a variant of Bridgeland and Bayer's stability conditions, which means that for each holomorphic subbundle $F \subset E$ we have for all $k \gg 0$ strict inequality \begin{equation}\label{zstable}\arg Z_{\Omega,k}(F) <\arg Z_{\Omega,k}(E),\end{equation} where $\arg$ denotes the argument of a complex number; this is distinct from the notions used by Bridgeland and Bayer. The fact that we consider only subbundles rather than arbitrary subsheaves is justified by our assumption that, following the terminology of Leung (introduced in the context of Gieseker stability \cite{leung1997einstein}), $E$ is \emph{sufficiently smooth}. As we show, asymptotic $Z$-stability with respect to subsheaves implies slope semistability, and sufficient smoothness means that the slope polystable degeneration of the slope semistable vector bundle is itself a holomorphic vector bundle (rather than a reflexive sheaf, in general); this assumption is crucial to the analysis. When the vector bundle is actually \emph{slope stable} rather than slope semistable, the obstructions we encounter in the proof of \cref{thm:intromainthm} vanish, producing the following:

\begin{corollary} For any polynomial central charge $Z_{\Omega,k}$, a slope stable vector bundle $E$ admits $Z_k$-critical connections for all $k \gg 0.$

\end{corollary}

We emphasise that these results are new even in the case of deformed Hermitian Yang--Mills connections. In particular, this gives the first construction of deformed Hermitian Yang--Mills connections on a vector bundle of rank at least two. It is worth remarking that there has been no derivation of the deformed Hermitian Yang--Mills equation on a holomorphic vector bundle of rank at least two, either in the physical literature or the mathematical literature, and our results thus give strong mathematical justification that the higher rank deformed Hermitian Yang--Mills equation suggested by Collins--Yau is indeed the appropriate equation. In general, the criteria of \cref{thm:intromainthm} can be explicitly checked in concrete examples, and we will later use it to give examples of slope semistable bundles which do (and do not) admit $Z$-critical connections for especially interesting choices of central charge $Z$.

Having provided algebro-geometric motivation for our notion of a $Z$-critical connection through the link with Bridgeland stability conditions, we turn to the differential-geometric motivation, which arises  from the theory of moment maps.

\begin{theorem}\label{intro:moment}
	Fix a Hermitian vector bundle $(E,h)$. Consider the action of the group $\G$ of unitary gauge transformations on the space $\scA(h)$ of $h$-unitary connections inducing integrable complex structures on $E$. Then for each polynomial central charge $Z_{\Omega}$, there exists a form $\eta_{Z_{\Omega}}$ on $\scA(h)$ such that the map $$D: A \to \Im \left( e^{-i\varphi(E)} \tilde{Z}_{\Omega}(E)(A) \right)$$ is a moment map for the $\G$-action. Moreover, $\eta_{Z_{\Omega}}$ is a K\"ahler metric on the locus of $\scA(h)$ consisting of subsolutions.
\end{theorem}

This extends a result of Collins--Yau, who gave a moment map interpretation for the deformed Hermitian Yang--Mills equation on line bundles \cite[\S 2]{collins2018moment}. A general approach to moment map problems for holomorphic vector bundles was suggested by Thomas \cite[p. 737]{clay}, and our strategy begins in a similar direction. From the moment map setup, one obtains several formal consequences, such as a topological moduli space of solutions; this moduli space suggests differential-geometric approach to studying moduli of Bridgeland stable objects.

A very important aspect of the above is the \emph{subsolution} hypothesis, which ensures $\eta_{Z_{\Omega}}$  is a K\"ahler metric (in general, the moment map condition is meant in a formal sense). In the theory of deformed Hermitian Yang--Mills connections on holomorphic line bundles, a central role is played by Collins--Jacob--Yau's notion of a \emph{subsolution} \cite[\S 3, \S 8]{collins20151}, which is a condition requiring positivity of an associated $(n-1,n-1)$-form, closely analogous  to  notions  of Song--Weinkove \cite{song2008convergence}, Sz\'ekelyhidi \cite{szekelyhidi2018fully} and Guan \cite{guan2014second} for other partial differential equations. We introduce a notion of a subsolution for the $Z$-critical equation even in higher rank; we expect that, at least in a suitable phase range of values of $\arg Z_{\Omega}(E)$, the existence of a solution to the $Z$-critical equation implies the existence of a subsolution. We demonstrate two other important situations in which our notion of a subsolution arises; one is through question of ellipticity of the $Z$-critical equation (which cannot hold in complete generality), the other is through the following existence result for $Z$-critical connections on holomorphic line bundles over complex surfaces:

\begin{theorem}\label{thm:introsurfaces} Suppose $L \to X$ is a holomorphic line bundle, and suppose $Z_{\Omega}(L)$ satisfies the volume form hypothesis. Then the following are equivalent:
\begin{enumerate}
\item $L$ admits a $Z$-critical connection;
\item $L$ admits a subsolution;
\item for all curves $C \subset X$ we have $\arg Z_{\Omega, C}(L) > \arg Z_{\Omega,X}(L)$.
\end{enumerate}
\end{theorem}

The volume form hypothesis, as well as the remaining terminology, is explained in \cref{sec:momentmaps}. This result generalises work of Jacob--Yau on  deformed Hermitian Yang--Mills connections \cite[Theorem 1.2]{jacob2017special}, and provides explicit examples of $Z$-critical connections away from the large-volume limit, in terms of an algebro-geometric notion of stability involving \emph{subvarieties} rather than \emph{subbundles}. The proof reduces to a complex Monge--Amp\`ere equation, for which the existence of solutions is fully understood through Yau's solution of the Calabi conjecture \cite{yau}. We are optimistic that many of the techniques which were later developed in the study of the deformed Hermitian Yang--Mills equation on line bundles may be adapted to our equations, though certainly many new ideas are needed as several new analytic issues arise. In higher rank, away from the large volume limit it seems highly nontrivial to understand the existence of $Z$-critical connections, due to the fact that the $Z$-critical equation is non-linear in the \emph{curvature} of the connection.

The proof of \cref{thm:intromainthm} relies on the Hitchin--Kobayashi correspondence, which produces a Hermite--Einstein metric on the polystable degeneration $\Gr(E)$ of $E$ \cite{donaldson1985anti, uhlenbeck1986existence}. While one cannot expect to perturb this Hermite--Einstein metric on $\Gr(E)$ to a $Z_k$-critical connection on $E$ in general, we show that the obstruction to this can precisely be understood in terms of asymptotic $Z$-stability. We use a new strategy to solve this problem, which we expect to be applicable to other obstructed moment map problems. Our approach is to consider a $k$-dependent \emph{family} of gradient flows to solve an associated finite-dimensional problem, where we precisely see asymptotic $Z$-stability enter as the relevant obstruction. This step uses the moment map interpretation of the $Z_k$-critical equation proven as \cref{intro:moment}. We then construct approximate solutions to the $Z_k$-critical equation, and study in detail the mapping properties of the linearisation of the $Z_k$-critical operator to produce solutions of the $Z_k$-critical equation equation for $k \gg 0$.

We mention influential work of Leung, which proves an analogous result to \cref{thm:intromainthm} for Gieseker stable vector bundles (which are hence slope semistable) and his almost Hermite--Einstein metrics \cite{leung1997einstein}. Leung's work appears to contain gaps, however, especially when the graded object has several components; this case is only briefly sketched by Leung. No details are given on the application of the implicit function theorem, and the difficulties caused by having isomorphic components in the graded object are not discussed (which, in our associated finite-dimensional problem, leads to the relevant groups being general reductive Lie groups rather than  complex tori). Thus we are lead to use a different approach to Leung. In the case when the associated graded object of $E$ has only two components, a more direct strategy (unreliant on gradient flows and moment maps) succeeds. The argument in this simpler case is analogous to  work of the third author and Tipler \cite{sektnan2020hermitian} on a different problem, and the direct proof of this case is detailed in the PhD thesis of the second author \cite{mccarthy}.

We remark that we do use Leung's ideas to show that the symplectic form used in Theorem \ref{intro:moment} is actually closed \cite{leung1998symplectic}, and also to show that the existence of $Z$-critical connections implies asymptotic $Z$-stability with respect to subbundles; our proof of the the latter is modelled on Leung's proof that the existence of almost Hermite--Einstein metrics implies Gieseker stability  \cite[\S 3]{leung1997einstein}. In any case, one should expect that the techniques going into the study of Gieseker stability and almost Hermite--Einstein metrics are necessarily different to those used to study $Z$-critical connections. This is because Gieseker stability is a genuine notion of stability arising from finite-dimensional geometric invariant theory, and typically does \emph{not} induce a Bridgeland stability condition, meaning almost Hermite--Einstein metrics do not fit into our framework.

As well as the central charge $Z(E) = -\int_X e^{-i[\omega]}\ch(E),$ which corresponds to the deformed Hermitian Yang--Mills equation, string theorists are also interested in the central charge $$Z_{\mathrm{Todd}}(E) = -\int_X e^{-i[\omega]} e^{-\beta} \ch(E) \cdot \sqrt{\Td(X)},$$ which is thought of as the charge of a D-brane in Type IIB string theory. This was the form in which the precursor to Bridgeland stability was first introduced by Douglas \cite{douglas2005stability}, then known as $\Pi$-stability (see for example \cite[\S 2.3]{aspinwall2002d} and the references therein for more discussion of the relations to string theory). Through our recipe this central charge corresponds to the equation \begin{equation}\label{eqn:introtodd} \Im(e^{-i\phi(E)}\tilde{Z}_{\mathrm{Todd}}(E) ) = 0.\end{equation} This equation is equivalent to one which has previously been considered by the physicists Enger-L\"utken \cite{enger-lutken} (as we learned near the completion of our work), who studied the case of a Calabi--Yau threefold and who proved much of \cref{thm:intromainthm} on Calabi--Yau threefolds for $Z_{\mathrm{Todd},k}$ and \eqref{eqn:introtodd} assuming the B-field $\beta$ vanishes; we note that their work is modelled on Leung's, and hence the same comments as above apply.

\subsection{Outlook}

As mentioned above, Leung--Yau--Zaslow's mathematical derivation of the deformed Hermitian Yang--Mills equation on a line bundle arises through SYZ mirror symmetry \cite{LYZ}. In SYZ mirror symmetry, Calabi--Yau manifolds are expected to come in mirror pairs, in which the symplectic and complex geometry are interchanged \cite{SYZ}. The deformed Hermitian Yang--Mills equation was then shown by Leung--Yau--Zaslow to be the mirror equation to the \emph{special Lagrangian equation}, which is a non-linear partial differential equation on the space of Lagrangian submanifolds of the Calabi--Yau manifold. The central conjecture, primarily due to Thomas and Thomas--Yau \cite{thomas2001moment,thomas2001special}, and as it relates to Bridgeland stability by Joyce \cite{joyce2014conjectures}, then states that there is a Bridgeland stability condition on the Fukaya category $\scD^b \F(X,\omega)$ of a Calabi--Yau manifold $(X,\omega)$ such that stability corresponds to the existence of a special Lagrangian in the given isotopy class of Lagrangians. The following is the ``mirror'' conjecture to our main results, and suggests that there should be variants of the special Lagrangian condition for which the Thomas--Yau conjecture holds. 

\begin{conjecture}\label{conjecture:fukaya}
There is a class of stability conditions on the Fukaya category such that for each stability condition in this class, there is an associated partial differential equation on the space of Lagrangians such that stability of an isotopy class of Lagrangians is equivalent to the existence of a Lagrangian in the isotopy class solving the equation, in such a way that the special Lagrangian equation is a special case.

\end{conjecture} 

The class of stability conditions should be mirror to the polynomial stability conditions we consider for holomorphic vector bundles. The conjecture is left deliberately vague; as explained by Joyce, there are many subtleties to this conjecture even in the case of special Lagrangians \cite{joyce2014conjectures}. For example, it is reasonable to think that such a correspondence only holds near the large complex structure limit, which is mirror to the large volume limit we have considered for $Z$-critical connections. Furthermore, the stability condition near this limit should initially be viewed only as a limiting case of a Bridgeland-type stability condition, just as our asymptotic $Z$-stability is a simplified limiting version of Bayer's polynomial Bridgeland stability conditions. There is still much work to formalise the link between Bridgeland stability and the stability conditions that have previously appeared in relation to special Lagrangians and dHYM connections, and subsequently our $Z$-critical connections. 

We make two slightly more precise remarks. Firstly, the most straightforward way to vary a stability condition on $\scD^b\Coh(X)$, the derived category of coherent sheaves, is through the addition of a $B$-field, which requires complexifying the K\"ahler metric. On the symplectic side, the role of K\"ahler metrics and complex structures and interchanged, meaning that the analogus variant of the special Lagrangian condition would likely land in the realm of Hitchin's generalised geometry \cite{hitchin}. Secondly, Smith explains that one expects to see the Teichm\"uller space of compatible complex structures of the underlying symplectic manifold to appear as a distinguished subspace of the the space of stability conditions on the Fukaya category of the Calabi--Yau, and each stability condition arising in this manner would correspond to the special Lagrangian condition with respect to that choice of complex structure \cite[\S 3.1]{smith2017stability}. It is thus natural to expect that the slice of stability conditions we conjecturally describe in \cref{conjecture:fukaya} should be, in some sense, transversal to this slice inside the stability manifold.

Returning to the realm of complex geometry, we speculate about the existence of $Z$-critical connections away from the large volume regime. As demonstrated by our main results, in higher rank one should consider not only \emph{subsheaves} but also \emph{subvarieties}. 

\begin{conjecture} $E$ admits a $Z$-critical connection if and only if for all coherent sheaves $F$ supported on a subvariety $V\subset X$ with $F \subset E|_{V}$ and with quotient $Q = (E|_{V})/F$ one has $$\arg Z_{\Omega,V}(Q) > \arg Z_{\Omega,X}(E).$$
\end{conjecture}

As in our previous conjecture, the hypotheses are unlikely to be completely correct, and the conjecture is meant primarily to serve as a guide for future work. For example, it is almost certain that one needs to assume an analogue of the supercritical phase condition for the conjecture to hold, and more structure on the central charge and representative form $\tilde U \in U$ is likely to be needed in general. Away from an analogue of the supercritical phase, we expect that there is an analogue of the above conjecture which requires stability with respect to subobjects on $\D^b\Coh(X)$ and which links more closely with Bridgeland stability. There is reason to be optimistic, coming from Chen's proof of a uniform version of the above for deformed Hermitian Yang--Mills connections on line bundles \cite{chen2019j,chen2020supercritical}, as well as our results \cref{thm:intromainthm}, concerning the large volume limit of this Conjecture, and \cref{thm:introsurfaces}, concerning this Conjecture for line bundles over complex surfaces.

\vspace{4mm} \noindent {\bf Acknowledgements:} The authors thank the referee for their very helpful comments. The second author thanks Richard Thomas for drawing their attention to the deformed Hermitian Yang--Mills equations, and their PhD supervisor Simon Donaldson for many interesting and useful comments. The authors also thank Jacopo Stoppa for interesting discussions. RD was funded by a Royal Society University Research Fellowship. JBM was funded by the EPSRC and the London School of Geometry and Number Theory. LMS's postdoctoral position is supported by Villum Fonden, grant 0019098. The work was revised after LMS joined the University of Gothenburg, funded by a Marie Sklodowska-Curie Individual Fellowship funded from the European Union's Horizon 2020 research and innovation programme under grant agreement No 101028041. 
\section{Preliminaries}\label{sec:prelims}

\subsection{Stability conditions}\label{sec:stabconds}

Let $E$ be a vector bundle over a smooth projective variety $X$, or more generally a compact K\"ahler manifold. We begin by defining the notion of stability for $E$ relevant to our problem. Our notion is a slightly simpler variant of Bridgeland, Douglas and Bayer's stability conditions \cite{bridgeland2007stability, douglas-icm, bayer}. As in their work, the crucial input is a \emph{central charge}.

\begin{definition} A \emph{central charge} is a group homomorphism $$Z: K(X) \to \C$$ which sends a coherent sheaf to the semi-closed upper-half plane $\HH \subset \C$. Here $\HH = \{ z \in \C : z \in \R_{>0} \cdot e^{i \theta}, \theta \in (0,\pi] \}$ and $K(X)$ denotes the Grothendieck group of coherent sheaves on $X$. \end{definition}

Thus $Z$ associates a complex number to each coherent sheaf on $X$, and is \emph{additive in short exact sequences}. That $Z$ has image in the upper-half plane means that one can meaningfully compare the arguments of $Z(F)$ and $Z(F')$ for any coherent sheaves $F, F'$ on $X$.

\begin{definition} For each coherent sheaf $F$ on $X$ with $Z(F)\ne 0$, we define the \emph{phase} $$\phi(F) = \arg Z(F)$$ to be the argument of the complex number $Z(F)$. We say that a coherent sheaf $E$ with $Z(E) \ne 0$ is \emph{Z-stable} if for all proper coherent subsheaves $F \subset E$, $Z(F) \ne 0$ and $$\phi(F) < \phi(E).$$ \end{definition}

Semistability and instability can be defined analogously: $E$ is $Z$\emph{-semistable} if $\phi(F) \leq \phi(E)$ for all $F \subset E$, and $Z$-\emph{unstable} if it is not $Z$-semistable. The notion of $Z$-polystability, which will be irrelevant to our work, means that $E$ is a direct sum of $Z$-stable coherent sheaves of the same argument. 

\begin{remark}\label{non-trivialZ} We shall always tacitly assume that $Z(E) \neq 0$, and hence $\phi(E)$ is well-defined, in order to make sense of this definition; this a non-trivial topological constraint on $E$ and $Z$ which due to \cref{non-trivialphase} will always be satisfied in the situations of primary interest in the present work.
\end{remark}

We will restrict ourselves to a subclass of central charges, which were introduced by Bayer \cite[Theorem 3.2.2]{bayer}; see also related work of Toda \cite{toda}. These require the input of a \emph{stability vector}, defined as follows. 

\begin{definition}
	A \emph{stability vector} $\rho$ is a sequence $(\rho_0, \rho_1, \dots, \rho_n)\in (\C^*)^{n+1}$ of non-zero complex numbers such that $\Im(\rho_n)>0$ and

	$$\Im\left( \frac{\rho_{n-1}}{\rho_{n}} \right)> 0,$$
	where $n = \dim X$.
\end{definition}
\begin{remark}
Bayer's notion of a stability vector requires that 	
$$\Im\left( \frac{\rho_d}{\rho_{d+1}} \right)> 0$$
	for all $0\le d \le n-1$. In the present work, all results apply assuming this condition only for $d=n-1$, thus we have slightly relaxed the definition of a stability vector.
\end{remark}

The condition that $\Im(\rho_n)>0$ is essentially a normalisation condition, one can otherwise rotate the entire stability vector to ensure this holds. Bayer then defines his class of central charges using a stability vector, together with appropriate topological data. 

\begin{definition}
	Let $\Omega=([\omega], \rho, U)$ be the data given by
	\begin{enumerate}
		\item a K\"ahler class $[\omega]\in H^{1,1}(X,\R)$ (for example, the first Chern class of an ample line bundle),
		\item a stability vector $\rho=(\rho_0,\dots,\rho_n)$,
		\item and a unipotent operator $U\in H^*(X,\R)$ (that is, $U=1+N$ where $N\in H^{>0}(X,\R)$).
	\end{enumerate}
	Define $Z_{\Omega,k}:K(X) \to \C[k]$ by 
	$$Z_{\Omega,k}(E) := \int_X \sum_{d=0}^n \rho_d k^d [\omega]^d \cdot \ch(E) \cdot U.$$
	We call a central charge  $Z_{\Omega,k}$ of this form a \emph{polynomial central charge}. We denote  $$\phi_{k}(E) = \arg Z_{\Omega,k}(E)$$ the associated argument using the principal branch $\arg: \C^* \to (-\pi, \pi]$, which we call the \emph{phase}. We define the \emph{$Z$-slope} by
	$$\mu_{Z_k}(E) = -\frac{\Re Z_k(E)}{\Im Z_k(E)}$$
	when $\Im Z_k(E) > 0$, so that $\mu_{Z_k}(E) = - \cot(\phi_k(E))$. Define $\mu_{Z_k}(E) = +\infty$ if $\Im Z_k(E) = 0$ and $\Re Z_k(E) < 0$, and $-\infty$ when $\Re Z_k(E) > 0$. If $Z_k(E)=0$ define $\mu_{Z_k}(E) = +\infty$ and $\varphi_k(E) = \pi$. 
\end{definition}

The most important consequence of this definition for us is that for any fixed $\Omega$, one obtains a \emph{polynomial family} of central charges, which we view as a sequence depending on $k\in \Z_{\ge 0}$. Slightly abusing notation, we usually fix $\Omega$ and denote by $Z_k$  the induced central charge. Strictly speaking, for each $k$, the function $Z_k: K(X) \to \C$ may not be a central charge, as it may not send coherent sheaves to the semi-closed upper half plane $\HH \subset \C$. By definition of a stability vector, what is certainly true and all that will be important for us is that for each coherent sheaf $F$ on $X$, we have $\Im Z_k(F) > 0$ for $k\gg 0$; this is a consequence of our assumption $\Im(\rho_n)>0$. Rather than $Z_k$-stability itself, the notion most relevant to us is the following.

\begin{definition} We say that a coherent sheaf $E$ with $Z_k(E) \ne 0$ for all $k\gg 0$ is \emph{asymptotically $Z$-stable} if for all proper, non-zero coherent subsheaves $F \subset E$ there exists some $k_0$, possibly depending on $F$, such that we have $$\phi_k(F) < \phi_k(E)$$ for all $k \ge k_0 \gg 0.$\end{definition}

We note the following equivalent formulations of asymptotic $Z$-stability. 

\begin{lemma}\label{lem:equivalentstabilityconditions}
	Let $E\to X$ be a holomorphic vector bundle and $F \subset E$ a proper, non-zero coherent subsheaf with $Z_k(F) \ne 0$. Then the following conditions are equivalent.
	\begin{itemize}
		\item $\varphi_k(F) < \varphi_k(E)$ for all $k\gg 0$.
		\item $\Im \left( \frac{Z_k(F)}{Z_k(E)} \right) < 0$ for all $k\gg 0$.
		\item $\mu_{Z_k}(F) < \mu_{Z_k}(E)$ for all $k\gg 0$. 
	\end{itemize} 
	When $Z_k(F) = 0$ for all $k\gg 0$ then (i) and (iii) are still equivalent.
\end{lemma}
\begin{proof}
	The proof follows quickly from noting that by our assumption $\Im \rho_n > 0$ on the stability vector, the central charge $Z_k(E)$ lies in the upper-half plane for all $k\gg 0$. Then it is a standard fact that if $z=re^{i\theta},z'=r' e^{i\theta'}$ are non-zero complex numbers and $|\theta - \theta'| < \pi$ so that $z,z'$ lie in the same half-plane, then the sign of $\sin (\theta - \theta')$ is the same as the sign of $\theta - \theta'$. Since $\Im (z/z') = r/r' \sin (\theta - \theta')$ one concludes $\theta < \theta'$ if and only if $\Im (z/z') < 0$. Setting $z=Z_k(F), z'=Z_k(E)$ then the above equivalences follow quickly.
\end{proof}

One can again define analogous notions of \emph{asymptotic $Z$-semistability, asymptotic $Z$-instability} and \emph{asymptotic $Z$-polystability}. From now on, we fix a polynomial central charge $Z$.

\begin{remark}\label{rmk:stability-remarks} The notion of stability we are using here is not a standard one appearing in either the Bridgeland stability or physics literature, thus we make the following remarks.
\begin{enumerate}
\item The principal difference between our definition and Bridgeland's is that his stability conditions are defined on the bounded derived category of coherent sheaves, hence involve more complicated subcomplexes of $E$. He also requires several other conditions, such as his support property, that are irrelevant to our work. We thus emphasise that while it is in general extremely challenging to produce Bridgeland stability conditions, it is straightforward to generate  polynomial central charges.

\item Our notion of asymptotic $Z$-stability involves inequalities induced from \emph{polynomials} rather than \emph{numbers}; this is in part motivated by Bayer's work \cite{bayer}, and as he demonstrates  is useful in studing stability in the ``large volume limit'' of $k \gg 0$. Bayer considers \emph{polynomial stability conditions} which are a more involved variant of the above closer to Bridgeland's work \cite[\S 2]{bayer}. His work also involves  a \emph{perversity function} \cite[Definition 3.1.1]{bayer}, which enables him to instead assume a minor variant of our assumption that $\Im(\rho_n)>0$ and which is important in constructing stability conditions on $t$-structures defined by perverse sheaves; these play no role in our work.

\item Polynomial central charges are not the most general for which our results apply. For example, all that will be important to us is that $\Im(\rho_n)$ and $\Im (\rho_{n-1}/\rho_{n})$ are strictly positive, rather than the further conditions on the stability vector. Bayer uses these properties in many ways; one is in his proof that the space of polynomial stability conditions forms a complex manifold \cite[\S 8]{bayer}. We mirror  almost all of Bayer's conditions to link with his work. The only restriction we make that Bayer does not is that we assume our unipotent operator $U$ is \emph{real} rather than \emph{complex}; this seems more natural in our arguments, and includes the cases of primary interest in string theory. As we note in \cref{Ucondition}, this can also be relaxed somewhat.

\end{enumerate}
\end{remark}

As mentioned above, it is straightforward to produce the polynomial central charges relevant to our work. We give two examples.

\begin{example}\label{example:dhympsc}
	Fix a class $\beta \in H^{1,1}(X,\R)$ on the compact K\"ahler manifold $(X,[\omega])$. Then setting$$\rho_d = i^{n+1} \frac{(-i)^d}{d!}, \qquad U=e^{-\beta}$$induces the polynomial central charge
	$$Z_{k}(E) = i^{n+1} \int_X e^{-ik[\omega]} e^{-\beta} \ch(E).$$ In this context $\beta$ is usually called the \emph{$B$-field}. If one instead takes the operator 
	$$U=e^{-\beta} \cdot \sqrt{\Td(X)},$$where $\sqrt{\Td(X)}$ denotes the square root of the Todd class, defined for example through the natural power series, one obtains the polynomial central charge
	$$Z_{\textrm{Todd},k}(E) = i^{n+1}\int_X e^{-ik[\omega]} e^{-\beta} \ch(E) \cdot \sqrt{\Td(X)}.$$ These are the two central charges which are most important in string theory. The factor $i^{n+1}$ is there just to ensure that the condition $\Im \rho_n > 0$ is satisfied.
\end{example}

Recall that a coherent sheaf $F$ on a compact K\"ahler manifold $(X,[\omega])$ has an associated \emph{slope} $$\mu(F) = \frac{[\omega]^{n-1}.c_1(F)}{\rk F},$$ and a coherent sheaf $F$ is called \emph{slope semistable} (respectively \emph{slope stable}) if $$\mu(F) \leq \mu(E) \textrm{ (respectively } \mu(F)<\mu(E) \textrm{)}$$ for all coherent subsheaves $F \subset E$.

\begin{lemma}\label{lemma:slopesemistable}
	If a coherent sheaf $E$ is asymptotically $Z$-stable, then it is slope semistable. 
	
\end{lemma}
\begin{proof}
	
	Asymptotic $Z$-stability for $k\gg 0$ states that
	$$\phi_k(F) = \arg Z_{k}(F) < \phi_k(E)=\arg Z_{k}(E)$$
	for all coherent subsheaves $F\subset E$. By \cref{lem:equivalentstabilityconditions} this is equivalent to asking
	\begin{equation}\label{eq:imsuchandsuch}\Im\left(\frac{Z_k(F)}{Z_k(E)}\right) < 0.\end{equation}
	Consider the expansion
	\begin{equation}\label{eq:Zkexp}
	Z_k(E) = \rho_n k^n [\omega]^n \rk(E) + \rho_{n-1} k^{n-1} [\omega]^{n-1}.(c_1(E) + \rk(E) U_2) + O(k^{n-2})
	\end{equation}
	where $U_2$ denotes the degree two component of the class $U\in H^*(X, \R).$ For simplicity let us write 
	\begin{equation}\label{degU}\deg_U (E) := [\omega]^{n-1}.(c_1(E) + \rk(E) U_2).\end{equation}
	Then dividing we see
	$$\frac{Z_k(F)}{Z_k(E)} = \frac{\rk(F)}{\rk(E)} + k^{-1} \frac{\rho_{n-1}}{\rho_n} \left( \frac{\deg_U F \rk E - \deg_U E \rk F}{[\omega]^n (\rk E)^2}\right) + O(k^{-2}).$$
	Since $\rho$ is a stability vector we have
	$$\Im \left(\frac{\rho_{n-1}}{\rho_n}\right) > 0.$$
	As Equation \eqref{eq:imsuchandsuch} holds for all $k$, the leading order term of its expansion in $k$ is non-negative. Thus by taking imaginary parts, this inequality produces
	$$\deg_U F \rk E - \deg_U E \rk F \leq 0$$
	which is equivalent to asking
	$$\frac{\deg_U F}{\rk F} \leq \frac{\deg_U E}{\rk E}.$$
	Since $\deg_U (E) = \deg E + \rk E [\omega]^{n-1}.U_2$, this can be further simplified to the usual slope inequality
	$$\frac{\deg F}{\rk F} \leq \frac{\deg E}{\rk E}.$$ \end{proof}

\begin{remark}\label{Ucondition}The condition that $U$ (or rather its degree two component) is a \emph{real} operator is used in the proof above; in general, for $U$ complex, one needs to impose a further topological hypothesis on the imaginary part of its degree two component, which seems slightly unnatural. \end{remark}

The following justifies \cref{non-trivialZ}, and follows directly from Equation \eqref{eq:Zkexp}.

\begin{corollary}\label{non-trivialphase} For an arbitrary coherent sheaf $E$ with $\rk E > 0$ and polynomial central charge $Z_k$, $Z_k(E) \neq 0$ for $k \gg0$.
\end{corollary}

Recall that a slope semistable coherent sheaf $E$ admits a \emph{Jordan--H\"older} filtration \cite[Proposition 1.5.2]{huybrechts-lehn}, namely a filtration by coherent subsheaves $$0 = S_0 \subset S_1 \subset \hdots \subset S_l = E,$$ such that each $Q_j = S_j/S_{j-1}$ is slope stable and $$\mu(S_1) =\hdots =  \mu(S_{l-1}) = \mu(E).$$ We denote $$\Gr(E) = \bigoplus_{j=1}^l S_j/S_{j-1},$$ called the \emph{associated graded object}, which is \emph{slope polystable} in the sense that it is a direct sum of stable coherent sheaves of the same slope. The associated graded object of a slope semistable coherent sheaf is unique up to re-ordering, although the Jordan--H\"older filtration itself is not unique in general. 

\begin{corollary}\label{cor:polystabledegen} Let $Z$ be a polynomial central charge. An asymptotically $Z$-stable coherent sheaf admits a Jordan--H\"older filtration. \end{corollary}

The following smoothness property introduced by Leung will be central in our analytic arguments.

\begin{definition}\cite[page 520]{leung1997einstein}  We say that a slope semistable holomorphic vector bundle $E$ is \emph{sufficiently smooth} if the associated graded object $\Gr(E)$ is itself a holomorphic vector bundle. \end{definition}

\begin{remark} We will primarily be interested in sufficiently smooth vector bundles in the present work. As we are considering \emph{asymptotic} $Z$-stability, it will then be most natural to consider stability with respect to \emph{subbundles}, rather than arbitrary coherent subsheaves. For clarity we will often then say that $E$ is \emph{asymptotically} $Z$\emph{-stable with respect to subbundles}. \end{remark}

The following so-called ``see-saw'' property explaining the properties of the arguments in short exact sequences will be used at several points. 

\begin{lemma}\label{lemma:seesaw}For any short exact sequence of coherent sheaves
	\begin{center}
		\begin{tikzcd}
		0\arrow{r} &F \arrow{r} &E \arrow{r} &G \arrow{r} &0
		\end{tikzcd}
	\end{center}
	with $\rk E > 0$, one has equivalences
	$$\varphi_k(F) \le \varphi_k(E) \text{ for all $k\gg 0$}\Longleftrightarrow \varphi_k(E) \le \varphi_k(G)\text{ for all $k\gg 0$}$$
	and
	$$\varphi_k(F) \ge \varphi_k(E) \text{ for all $k\gg 0$} \Longleftrightarrow \varphi_k(E) \ge \varphi_k(G)\text{ for all $k\gg 0$}.$$
	Furthermore equality holds if and only if equality holds above, except if $Z_k(F) = 0$ for all $k\gg 0$ in which case $\varphi_k(F) > \varphi_k(E)$ and $\varphi_k(E) = \varphi_k(G)$, or if $Z_k(G) = 0$ for all $k\gg 0$ in which case $\varphi_k(F) = \varphi_k(E)$ and $\varphi_k(E) < \varphi_k(G)$.
	
	The same equivalences hold when replacing the phase $\varphi_k$ by the $Z$-slope $\mu_{Z_k}$. 
\end{lemma}
\begin{proof} This is essentially the standard see-saw property for Bridgeland stability, and in our situation follows from the facts that the central charge $Z_k$ is additive in short exact sequences, being a group homomorphism from the K-group of $X$ to $\C$, and that the central charge has image in the upper-half plane $\HH \subset \C$ for all $k \gg 0$. 
	
The result is most easily proven using the $Z$-slope $\mu_{Z_k}$ and verifying the inequalities above. The only detail to check is in the case where $Z_k(F) = 0$ for all $k \gg 0$ or $Z_k(G) = 0$ for all $k \gg 0$. Both situations can easily be covered by using the expression 
$$\Im Z_k(F) (\mu_{Z_k}(F) - \mu_{Z_k}(E)) = \Im Z_k(G) (\mu_{Z_k}(E) - \mu_{Z_k}(G))$$
which follows from the additivity of $Z_k$ in short exact sequences and proceeding case by case.
\end{proof}

Using this we can establish simplicity of asymptotically $Z$-stable sheaves.

\begin{lemma}\label{lemma:simple}
	An asymptotically $Z$-stable vector bundle is simple.
\end{lemma}
\begin{proof}
	This only uses the see-saw property, and therefore follows from a similar argument as for slope stable or Gieseker stable sheaves \cite[Corollary 1.2.8]{huybrechts-lehn}.  Explicitly, let $u: E \to E$ be a non-zero endomorphism and suppose $\ker u \ne 0$. Then by asymptotic $Z$-stability, since $\ker u \subset E$ we must have $$\varphi_k(\ker u) < \varphi_k(E)$$ for all $k\gg 0$. However, the see-saw property of \cref{lemma:seesaw} produces the inequality $$\varphi_k(E) < \varphi_k(E/\ker u) = \varphi_k(\im u)$$ for all $k\gg 0$, contradicting asymptotic $Z$-stability with respect to the subsheaf $\im u\subset E$. Thus $\ker u = 0$ and $u$ is injective. Therefore the induced morphism $\det u: \det E \to \det E$ is injective and therefore an isomorphism (being an injective morphism of line bundles of the same degree), so $u$ must be an isomorphism.
	
	It is then a standard fact that $H^0(X, \End E)=\C$. Indeed, let $x\in X$ be any point and $\lambda\in \C^*$ any eigenvalue of the invertible map $u_x: E_x \to E_x$, and denote by  $$\tilde u = u -\lambda \Id_E$$ the induced endomorphism. Then $\tilde u$ cannot be an isomorphism, as it has non-trivial kernel at $x$. Thus $\tilde u = 0$, implying that $u=\lambda \Id_E$ for some $\lambda \in \C^*$.
\end{proof}

\begin{remark} One could go further in developing the purely algebro-geometric theory of asymptotic $Z$-stability, for example by constructing Jordan--H\"older and Harder-Narasihman filtrations. Analogous results are proven by Bayer and Toda in somewhat different situations \cite{bayer, toda}. Such results play no role in our main results, which link instead with complex differential geometry, so we leave this to future work. \end{remark}

\subsubsection{An example} The inequalities used in the definition of asymptotic $Z$-stability can be calculated explicitly in concrete examples, with the interesting  case,  by  \cref{lemma:slopesemistable},  being that of a slope semistable vector bundle $E$ which is not slope polystable. If the bundle $E$ has Jordan--H\"older filtration given by the inclusion of a single subbundle $F$, it will follow from our main result  \cref{thm:intromainthm} that $E$ is asymptotically $Z$-stable if and only if $\phi_k(F)<\phi_k(E)$ for all $k \gg 0$. As we now demonstrate, this means that the resulting criterion can  actually be understood in practice. We describe the conclusions concerning $Z$-critical connections, even though the precise definition of such connections is postponed to \cref{sec:connections}.

\begin{example}
We consider a simple rank three bundle on $\mathbb{P}^2$ considered by Maruyama in the context of Gieseker stability, and we follow its presentation in \cite[p. 90]{okonekschneiderspindler80}. Thus let $F$ be a slope stable vector bundle of rank $2$ on $\mathbb{P}^2$ with $c_1(F) =0$ and $H^1 (\mathbb{P}^2, F ) \neq 0$. Given $\tau \in H^1 (\mathbb{P}^2, F )  \setminus \{ 0 \}$, define an extension 
$$ 
0 \to F \to E \to \mathcal{O}_{\mathbb{P}^2} \to 0
$$
using $\tau.$ The Chern characters of $E$ are then given by $$c_1 (E) = 0 \qquad c_2(E) = c_2(F).$$ Note by the Bogomolov inequality that $c_2 (F) \geq 0$, since $F$ is a slope stable bundle with vanishing first Chern class. The bundle $E$ is then not slope stable, since $\mu(F) = \mu(E)$ and hence destabilises $E$, but $E$ is in fact Gieseker stable. It follows that $E$ is simple and slope semistable, with
$$
0 \subset F \subset E
$$ 
giving a Jordan--H\"older filtration of $E$, meaning $E$ has associated graded object 
$$
\Gr (E) = F \oplus \mathcal{O}_{\mathbb{P}^2}.
$$

We will consider central charges of two forms, as described in \cref{example:dhympsc}. First, consider the dHYM central charge with B-field, given by
$$
Z_{\mathrm{dHYM}} (E) = i^3 \int_{\mathbb{P}^2} e^{- i k [\omega]} \ch (E) e^{-B},
$$
for some B-field $B\in H^2(\mathbb{P}^2, \R).$ Then since $c_1(E)=0$, we have $\ch(E) = \rk(E) - c_2(E)$. Let us denote $$\sigma := \int_{\mathbb{P}^2} c_2(E) = \int_{\mathbb{P}^2} c_2(F) \ge 0.$$ 
Denoting $h=c_1(\mathcal{O}(1)) = [\omega]$, which satisfies $\int_{\mathbb{P}^2} h^2 = 1$, one can compute
$$Z_{\mathrm{dHYM}}(E) = i \sigma + i\frac{\rk(E)}{2} k^2 + k\rk(E) B.h - i \frac{\rk{E}}{2} B^2,$$
and so the imaginary part of $Z(F)\big/Z(E)$ is a positive constant multiple of 
\begin{align*}
&(\sigma + \frac{\rk(E)}{2} (k^2-B^2) ) \cdot (-k\rk(F) B.h) - (\sigma + \frac{\rk(E)}{2} (k^2-B^2) ) \cdot (-k\rk(E) B.h)\\
&= k \sigma B.h (\rk(E)-\rk(F)).
\end{align*}

When $\sigma>0$,  since $\rk(E) - \rk(F) >0$, we have $\varphi_k(F) < \varphi_k(E)$ whenever the B-field is chosen such that $B.h<0$. Thus with an appropriately chosen $B$-field, $E$ admits deformed Hermitian Yang-Mills connections for all $k \gg 0$; note that this requires the $B$-field to be non-trivial.  In the remaining case $B.h\ge 0$ or $\sigma=0$, which includes the case of trivial $B$-field, we have $\varphi_k(F)\ge \varphi_k(E)$, which implies $E$ is not asymptotically $Z_{\mathrm{dHYM}}$-stable, and therefore by \cref{thm:intromainthm} $E$ cannot admit deformed Hermitian Yang--Mills connections for $k \gg 0$. Note that we have only used the fact that $c_2(E)=c_2(F)$ in this calculation. Considering the same question for the dual $E^*$, we have that $c_2(E)=c_2(E^*)$, but the subbundle $F\subset E$ induces a quotient $E^* \to F^*$, and so the opposite conclusion will hold for $E^*$. That is, $E^*$ fails to be asymptotically $Z_{\mathrm{dHYM}}$-stable whenever $\sigma B.h\le 0$ and is stable otherwise.

Next, we consider the other central charge of relevance to string theory
$$
Z_{\mathrm{Todd}} (E)= i^3 \int_{\mathbb{P}^2} e^{- i k [\omega]} \ch (E) \sqrt{\Td (\mathbb{P}^2) } e^{-B};
$$
that is, where the unipotent operator $U=\sqrt{\Td (\mathbb{P}^2) } e^{-B}$ contains both the B-field and the square root of the Todd class. The Todd class of $\mathbb{P}^2$ is given by 
$$
\Td(\mathbb{P}^2) = 1 + \frac{3}{2} h + h^2,
$$
and so 
$$
\sqrt{\Td (\mathbb{P}^2) } = 1 + \frac{3}{4}h + \frac{7}{32}h^2.
$$
By the same computations as above we obtain that the imaginary part of $\frac{Z_{\mathrm{Todd}}(F)}{Z_{\mathrm{Todd}}(E)}$ is a positive constant multiple of
$$k \sigma (\rk E - \rk F) \left(B.h - \frac{3}{4}\right).$$
Thus $\varphi_k(F)<\varphi_k(E)$ holds whenever $B.h < \frac{3}{4}$, producing $Z_{\mathrm{Todd}}$-critical connections, and in the case $B.h\ge \frac{3}{4}$ asymptotic $Z_{\mathrm{Todd}}$-stability is violated and such connections cannot exist. As before, the opposite conclusions will hold when taking $E^*$ instead of $E$.

As mentioned above, as a consequence of \cref{thm:intromainthm} it suffices to check asymptotic $Z$-stability only with respect to $F$. Since $c_2(E) \geq 0$ always holds by the slope stability of $F$, we can summarise our findings as follows:
\begin{enumerate}
\item If $c_2(E)>0$ then $E$ is asymptotically $Z_{\mathrm{dHYM}}$-stable whenever we have chosen a B-field with $B.h<0$, and asymptotically $Z_{\mathrm{dHYM}}$-unstable if $B.h\ge 0$. Therefore $E$ admits a deformed Hermitian Yang--Mills metric with B-field for $B.h<0$, but does not admit deformed Hermitian Yang--Mills connections when $B.h\ge 0$ in the large volume regime. In particular $E$ does not admit a deformed Hermitian Yang--Mills connection with vanishing B-field.
\item The opposite conclusions as above hold when $E$ is replaced by $E^*$. That is, $E^*$ admits deformed Hermitian Yang--Mills connections  whenever the B-field is chosen such that $B.h>0$, and does not admit deformed Hermitian Yang--Mills connections when $B.h\le 0$. In particular neither $E$ or $E^*$ admit dHYM metrics with vanishing B-field.
\item If $c_2(E)>0$ then $E$ is asymptotically $Z_{\mathrm{Todd}}$-stable whenever $B.h<\frac{3}{4}$, and is unstable otherwise. Therefore $E$ admits a $Z_{\mathrm{Todd}}$-critical connetions whenever $B.h<\frac{3}{4}$. In particular this includes the case with vanishing $B$-field. 
\item The opposite conclusion as above holds for $E^*$. That is, $E^*$ admits a $Z_{\mathrm{Todd}}$-critical metric whenever $B.h>\frac{3}{4}$. In particular this does \emph{not} include the case with vanishing B-field, in contrast to $E$.
\item If $c_2(E)=0$ then $E$ is asymptotically $Z$-semistable with respect to either central charge.

\end{enumerate}
\end{example}

\begin{remark}
	The example above demonstrates a simple example of a wall-crossing phenomenon for $Z$-stability and the existence of $Z$-critical connections. In particular as the stability conditions $Z_{\mathrm{dHYM}}$ and $Z_{\mathrm{Todd}}$ are varied by replacing $B$ with $t B$ and letting $t\in \R$ vary, we see a jump from stability to instability across critical thresholds ($t=0$ and $t=\frac{3}{4}$ respectively). 
	
	One can observe similar wall-crossing phenomena if we choose to vary the stability vector for our stability conditions instead. Let us work just with $Z_{\mathrm{dHYM}}$. If we change our stability vector to $\rho = (\rho_0, \rho_1, -1)$ to obtain a new stability condition $Z$, having normalised $\rho_2=-1$ for convenience, then regardless of the B-field chosen we obtain $Z$-stablity whenever $\Im \rho_0 < 0$, and also when $\Im \rho_0 = 0$ and $B.h<0$. Instability follows whenever $\Im \rho_0 > 0$ or $\Im \rho_0 = 0$ and $B.h\ge 0$. Similar conclusions will follow for $Z_{\mathrm{Todd}}$, and one can also replace $E$ by $E^*$ and flip the inequality $B.h<0$ to $B.h>0$. 
	
	In particular, it follows from \cref{thm:intromainthm} that we find non-trivial examples of $Z$-critical connections in higher rank even for central charges that are different from those that have previously appeared in the deformed Hermitian Yang--Mills literature. 
\end{remark}

\subsection{$Z$-critical connections}\label{sec:connections} We now turn to complex differential geometry, by introducing a geometric partial differential equation, involving various curvature quantities related to connections on holomorphic vector bundles, to each polynomial central charge.

Thus let $(X,\omega)$ be a compact K\"ahler manifold, and let $E\to X$ be a holomorphic vector bundle. We fix a Hermitian metric $h$ on $E$ with Chern connection $A$. Denoting by $c(E)$ the total Chern class of $E$, we denote by $\tilde{c}_i(E)$ the differential forms given by the Chern--Weil representatives of $c_i(E)$ with respect to $A$. 

Fix now a polynomial central charge $Z_{\Omega,k}$ induced by $\Omega=([\omega], \rho, U)$, a K\"ahler metric $\omega \in [\omega]$ and let $$\tilde U\in \Omega^*(X,\R)$$ be a differential form representing a unipotent class $U \in H^*(X,\R)$ in cohomology. We denote $\tilde{Z}_{\Omega,k}(A)$ the endomorphism-valued $(n,n)$-form on $X$ defined by 
	$$\tilde{Z}_{\Omega,k}(A) = \sum_{d=0}^n \rho_d k^d \omega^d \wedge \tilde\ch(E) \wedge \tilde U;$$ this sum contains forms of degree other than $(n,n)$ in general, but we select only the $(n,n)$-component. Recall that the trace of $\tilde \ch(E) = \exp(\frac{i}{2\pi} F_A)$ is the Chern--Weil representative of $\ch(E)$. When $k=1$, we omit the subscript $k$ and hence adopt the notation $Z_{\Omega}, \tilde{Z}_{\Omega}$ for the resulting central charge and differential form. As in  \cref{non-trivialZ}, we will always assume $Z(E) \neq 0$.

\begin{definition} We say that $A$ is a \emph{$Z$-critical connection} if $$\Im \left( e^{-i\varphi(E)} \tilde{Z}_{\Omega}(A) \right) = 0.$$
	Here $\varphi(E) = \arg Z_{\Omega}(E)$ is the argument of $ Z_{\Omega}(E) \in \C$. With varying $k$ we say that $A$ is a \emph{$Z_k$-critical connection}. 
\end{definition}

We note that the above condition depends on the particular choice of representative $\omega \in [\omega]$ and $\tilde U \in U$, which we now take to be included in our choice $\Omega$ of stability data.

\begin{remark} The real and imaginary parts here are taken with respect to the Hermitian metric $h$, or rather the induced metric on $\End E$. Via the Hermitian metric, one can define the adjoint of any section of $\End E$, giving a decomposition of any such section into a self adjoint and an anti-self adjoint part. The \emph{real} part is then by definition the self adjoint component, while the \emph{imaginary} part is ($-i$ times) the anti-self adjoint component. These are more commonly called the Hermitian and skew-Hermitian components, but to preserve the notation typically used in the study of deformed Hermitian Yang--Mills connections on holomorphic line bundles, we use the terminology of real and imaginary components. The Chern curvature $F_h$ is an imaginary endomorphism with respect to $h$, meaning $\frac{i}{2\pi} F_h$ is a real endomorphism.

\end{remark}

We next check that $\phi_k = \arg Z_{\Omega,k}(E) $ is the appropriate topological constant to make sense of our equations.

\begin{lemma}\label{lemma:right-constant} The constants $\phi_k$ satisfy $$\int_X \tr \Im \left( e^{-i\varphi_k(E)} \tilde{Z}_{\Omega,k}(A) \right) = 0.$$\end{lemma}

\begin{proof} Notice first of all from Chern--Weil theory that $$Z_{\Omega,k}(E) = \int_X \tr \tilde{Z}_{\Omega,k}(A).$$ Then since $\phi_k(E) = \arg Z_{\Omega,k}(E) ,$ we must have $$ e^{-i\varphi_k(E)}  = \frac{r_k(E)}{Z_{\Omega,k}(E) },$$ where $r_k(E) \in \R$. Thus $$\int_X \tr \Im \left( e^{-i\varphi_k(E)} \tilde{Z}_{\Omega,k}(A) \right)  = \Im\left(  \frac{r_k(E)}{Z_{\Omega,k}(E) }Z_{\Omega,k}(E)\right ) = 0,$$ as required. \end{proof}

\begin{example}\label{example:dhym-eqn}
	Consider the polynomial central charge of \cref{example:dhympsc} with trivial $B$-field, given by $$Z_{k}(E) = i^{n+1} \int_X e^{-ik[\omega]} \ch(E).$$
	Since $e^{-i\omega} \tilde{\ch}(E) = e^{-i \omega \otimes \Id_E} e^{\frac{i}{2\pi}F_A} =e^{-i (\omega \otimes \Id_E- \frac{F_A}{2\pi})}$, we get that the $(n,n)$-part of this is

	\begin{align*}
	 \frac{(-i)^n}{n!} (\omega \otimes \Id_E- \frac{F_A}{2\pi})^n.
	\end{align*}
	The phase is given by
	\begin{align*}
	e^{-i\varphi(E)} &= \frac{r(E)}{Z(E)}\\
	&= r(E) \left( \frac{i}{n!} \int_X \tr \left( \omega \otimes \Id_E - \frac{F_A}{2\pi}\right)^n\right)^{-1}.
	\end{align*}
	If we define $$\hat z = i \int_X \tr \left( \omega \otimes \Id_E - \frac{F_A}{2\pi}\right)^n$$
	and $\hat z = \hat r e^{i\hat \phi}$, then since $r/\hat r = \frac{1}{n!}$ it follows that
	$$\Im (e^{-i\varphi} \tilde{Z}_{\Omega,k}(A)) = \frac{1}{n!} \Im \left(e^{-i\phi} \left( \omega \otimes \Id_E - \frac{F_A}{2\pi}\right)^n\right).$$ This is the \emph{deformed Hermitian Yang--Mills equation} for vector bundles, first proposed in higher rank by Collins--Yau \cite[\S 8.1]{collins2018moment}. When $E$ is a line bundle we recover the regular deformed Hermitian Yang--Mills equation which is well-studied in the literature. One can make a similar computation with non-trivial $B$-field.
\end{example}

\begin{example}
	Considering the polynomial central charge of primary relevance to string theory
	$$ Z_{\textrm{Todd},k}(E) = i^{n+1}\int_X e^{-ik\omega} e^{-\beta} \ch(E) \cdot \sqrt{\Td(X)},$$ one obtains a differential equation
	$$\Im\left(e^{-i\varphi_k(E)}e^{-ik\omega} e^{-\tilde \beta} \widetilde{\ch(E)} \cdot \widetilde{\sqrt{\Td(X)}}\right) = 0,$$ with $\tilde \beta \in \beta$ the $B$-field. This is the equation introduced by the physicists Enger-L\"utken \cite{enger-lutken}. It is most natural to use the Chern-Weil representatives of $\sqrt{\Td(X)}$ induced by viewing the K\"ahler metric $\omega$ as a Hermitian metric on the holomorphic tangent bundle of $X$, though other choices can be made.
\end{example}

\begin{example}\label{exa:linebundle} Suppose the vector bundle under consideration $L$ is a line bundle, meaning $\tilde Z_{\Omega}(A)$ is an $(n,n)$-form. Then writing $$\tilde Z_{\Omega}(A)= \Re \tilde{Z}_{\Omega}(A) + \Im  \tilde{Z}_{\Omega}(A),$$ the $Z$-critical equation asks 
$$\Im \left( e^{-i\varphi(L)} \tilde Z_{\Omega}(A)\right) =\Im  \left( e^{-i\varphi(L)} \exp\left(i\arctan\left(\frac{\Im  \tilde{Z}_{\Omega}(A) }{\Re \tilde{Z}_{\Omega}(A)}\right)\right) \right) =
 0.$$ One sees that this is equivalent to the condition $$\arctan\left(\frac{\Im  \tilde{Z}_{\Omega}(A) }{\Re \tilde{Z}_{\Omega}(A)}\right) = \varphi(L) \textrm{ modulo } 2\pi.$$ This is analogous to how the deformed Hermitian Yang--Mills equation is often presented on line bundles, using the special form of $ \Im  \tilde{Z}_{\Omega}(A)$ and $\Re \tilde{Z}_{\Omega}(A )$ in that case \cite[\S 2]{jacob2017special}. It is interesting to ask to what extent such a formulation is possible in higher rank, particularly as this perspective has been crucial in the understanding of the deformed Hermitian Yang--Mills equation. We will discuss this representation and the case of line bundles in more detail in \cref{sec:momentmaps}.
\end{example}

Of particular interest in the present work is the \emph{large volume limit}, when we consider $k \gg 0$ to be large (making the volume of $X$ with respect to $k\omega$ large). We show that the leading order condition for $Z$-critical connection is the \emph{weak Hermite--Einstein} condition, where we recall that a connection $A$ is a weak Hermite--Einstein metric if $$\Lambda_{\omega} F_A = if \Id_E,$$ with $f \in C^{\infty}(X,\R)$ a function which integrates to the appropriate topological constant.

\begin{lemma}\label{lemma:largevolume}
	In the large volume regime $k \gg 0$, the leading order condition for $A$ to be a $Z$-critical connection for a polynomial stability condition is given by the weak Hermite--Einstein equation. More precisely, there is an expansion in $k$ of the form \begin{align*}\Im \big( &e^{-i\varphi_k(E)} \tilde{Z}_{\Omega,k}(A) \big)=  \\ & ck^{2n-1}\left([\omega]^n \rk (E) \omega^{n-1} \wedge \left( \frac{i}{2\pi} F_A + \tilde U_2 \Id_E\right) - \deg_U( E)  \omega^n\otimes \Id_E\right) + O(k^{2n-2}),\end{align*} where $c\in \R_{>0}$ is a positive constant depending on $\rho$ and the central charge,  $\tilde U_2$ is the degree two part of the differential form $\tilde U$, and $\deg_U(E)$ is the degree defined as Equation \eqref{degU}. 
\end{lemma}

To leading order, this is equivalent to  $$\contr_{\omega} F_A +2\pi i \left( \frac{\deg_U(E)}{(n-1)! [\omega]^n \rk(E)} - \contr_{\omega} \tilde U_2 \right) \Id_E,$$ which is the  weak Hermite--Einstein condition with function $$f = -2\pi \left( \frac{\deg_U(E)}{(n-1)! [\omega]^n\rk(E)} - \contr_{\omega} \tilde U_2 \right).$$

\begin{proof}
	This follows from computing the leading order terms directly. Writing $Z_k(E) = r_k (\cos \varphi_k + i \sin \varphi_k)$, one computes
	\begin{align*}
	r_k \cos \varphi_k &= \Re \rho_n k^n [\omega]^n \rk E + \Re \rho_{n-1} k^{n-1} [\omega]^{n-1}.(c_1(E) + \rk (E) U_2) + O(k^{n-2}),\\
	r_k \sin \varphi_k &=  \Im \rho_n k^n [\omega]^n \rk E + \Im \rho_{n-1} k^{n-1} [\omega]^{n-1}.(c_1(E) + \rk (E) U_2) + O(k^{n-2}).
	\end{align*}
	We also have
	\begin{align*}
	\Re \tilde Z_k(A) &= \Re \rho_n k^n \omega^n \otimes \Id_E + \Re \rho_{n-1} k^{n-1} \omega^{n-1} \wedge \left( \frac{i}{2\pi} F_A + \tilde U_2 \Id_E\right)   + O(k^{n-2}),\\
	\Im \tilde Z_k(A) &= \Im \rho_n k^n \omega^n \otimes \Id_E + \Im \rho_{n-1} k^{n-1} \omega^{n-1} \wedge \left( \frac{i}{2\pi} F_A + \tilde U_2 \Id_E\right)  + O(k^{n-2}).
	\end{align*}
	The condition for $A$ to be $Z_k$-critical is thus
	$$r_k \cos \phi_k \Im \tilde Z_k(A) - r_k \sin \phi_k \Re \tilde Z_k(A) = 0$$
	Computing the induced expansion in $k$ gives
	\begin{align*}
	&\Im \big( r_k e^{-i\varphi_k(E)} \tilde{Z}_{\Omega,k}(A) \big) \\
	&= k^{2n-1}\bigg((\Re \rho_n [\omega]^n \rk E)\left(\Im \rho_{n-1} \omega^{n-1} \wedge \left( \frac{i}{2\pi} F_A + \tilde U_2 \Id_E\right)\right)\\
	&\quad+ (\Re \rho_{n-1} \deg_U(E)) (\Im \rho_n \omega^n \otimes \Id_E)\\
	&\quad- (\Im \rho_n [\omega]^n \rk E ) \Re \rho_{n-1} k^{n-1} \omega^{n-1} \wedge \left( \frac{i}{2\pi} F_A + \tilde U_2 \Id_E\right)\\
	&\quad-  (\Im \rho_{n-1} \deg_U(E)) (\Re \rho_n \omega^n \otimes \Id_E)\bigg) + O(k^{2n-2}).
	\end{align*}
	We simplify by dividing by the factor 
	\begin{align*} c =& r_k( \Re \rho_n \Im \rho_{n-1} - \Im \rho_n \Re \rho_{n-1}) \\
	=& r_k \Im (\rho_{n-1} \overline \rho_n),
	\end{align*}
	which is non-zero, and in fact positive, by the stability vector assumption that $\Im (\rho_{n-1}/\rho_n) >0.$ Thus the $k^{2n-1}$-coefficient becomes
	$$[\omega]^n \rk (E) \omega^{n-1} \wedge \left( \frac{i}{2\pi} F_A + \tilde U_2 \Id_E\right) - \deg_U (E)  \omega^n\otimes \Id_E,$$
	as desired. \end{proof}

\begin{remark} We use in the above that $U \in H^{>0}(X,\R)$   (or rather its degree two component) is a real operator. One could relax this assumption, and perform the same computation, which would then require a rather unnatural topological assumption to ensure the constant $c$ produced is non-zero. As mentioned in \cref{rmk:stability-remarks} $(iv)$, in all cases of  interest to string theory, $U$ is real. \end{remark}

While we are primarily interested in the large volume limit in the present work, the condition for a connection to be $Z$-critical is of interest for all $k$. Nevertheless, even for the deformed Hermitian Yang--Mills equation on a line bundle, one must impose a positivity condition on the connections considered. The following is the natural positivity condition in the study of $Z$-critical connections, and was suggested in the case of the deformed Hermitian Yang--Mills connections in higher rank by Collins--Yau \cite[\S 8.1]{collins2018moment}.

\begin{definition} We define the space $\H_Z$ of \emph{$Z$-calibrated connections} to be 
$$\H_{Z} := \left\{ A\mid \Re \left( \tr \left( e^{-i\varphi(E)}  \tilde{Z}_{\Omega}(A)\right)\right) > 0\right\},$$ with positivity meant in the sense of volume forms. \end{definition}

The reason that this does not enter our work is the following.

\begin{lemma} Fix a polynomial central charge $Z$, and a holomorphic vector bundle $E$ with connection $A$. Then $A$ is $Z_k$-calibrated for all $k \gg 0$. \end{lemma}

\begin{proof} 

This follows by a similar argument as \cref{lemma:largevolume}, which produces $$\Re \left( \tr \left( e^{-i\varphi(E)}  \tilde{Z}_{\Omega}(A)\right)\right) = k^{2n}\left( (\rk E)^2 [\omega]^n \vert \rho_n\vert^2 \right)\omega^n + O(k^{2n-2}),$$ implying positivity for $k \gg 0$ since $\rho_n \neq 0$ by definition of the stability vector. \end{proof}

We now turn to the linearisation of the $Z$-critical connection equation. Denote by $ \G^{\C} = \Gamma(X,\GL(E,\C))$ the complex gauge group. The gauge transformations we will be interested in will be the Hermitian endomorphisms, and hence will take the form $$g\in \G^{\C}\cap \End_h(E).$$ Recall that any such $g$ acts on the space of connections by
\begin{align}\label{gaugegroupaction}g\cdot d_A &= g^* \circ \del_A \circ g^{*-1} + g^{-1} \circ \db_A \circ g, \\ &= g \circ \del_A \circ g^{-1} + g^{-1} \circ \db_A \circ g,\end{align}
where we have used that $g^* = g$, since $g$ is Hermitian.

 It will be useful to denote the equation of interest $$D_k(A) = \Im \left( e^{-i\varphi_k(E)} \tilde{Z}_{\Omega,k}(A) \right).$$  We will similarly write $P_{k}$ for the linearisation of $D_{k}$ at a connection $A$, where $A$ is the associated Chern connection to $h$.

\begin{lemma}\label{lemma:linearisation} For $k \gg 0$, the $Z$-critical connection equation is elliptic. Moreover, the linearisation admits an expansion $$P_{k} = ck^{2n-1}\rk (E) [\omega]^n \Lap_{A^{\End E}} + O(k^{2n-2}),$$ where $c\neq 0$ is the real constant produced by \cref{lemma:largevolume} and $ \Lap_{A^{\End E}} $ is the Laplacian on $\End E$ defined via the connection $A^{\End E}$ induced by $A$.

\end{lemma}

\begin{proof}

We first compute the asymptotics of the linearised operator $P_k$. Considering a path of gauge transformations $$g_t = g\exp(tf),$$ where $f$ is Hermitian,  a standard calculation produces 
$$\deriv{}{t}_{t=0} F_{g_t\cdot A} = (\del_{A^{\End E}} \db_{A^{\End E}} - \db_{A^{\End E}} \del_{A^{\End E}}) f.$$ Thus the linearisation of the operator $$f \to n\omega^{n-1} \wedge F_{g_t.A}$$ is given by
$$f \to i\contr_{\omega} (\del_{A^{\End E}} \db_{A^{\End E}} - \db_{A^{\End E}} \del_{A^{\End E}}) f \omega^n = (\Lap_{A^{\End E}} f) \omega^n.$$

By \cref{lemma:largevolume}, the order $k^{2n-1}$ component of the operator is given by the weak Hermite--Einstein operator. Thus one term of the order $k^{2n-1}$ component of the linearisation takes the form  $c\rk (E) [\omega]^n \Lap_{A^{\End E}}$. However, the term $$c\left([\omega]^n \rk (E) \omega^{n-1} \wedge \left( \tilde U_2 \Id_E\right) - \deg_U( E)  \omega^n\otimes \Id_E\right)$$ is independent of $g \in \G^{\C}\cap \End_h(E)$, hence does not affect the linearisation. 

Ellipticity is an open condition in the space of linear operators, hence ellipticity of $P_{k}$ follows from ellipticity of the Laplacian $\Lap_{A^{\End E}} $. \end{proof}
The following will also be important in our later considerations.

\begin{lemma}\label{lemma:dhymtracezero}
	Fix a  gauge transformation $g\in  \G^{\C}\cap \End_h(E)$. The trace of each term of the expansion of $D_{k}(g.A)$ in $k$ integrates to zero over $X$.
\end{lemma}

\begin{proof}
Chern--Weil theory implies that the differential form $ \tilde{Z}_{\Omega,k}(g.A)$ induced by $g.A$ has trace-integral independent of gauge transformation $g$, and hence $g$ plays no role in the statement. Thus to obtain the result, one recalls from \cref{lemma:right-constant} that $$\int_X \tr \Im \left( e^{-i\varphi_k(E)} \tilde{Z}_{\Omega,k}(A) \right) = 0,$$ and expands this equality in powers of $k$.\end{proof}

\subsection{Moment maps, subsolutions and existence on surfaces}\label{sec:momentmaps}
While our primary motivation for the notion of a $Z$-critical connection is Bridgeland stability, the equation also arises through a natural infinite-dimensional moment map, as we now describe. This is motivated by the Collins--Yau  moment map picture for the deformed Hermitian Yang--Mills equation on a line bundle \cite[\S 2]{collins2018moment}. The manifold on which we construct the moment map is the space of integrable unitary connections on a Hermitian vector bundle. To construct a moment map, one then first needs to construct a K\"ahler metric on this space. As in the work of Collins--Yau, the natural pairing is degenerate in general, and to obtain that this is non-degenerate leads us to the notion of a \emph{subsolution}. 

We now discuss subsolutions, which will allow us to discuss not only moment maps, but also ellipticity conditions away from the large volume limit, as well as a complete understanding of the existence of $Z$-critical connections on line bundles over complex surfaces. The subsolution condition is related to a formal linearisation of the equation;  that this sort of linearisation is related to moment map problems on holomorphic vector bundles has also been suggested by Thomas in another framework (who we thank for pointing this out to us) \cite[p. 737]{clay}.

\subsubsection{Subsolutions\label{sec:subsolutions}}

We return to a Hermitian holomorphic vector bundle $(E,h)$ with Chern connection $A$. We will work with the convention of fixing $h$ and varying $A$. Consider the endomorphism-valued form $\tilde Z_{\Omega}(A)$, for $\Omega$ the data of a polynomial stability condition. Then one obtains a $Z$-critical operator
$$D_Z: A\mapsto \Im(e^{-i\phi(E)}\tilde Z_{\Omega}(A)).$$
Varying the Chern connection $$A\mapsto \exp(tV) \cdot A$$ for a Hermitian endomorphism $V$ of $(E,h)$, the curvature transforms
$$F_{\exp(tV)\cdot A} = F_A + (\delbar_A \del_A - \del_A \delbar_A) V t + O(t^2).$$
Each term in $D_Z$ is some real constant multiple of a form
$$\omega^d \wedge \left( \frac{i}{2\pi} F_A\right)^j \wedge \tilde U_\ell$$
such that $d+j+\ell = n$. Perturbing $A$ to $A_t = \exp(tV)\cdot A$, this term linearises to
$$\omega^d \wedge j \left[  \underbrace{\left(\frac{i}{2\pi} F_A\right)\wedge \cdots \wedge \left( \frac{i}{2\pi} F_A\right)}_{j-1 \text{ times}} \wedge \frac{i}{2\pi} (\delbar \del - \del \delbar) V\right]_{\mathrm{sym}} \wedge \tilde U_\ell.$$
Here $\mathrm{sym}$ denotes the graded symmetric product of the endomorphism-valued two-forms $F_A$ and $(\delbar \del - \del \delbar)V$. This is defined abstractly by the formula
$$[B_1 \wedge \cdots \wedge B_j]_{\mathrm{sym}} := \frac{1}{j!} \sum_{\sigma \in S_j} (-1)^{\mathrm{grsgn}\, \sigma} B_{\sigma(1)} \wedge \cdots \wedge B_{\sigma(j)}$$
where $B_i$ are endomorphism-valued forms. The graded sign $(-1)^{\mathrm{grsgn} \, \sigma}$ of a permutation is the sign obtained by permuting the wedge product $B_1 \wedge \cdots \wedge B_j$ to $B_{\sigma(1)} \wedge \cdots \wedge B_{\sigma(j)}$ as differential forms (ignoring $\End E$ factors). For example the graded sign of $\sigma$ when permuting $\alpha \wedge \beta$ to $\beta \wedge \alpha$ is $+1$, where $\alpha$ is a two-form and $\beta$ a 1-form.

From the above discussion, we see the linearisation of $D_Z$ at a Chern connection $A$ takes the form
$$P_Z: V\mapsto \left[ \Im(e^{-i\varphi(E)} Z'(A)) \wedge \frac{i}{2\pi} (\delbar \del - \del \delbar)V  \right]_{\mathrm{sym}}$$
where $\tilde Z'(A)$ is the formal derivative of $\tilde Z(A))$ with respect to $\frac{i}{2\pi} F_A$, using the convention $d(\frac{i}{2\pi} F_A) = \Id_E$. Thus $\tilde Z'(A)$ is an $\End E$-valued $(n-1,n-1)$-form where each term consists of a wedge product of $n-1$ endomorphism-valued two-forms. We take the graded symmetrisation with respect to this wedge product.

\begin{remark}
	In the case where $E$ is a line bundle so that the curvature terms commute, the linearisation is simply the formal derivative of $D(h)$ with respect to $F(h)$, wedged with $(\delbar \del_h- \del_h\delbar) V$, and no symmetrisation is needed.
\end{remark}

Using the above expressions for the linearisation $P_Z$ of $D_Z$, we can define a subsolution. 

\begin{definition}\label{def:subsolution}
	We say that a Chern connection $A$ satisfies the \emph{subsolution condition} for the $Z$-critical equation if for all $p\in X$ and $u\in T_{0,1}^* X_p \otimes \End E_p$ non-zero, the trace
	$$\tr i \left[ \Im(e^{-i\phi(E)} \tilde Z'(A)) \wedge u^* \wedge u \right]_{\text{sym}} > 0.$$
	That is, the $\End E$-valued $(n-1,n-1)$-form $\Im (e^{-i\varphi(E)} \tilde Z'(A))$ is positive. 
\end{definition}

The importance of the subsolution condition is apparent in the following.

\begin{lemma}\label{cor:ellipticity} Suppose $A$ is a subsolution of the $Z$-critical equation. Then the $Z$-critical operator is elliptic at $A$. \end{lemma}

\begin{proof}
	A non-linear differential operator is elliptic, by definition, when its linearisation is a linear elliptic differential operator. Our linearisation is a differential operator $P_Z: \Gamma(\End E) \to \Gamma(\End E)$ of the form
	$$V\mapsto \frac{1}{\omega^n} \left[ \Im(e^{-i\varphi(E)} Z'(A)) \wedge \frac{i}{2\pi} (\delbar \del - \del \delbar)V  \right]_{\mathrm{sym}}.$$
	We may ignore any dependence of $\del$ and $\db$ on the Chern connection $A$, because the dependence on the connection form $A$ appears only at first order and below in $\dbard$ (meaning these terms are irrelevant in demonstrating ellipticity). 
	
	To compute the symbol $\sigma$ of this operator, notice that we if we have a test differential form $\eta = a_j dx^j + b_k dy^k$ then we have
	$$\sigma\left(\pderiv{}{z^j}\right) = \frac{1}{2}(a_j - i b_j),\quad \sigma\left( \pderiv{}{\bar z^k}\right) = \frac{1}{2}(a_k + i b_k)$$
	so if we define $\xi := \xi_j d\bar z^j = (a_j + i b_j) d\bar z^j \in T_{p; 0,1}^* X$ then in the operator
	$$V\mapsto \dbard V = \sum_{j,k} \pderiv{}{\bar z^k} \pderiv{}{z^j}  (V)  d\bar z^k \wedge dz^j$$
	we compute the symbol by replacing
	$$\pderiv{}{\bar z^k} \mapsto \xi_k, \quad \pderiv{}{z^j} \mapsto \bar \xi_j,$$
	and similarly for $\del \delbar$. Thus the principal symbol $\sigma_{\xi} (P_Z(A)): \End E \to \End E$ becomes
	$$\sigma_{\xi} (P) (V) = \frac{1}{\omega^n} \left[\Im (e^{-i\varphi(E)} \tilde Z'(A))\wedge \left(\frac{i}{2\pi} 2 \xi \wedge \bar \xi \otimes V\right)\right]_{\mathrm{sym}}.$$
	Ellipticity holds when $\sigma_\xi(P)$ is invertible for any $\xi\ne 0$. Since $A$ is a subsolution, we have
	$$\tr i \left[ \Im (e^{-i\varphi(E)} \tilde Z'(A) \wedge u^* \wedge u) \right]_{\mathrm{sym}} > 0.$$
	Choosing $u=\xi \otimes V$ and using the tracial property to cyclically permute one of the $V$'s appearing in the above terms to the end of the expression, we obtain
	$$\tr (\sigma_\xi (P)(V) V) = \frac{1}{2} \tr i \left[ \Im (e^{-i\varphi(E)} \tilde Z'(A) \wedge u^* \wedge u) \right]_{\mathrm{sym}} > 0.$$
	It follows that $\sigma_\xi (P)(V) \ne 0$ for any $V\ne 0$, so $\sigma_\xi (P)$ has no kernel and $D_Z$ is elliptic.
\end{proof}

\begin{remark} Our original proof of this statement contained an error pointed out to us by B. Ackermann and R. Takahashi. The final step of the corrected proof given here, namely the step of equation $\tr (\sigma_{\xi}(P)(V) V)$, is modelled on Takahashi's work (which came after the original version of our paper), where he proves that an analogue of the subsolution condition for the higher rank $J$-equation implies ellipticity \cite[Lemma 3.5]{takahashi2021j}. We thank both B. Ackermann and R. Takahashi for their advice on this point.
	
\end{remark}

Note that the equation is not elliptic in general, as can be seen, for example, by scaling the unipotent operator $U$. Near the large volume limit, however, the following lemma shows that $A$ is automatically a subsolution with respect to $Z_k$ for $k \gg 0$, and so the equation is elliptic for such $k$. This gives another proof of the ellipticity claim of \cref{lemma:linearisation}, although the two proofs use essentially the same facts.

\begin{lemma}\label{lemma:largevolumeelliptic} Any $A$ is a subsolution of the $Z_k$-critical equation for $k \gg 0$.
\end{lemma}

\begin{proof} Analogously to \cref{lemma:linearisation}, one checks that $$\Im(e^{-i\phi(E_k)} \tilde Z'(A)) = k^{2n-1}c[\omega]^n\omega^{n-1}\otimes \Id_E + O(k^{2n-3})$$ for $c \neq 0$, which implies the result. \end{proof}

When the rank is one, ellipticity precisely captures the subsolution condition.
\begin{lemma}
	If $\rk(E) = 1$, then the $Z$-critical equation is elliptic at $A$ if and only if $A$ is a subsolution.
\end{lemma}
\begin{proof}
	In this case, the trace is just the identity operator. The ellipticity condition is therefore that for all $\xi$,
	$$
	\Im (e^{-i\varphi(E)} \tilde Z'(A) \wedge \xi \wedge \bar \xi \otimes f))
	$$
	is nowhere-vanishing if $f$ is, while the subsolution condition is that
	$$
	\Im(e^{-i\phi} \tilde Z'(A) \wedge i u^* \wedge u)) > 0
	$$
	when $u\ne 0$. Without loss of generality, we can assume that $f$ is a positive function, and then we can take $u = \sqrt{f} \xi$ to see that the subsolution condition implies the ellipticity condition, as before. For the converse, we can take $f = 1$ and $\xi = u$, since $\xi$ and $u$ are both non-zero sections of $\Omega^{0,1}$ in the rank one case.
\end{proof}

We now make some comments on the origin of the subsolution condition.

\begin{remark}
	In order to explain the origin of the subsolution condition \cref{def:subsolution}, we note this is the first example for vector bundle equations of a type of condition which has already appeared in the literature for the study of generalised complex Monge--Amp\`ere equations such as the $J$-equation and the deformed Hermitian Yang--Mills equation \cite{song2008convergence,collins20151}. For the $J$-equation
	$$c\alpha^n = \alpha^{n-1}\wedge \omega,$$
	on a K\"ahler manifold $(X,\omega)$ with  $\omega$ a fixed K\"ahler metric, $\alpha$ the unknown K\"ahler metric and $c$ the appropriate topological constant, the natural subsolution condition is given by
	$$cn \alpha^{n-1} - (n-1)\alpha^{n-2} \wedge \omega > 0$$
	in the sense of $(n-1,n-1)$-forms \cite{song2008convergence}. The positivity of $(n-1,n-1)$-forms simply asserts that 
	$$i \bar \xi \wedge  \xi \wedge (cn \alpha^{n-1} - (n-1)\alpha^{n-2} \wedge \omega) > 0$$
	as a top degree form, for all nowhere-vanishing $\xi \in \Omega^{0,1}(X)$. Song--Weinkove show that the existence of a subsolution for the $J$-equation is equivalent to the solvability of the equation. Collins--Jacob--Yau similarly construct a subsolution condition for the deformed Hermitian Yang--Mills equation by requiring positivity of an $(n-1,n-1)$-form \cite[Prop. 8.1]{collins20151}; they again show that existence of a solution is equivalent to the existence of a subsolution  \cite[Theorem 1.2]{collins20151}.  Our subsolution condition can thus be seen as the natural generalisation of these conditions to higher rank and more general equations.
\end{remark}

\begin{remark}\label{rmk:subsolns} To further explain the terminology \emph{subsolution}, recall from \cref{exa:linebundle} that when $E=L$ is a line bundle and $\alpha$, the curvature of the metric, is the unknown, the equation can be rewritten \begin{equation}\label{eq:arctan}\arctan\left(\frac{\Im  \tilde{Z}_{\Omega}(A) }{\Re \tilde{Z}_{\Omega}(A)}\right) = \varphi(L) \textrm{ modulo } 2\pi.\end{equation} Choosing coordinates pointwise such that $$\omega = \sum_{j=1}^n idz_j\wedge d\bar z_j, \qquad \alpha = \sum_{j=1}^n i\lambda_jdz_j\wedge d\bar z_j,$$ with $\lambda_j \in \R$ the collection of eigenvalues, the equation is thus of the form $$\arctan(G(\lambda_1,\hdots \lambda_n)) =  \varphi(L) \textrm{ modulo } 2\pi,$$ for some function $G$. This is then an equation which is closely analogous to Sz\'ekelyhidi's important general notion of a subsolution for non-linear equations of eigenvalues of the unknown form $\alpha$ \cite{szekelyhidi2018fully}; see also Guan \cite{guan2014second}. We emphasise, however, that the function $\arctan(G(\lambda_1,\hdots \lambda_n)) $ does not satisfy Sz\'ekelyhidi's hypotheses: this is true even for the deformed Hermitian Yang--Mills equation, which fails his concavity hypothesis. The situation is even more drastic for arbitrary polynomial central charges, which may also fail to be symmetric in the eigenvalues $\lambda_1, \hdots, \lambda_n$.  We hope, however, that many of Sz\'ekelyhidi's techniques can be applied to the $Z$-critical equation, as is true for the deformed Hermitian Yang--Mills equation \cite{collins20151}. 
\end{remark}

Clearly, by \cref{lemma:largevolumeelliptic}, it cannot be true in higher rank that the existence of a subsolution for the $Z$-critical equation is equivalent to the existence of a solution, as there are algebraic obstructions to the existence of a $Z$-critical connection in the large volume limit which are not necessarily satisfied for every vector bundle, as we shall see in \cref{sec:leunganalogue}. We expect that, in certain phase ranges, the subsolution condition is  however a consequence of the existence of a solution; even in Sz\'ekelyhidi's situation, it seems there is no general recipe to show that the existence of a solution implies the existence of a subsolution \cite{szekelyhidi2018fully}.

From experience with the deformed Hermitian Yang--Mills equation \cite[\S 8]{collins20151}, one expects that non-trivial subsolution conditions lead to obstructions to existence related to \emph{subvarieties} (and perhaps more generally, saturated subsheaves supported on subvarieties). We will see something analogous in results to come when we study the $Z$-critical equation on holomorphic line bundles over complex surfaces. That the subsolution conditions are automatic when $k \gg 0$ gives some conceptual explanation as to why only subbundles arise as obstructions to the existence of $Z$-critical connections in our main result.

\begin{remark} When $E=L$ is a holomorphic line bundle, Equation \eqref{eq:arctan} demonstrates that one can equivalently reformulate the $Z$-critical operator as a partial differential operator on the space of Hermitian metrics taking the form$$h \to \arctan\left(\frac{\Im  \tilde{Z}_{\Omega}(A_h) }{\Re \tilde{Z}_{\Omega}(A_h)}\right) - \varphi(L) - 2\pi k,$$
where $A_h$ is the Chern connection of $h$. In the setting of the deformed Hermitian Yang--Mills equation, using the special form of $\tilde{Z}_{\Omega}(A_h)$ in that situation, Jacob--Yau show that this partial differential equation is \emph{always} elliptic, without any subsolution hypothesis \cite[\S 2]{jacob2017special}. Thus this reformulation has several practical advantages over the original formulation.
\end{remark}

\subsubsection{Moment maps} We now produce the moment map. Consider $\scA(h)$, the space of $h$-unitary connections on $E\to X$. In this section we will allow the complex structure the connection induces to be non-integrable -- this will be needed in \cref{sec:choosecplxstr}. The group of unitary gauge transformations $\G$ acts on $\scA(h)$ in the standard way: if $g=\exp (\psi)$ is a gauge transformation for some $\psi \in \End_{SH}(E)$, where $\End_{SH}(E)$ denotes endomorphisms that are skew-Hermitian with respect to $h$,  the action is given by
$$g\cdot (A^{1,0} + A^{0,1}) = (A^{1,0} + (\bar\partial_A \psi)^*) + (A^{0,1} - \bar\partial_A \psi),$$
or
$$g\cdot A^{1,0} = A^{1,0} + \del_A \psi^*,\quad g\cdot A^{0,1} = A^{0,1} - \delbar_A \psi.$$
We can identify the tangent space $T_A \scA $ with $\Omega^{0,1}(\End E)$ -- note that the tangent space to the subspace of integrable complex structures is given by those one-forms $a\in T_A \scA$ such that $(\db_E + a)^2 = 0$. The associated $h$-unitary connection is defined by $\nabla_A + a - a^*$, where the adjoint is taken with respect to $h$. Here we could equivalently work with the space of skew-Hermitian endomorphisms $\hat a = \hat a^{1,0} + \hat a^{0,1} \in \Omega^1(\End E)$ satisfying $(\hat a^{0,1})^* = - \hat a^{1,0}$ and on occasion we change convention when convenient.

Define a Hermitian pairing on $T_A\scA(h) = \Omega^{0,1}(\End E)$ depending on $Z$ by 
$$\langle a,b \rangle_A = -i \int_X \tr \left[ \Im (e^{-i\varphi(E)} \tilde Z'(A)) \wedge a \wedge b^*\right]_{\text{sym}}.$$
We note that this Hermitian pairing on $\scA(h)$ is non-degenerate if and only if $A$ is a subsolution of $Z$ in the sense of \cref{def:subsolution}. 

With respect to the complex structure $a\mapsto ia$ for $a\in \Omega^{0,1}(\End E)$ we can take the $(1,1)$-form $\Omega$ by 
$$\Omega_A(a,b) = \Im\langle a,b\rangle_A.$$
To express $\Omega$ explicitly we change convention to skew-Hermitian endomorphisms $\hat a, \hat b$ where we see
\begin{equation}\label{eqn:symplecticform}
	\Omega_A(\hat a, \hat b) = \tr \int_X \left[\Im(e^{-i\varphi(E)} \tilde Z'(A)) \wedge \hat a \wedge \hat b\right]_{\text{sym}}.
\end{equation}

Precisely because $A$ is a subsolution of the $Z$-critical equation, we see that the above pairing $\langle-,-\rangle_A$ is a Hermitian \emph{inner product}, and therefore this $(1,1)$-form is non-degenerate. 

\begin{proposition}
	The $Z$-dependent form $\Omega$ is closed and $\G$-invariant, and is consequently a symplectic form on the locus of $\scA(h)$ consisting of subsolutions.
\end{proposition}

In general, without the subsolution condition, the form is still closed and $\G$-invariant. We note that the subsolution condition is open in $\scA(h)$.

\begin{proof}
	The $\G$-invariance of the symplectic form is obvious due to the presence of the trace. To verify the closedness, we employ the convention of skew-Hermitian endomorphisms $\hat a\in \Omega^1(\End E)$, where we suppress the circumflexes in the following. Take constant vector fields $a,b,c\in \Gamma(T\scA(h))$ on $\scA(h)$ which do not depend on the connection $A$. It is enough to show that
	$$(d\Omega)_A(a,b,c) = 0.$$
	Since the vector fields are constant with respect to $A\in \scA(h)$, the Lie brackets all vanish; that is, $$[a,b]=[a,c]=[b,c]=0.$$ Thus we must verify 
	$$(d\Omega)_A(a,b,c) = d(\Omega_A(b,c))(a) - d(\Omega_A(a,c))(b) + d(\Omega_A(a,b))(c) = 0.$$
	An arbitrary term of $d(\Omega(b,c))(a)$ is of the form
	$$\int_X \tr \left[ \omega^{n-j-k} \wedge \tilde U_k \wedge F_A^{j-2} \wedge d_A(a) \wedge b \wedge c \right]_{\text{sym}}.$$
	Using the second Bianchi identity and the closedness of the forms $\omega$ and $\tilde U_k$, we see that by summing over the cyclic permutations of $\{a,b,c\}$ we obtain
	$$d\Omega(a,b,c) = \int_X \tr d_A \left[ \Im(e^{-i\varphi(E)} \tilde Z^{(2)}(A)) \wedge a \wedge b \wedge c\right]_{\text{sym}}$$
	where $\tilde Z^{(2)} (A)$ denotes the second formal derivative of $\tilde Z(A)$ with respect to $\frac{i}{2\pi} F_A$. Now since $\tr d_A = d \tr$, Stokes' theorem implies closedness of $\Omega$.

\end{proof}

\begin{remark} This result is new even in the case of deformed Hermitian Yang--Mills connections when the rank is at least two. The idea is based in part on work of Leung \cite[Corollary 1]{leung1998symplectic}, who proved something analogous for symplectic forms associated to his almost Hermite--Einstein connections.
\end{remark}

As is customary, we formally identify the dual of the space $\Lie \G \cong \End_{SH}(E)$ of skew-Hermitian endomorphisms with $\Omega^{2n}(\End_{SH}(E))$ through the non-degenerate pairing $$(\psi, \phi) \mapsto -\int_X \tr \psi \phi.$$ 
We recall \cref{intro:moment}, which we are now ready to prove.

\begin{theorem}\label{moment-map-theorem} The map $$D_Z: A \to2\pi i  \Im \left( e^{-i\varphi(E)} \tilde{Z}(A) \right)$$ is a moment map for the $\G$-action on $(\scA(h), \Omega).$\end{theorem}

\begin{proof}
	
	Having establishing the formalism above, the proof is a reasonably direct analogue of Leung's work on the moment map interpretation for Gieseker stability \cite[\S 3]{leung1998symplectic}, or the Collins--Yau moment map interpretation of the deformed Hermitian Yang--Mills equation on line bundles \cite[\S 2]{collins2018moment}. 
	Here we work in the convention that $a,b$ are skew-Hermitian endomorphisms. Let $\psi\in \Lie \G$ be a skew-Hermitian endomorphism generating a vector field $d_A \psi$ on $\scA(h)$. For any $b\in \Omega^1(\End E)$ a skew-Hermitian one-form, we wish to verify the moment map identity
	$$d(D_Z,\psi)(A)(b) = - \Omega_Z(A)(d_A \psi, b).$$
	On the right-hand side we compute
	\begin{align*}
		\Omega_Z(d_A \psi, b) &= \int_X \tr \left[ \Im(e^{-i\varphi(E)} \tilde Z'(A))\wedge d_A \psi \wedge b\right]_{\mathrm{sym}}\\
		&=-\int_X \tr \left[ \Im(e^{-i\varphi(E)} \tilde Z'(A)) \psi \wedge d_A b\right]_{\mathrm{sym}}
	\end{align*}
	using integration by parts. Now recall that if $F_A$ is a curvature form then in the direction of a tangent vector $b$, the curvature transforms by
	$$F_{A+tb} = F_A + t d_A(b) + t^2 b\wedge b.$$
	Therefore, when computing the left-hand side of the moment map equation using our non-degenerate pairing, we have
	\begin{align*}
		d(D_Z,\psi)(A)(b) &= -2\pi i \left.\frac{d}{dt}\right|_{t=0} \int_X \tr \Im (e^{-i\varphi(E)} \tilde Z(A+tb)) \psi\\
		&= -2\pi i \int_X \tr \left[ \Im(e^{-i\varphi(E)} \tilde Z'(A)) \wedge \frac{i}{2\pi} d_A b\right]_{\mathrm{sym}} \psi\\
		&= \int_X \tr \left[ \Im(e^{-i\varphi(E)} \tilde Z'(A)) \wedge d_A b \psi \right]_{\mathrm{sym}}.
	\end{align*}
	Thus we have 
	$$d(D_Z,\psi)(A)(b) = - \Omega_Z(A)(d_A \psi, b)$$
	as desired. The $\G$-equivariance of $D_Z$ as a moment map is straightforward and similar to the equivariance of $\Omega_Z$ itself, following from the presence of the trace. 
\end{proof}

\begin{remark} The moment map condition is still formally satisfied without the subsolution condition, which is only used to ensure $\Omega$ is non-degenerate. \end{remark}

\subsubsection{Existence on surfaces}

The third and final way in which we wish to demonstrate the naturality of the subsolution condition is through the following existence result for $Z$-critical connections on line bundles over complex surfaces.  

Let $L$ be a holomorphic line bundle over a complex surface $X$, and denote by $\alpha$ the curvature of a Hermitian metric on $L$; this then matches the usual notation in the study of the deformed Hermitian Yang--Mills equation. Without loss of generality, we may assume $\rho_0=2$ and hence the Chern--Weil representative of the central charge takes the form $$\tilde Z_{\Omega}(A) = \rho_2\omega^2 + \rho_1\omega\wedge(\alpha + \tilde U_{1,1})+  2\tilde U_{2,2}  + 2\alpha\wedge\tilde U_{1,1} + \alpha^2$$ where $\alpha = \frac{i}{2\pi} F(h)$ is the  curvature of the Chern connection $A$ of a Hermitian metric on $L$. The subsolution condition asks 
$$\Im( e^{-i\varphi} \tilde Z_\Omega'(A)) >0$$
as a $2$-form on the complex surface. Let us recall from \cref{sec:subsolutions} that $\tilde Z_\Omega'(A)$ is defined as follows. Let $A$ be the Chern connection with respect to $h$, and deform $h \mapsto e^{tf} h$. Then the representative $\alpha\in c_1(L)$ defined by $\alpha = \frac{i}{2\pi} F_A$ changes by
$$\alpha \mapsto \alpha - \frac{i}{2\pi} t \ddbar f.$$
With respect to this deformation we differentiate the $Z$-critical equation at $t=0$, to obtain the linearisation 
$$f\mapsto \Im(e^{-i\varphi} \tilde Z') \wedge 2 i \ddbar f.$$
In this notation one may check
\begin{equation}\label{eqn:explicitsubsoln}\tilde Z_{\Omega}'(A) = \rho_1\omega + 2\tilde U_{1,1}+2\alpha.\end{equation} Here positivity due to the subsolution condition means that there is a representative $\alpha \in c_1(L)$ for which the resulting $(1,1)$-form $\Im( e^{-i\phi} \tilde Z_{\Omega}'(A))$ is K\"ahler. It will also be useful to denote the $Z$-critical operator\begin{equation}\label{eqn:surfaces}\Im(e^{-i\phi}\tilde Z_{\Omega}(A)) = a\alpha^2 + \alpha \wedge \beta + \gamma,\end{equation}
where $a\in \C$ is some constant, possibly vanishing.

In order to phrase the stability hypothesis below, for $C \subset X$ a curve set \begin{equation}\label{eqn:curve}Z_{\Omega,C}(L) = (\rho_1 \omega + 2\alpha + 2\tilde U_{1,1})[C]\end{equation} and define $\phi_{\Omega, C} = \arg(Z_{\Omega,C}(L)).$ For clarity we also then denote $\phi_{\Omega,X} = \arg(Z_{\Omega}(L))$.  We will then say that $\tilde Z_{\Omega}$ satisfies the \emph{volume form hypothesis} if $a\ne 0$ and the real $(2,2)$-form $$\frac{1}{4}\beta^2 - a\gamma>0$$ is a volume form.

We now characterise the existence of $Z$-critical connections on holomorphic line bundles over complex surfaces.

\begin{theorem}\label{thm:introsurfaces} Suppose $L \to X$ is a holomorphic line bundle over a surface, and suppose $\tilde Z_{\Omega}$ satisfies the volume form hypothesis. Then the following are equivalent:
	\begin{enumerate}
		\item $L$ admits a $Z$-critical connection;
		\item $L$ admits a subsolution;
		\item for all curves $C \subset X$ we have $\arg Z_{\Omega, C}(L) > \arg Z_{\Omega,X}(L)$.
	\end{enumerate}
\end{theorem}

\begin{proof}\label{proof:introsurfaces}
	
	We use the elementary fact that one can write a polynomial of degree two $$p(x) = ax^2 + bx +c = \frac{1}{a}\left(\frac{1}{2} p'(x)\right)^2 +c -\frac{1}{4a}b^2.$$ Note that Equation \eqref{eqn:surfaces} can be rewritten $$\left(a\alpha + \frac{1}{2}\beta\right)^2 +a\gamma - \frac{1}{4}\beta^2=0,$$ which therefore implies $$2 a\alpha + \beta = \Im( e^{-i\phi} \tilde Z_{\Omega}'(A)).$$ Thus in this notation the subsolution condition asks that there is an $\alpha \in c_1(L)$ such that $$2 a\alpha + \beta>0,$$ in the sense that this closed $(1,1)$-form is K\"ahler. Equivalently, the subsolution condition asks that $\Im( e^{-i\phi} Z_{\Omega}'(A))$ is a K\"ahler class.
	
	By Yau's solution of the Calabi conjecture \cite{yau}, since by assumption $\frac{1}{4}\beta^2 - a\gamma>0$ is a volume form, the equivalent equation $$\left(a\alpha + \frac{1}{2}\beta\right)^2 = \frac{1}{4}\beta^2-a\gamma$$ admits a solution if and only if the class $\pm[a\alpha + (1/2)\beta]$ is a K\"ahler class (the case $-[a \alpha + (1/2) \beta]$ being K\"ahler corresponds to the alternative subsolution condition where $\Im(e^{-i\varphi} \tilde Z') < 0$; this case can be similarly dealt with above by applying the Demailly--P\u{a}un theorem to the negation of this class), which by the previous paragraph is our subsolution condition. Thus $(i)$ is equivalent to $(ii)$. 
	
	What remains to be shown is that the inequality $\phi_{\Omega, C} < \phi_{\Omega,X}$ for all curves $C \subset X$ is equivalent to $\Im( e^{-i\phi} Z_{\Omega}'(L))$ being a K\"ahler class. To show this, first of all note that the inequality $$\arg(Z_{\Omega,C}(L)) = \phi_{\Omega, C} < \phi_{\Omega,X} = \arg(Z_{\Omega,X}(L))$$ is equivalent to $$\Im\left(\frac{Z_{\Omega,C}}{Z_{\Omega,X}(L)}\right) = \Im(e^{-i\phi}Z_{\Omega,C})>0.$$ By Demailly--P\u{a}un's generalisation of the Nakai-Moishezon criterion to K\"ahler classes \cite{DP}, for $\Im( e^{-i\phi} Z_{\Omega}'(L))$ to be a K\"ahler class it is equivalent to ask for positivity of the integral $$\int_C \Im( e^{-i\phi} Z_{\Omega}'(L))|_C >0$$ for all curves $C \subset X$. Equations \eqref{eqn:explicitsubsoln} and \eqref{eqn:curve} imply $$\tilde Z'_{\Omega}|_C \in Z_{\Omega,C},$$hence showing  that $(ii)$ is equivalent to $(iii)$.\end{proof}

This result is due to Jacob--Yau when the polynomial central charge $Z$ induces the deformed Hermitian-Yang Mills equation \cite[Theorem 1.2]{jacob2017special}. As in their case, the main point is to rephrase the equation as a complex Monge-Amp\`ere equation, with the deep analysis then coming from the Calabi--Yau and Demailly--P\u{a}un Theorems \cite{yau, DP}. This technique seems to go back to Chen's use of the Calabi--Yau Theorem \cite[\S 4.2]{chen-mabuchi} and Lejmi-Sz\'ekelyhidi's use of the Demailly--P\u{a}un Theorem as a ``stability condition'' \cite[Proposition 14]{lejmi-szekelyhidi}; both of these related to the J-equation. Note that, much as in \cref{rmk:subsolns}, new analytic questions arise in the study of $Z$-critical connections which were not present in the case of deformed Hermitian Yang--Mills connections; in this case, this is the volume form hypothesis $(1/4)\beta^2 - a\gamma>0$. This is trivial in the deformed Hermitian Yang--Mills case, as there one checks $$\frac{1}{4}\beta^2 - a\gamma = \frac{1}{4} \omega^2,$$ which is always a volume form.  It would be interesting to understand positivity of this form geometrically in general.

\subsubsection{Self-adjointness} We end with a self-adjointness property of the linearisation, unrelated to subsolutions, which will be important in the analysis to come. Define the operator $$
\scL: V \to \frac{\left[\Im(e^{-i\varphi(E)} \tilde Z'(A)) \wedge \frac{i}{2\pi} (\delbar \del - \del \delbar) V\right]_{\mathrm{sym}}}{ \omega^n}
$$ 
which as before can be viewed as the linearisation of the $Z$-critical equation. 
\begin{lemma} Let $U, V$ be Hermitian endomorphisms. Then $$\int_X \langle \scL (U), V\rangle \omega^n = \int_X \langle U, \scL(V)\rangle \omega^n.$$ \end{lemma}

\begin{proof} Consider the operator $$U \to \frac{\tilde Z_{\Omega}(A)(\exp(tU)\cdot A)}{\omega^n}.$$ Then the linearisation of this operator takes the form
	$$
	V\to \frac{\left[\tilde Z'(A) \wedge \frac{i}{2\pi}(\delbar \del - \del \delbar) V\right]_{\mathrm{sym}} }{\omega^n}.
	$$ 
	Self-adjointness will then follow from 
	\begin{equation}\label{equation:adjointness}
		\int_X \left(U  ,\left[\tilde Z'(A) \wedge \frac{i}{2\pi}(\delbar \del - \del \delbar) V\right]_{\mathrm{sym}}\right) = \int_X \left(V,\left[\tilde Z'(A) \wedge \frac{i}{2\pi}(\delbar \del - \del \delbar) U\right]_{\mathrm{sym}} \right),
	\end{equation} 
	where the inner-product is taken on the $\End E$-part. 
	
	Since $\left[\tilde Z'(A) \wedge \frac{i}{2\pi}(\delbar \del - \del \delbar) V\right]_{\mathrm{sym}}$ is a sum of terms of the form 
	$$\omega^{n-j-k} \wedge \tilde U^k \wedge \sum_{\ell = 1}^j \left(F(h)^{\ell -1} \wedge \beta \wedge F(h)^{j-\ell }\right),
	$$ 
	which wedges $\beta=\frac{i}{2\pi} (\delbar \del - \del \delbar)V$ with forms that all are closed and $A$ is the Chern connection, we have that

	\begin{align*} 0 =&\int_X \bar \partial \left(U  , \left[\tilde Z'(A) \wedge \frac{i}{2\pi}\del_A  V\right]_{\mathrm{sym}}\right) \\
		=& \int_X  \left(\delbar_A U  , \left[\tilde Z'(A) \wedge \frac{i}{2\pi}\del_A  V\right]_{\mathrm{sym}}\right) + \int_X  \left( U  , \left[\tilde Z'(A) \wedge \frac{i}{2\pi}\delbar_A \del_A  V\right]_{\mathrm{sym}}\right) .
	\end{align*} 
	Similarly  
	$$\int_X  \left(\partial_A U  ,  \left[\tilde Z'(A) \wedge \frac{i}{2\pi}\delbar_A  V\right]_{\mathrm{sym}}\right) = - \int_X  \left( U  ,  \left[\tilde Z'(A) \wedge \frac{i}{2\pi}\del_A \delbar_A  V\right]_{\mathrm{sym}}\right).$$ 
	Repeating this procedure again and using that
	$$
	\left( U  ,  \left[\tilde Z'(A) \wedge \frac{i}{2\pi}(\delbar_A \del_A -\del_A \delbar_A )V\right]_{\mathrm{sym}}\right) = \left(  (\delbar_A \del_A -\del_A \delbar_A )V ,  \left[\tilde Z'(A) \wedge \frac{i}{2\pi}U \right]_{\mathrm{sym}}\right)
	$$
	since $(A,B) = \tr(AB)$ on Hermitian endomorphisms, produces the desired Equation \eqref{equation:adjointness}. \end{proof}

\section{Asymptotic $Z$-stability of $Z$-critical vector bundles}\label{sec:leunganalogue}

This section proves one direction of our main theorem: the existence of $Z$-critical connections on a sufficiently smooth holomorphic vector bundle $E$ over a compact K\"ahler manifold $(X,k\omega)$ for all $k \gg 0$ implies asymptotic $Z$-stability. We follow a general strategy for gauge-theoretic equations, and use the basic principle that curvature decreases in subbundles and increases in quotients. 

It will be convenient to rephrase the $Z$-critical equation in terms of the parameter $\oldepsilon = 1/k$ in the following.  We then recall that our definition of asymptotic $Z$-stability asks that for all holomorphic subbundles $S \subset E$, for $0 < \oldepsilon \ll 1$ we have strict inequality $$\phi_{\oldepsilon}(S) < \phi_{\oldepsilon}(E).$$ Our restriction to \emph{subbundles} rather than \emph{subsheaves} is natural given our hypothesis that $E$ is sufficiently smooth, and will arise naturally in the analysis to follow.

In our new notation, using $\oldepsilon$ rather than $k$, our central charge takes a slightly different form. Namely, define 
$$\ch^{\oldepsilon}(E) = \sum_{j=0}^n \oldepsilon^j \ch_j(E),\quad U^{\oldepsilon} = \sum_{j=0}^n \oldepsilon^j U_j$$
where we recall $U_j$ refers to the degree $2j$ component of $U$. The central charge associated to the input $(\omega, \rho, U)$ may then be written in terms of $\oldepsilon$ as
$$Z_{\oldepsilon}(E) = \int_X \sum_{d=0}^n \rho_{n-d} \omega^{n-d} \cdot \ch^{\oldepsilon}(E) \cdot U^{\oldepsilon},$$
and noting that we have $$\frac{1}{\oldepsilon^n} Z_{\oldepsilon}(E) = Z_k(E).$$ The $Z$-critical equation is simply produced as before, and we see a solution to the $Z_k$-critical equation is equivalent to a solution to the $Z_{\oldepsilon}$-critical equation at $\oldepsilon=1/k$.

\begin{remark}
	In \cref{sec:perturbation} we will make the different substitution $\varepsilon^2 = 1/k$, which will avoid the introduction of fractional powers in that argument, but is not necessary here. To emphasise the differences in parameters, we use the notation $\varepsilon^2 = \oldepsilon$.
\end{remark}

It will be more convenient to use Hermitian metrics rather than connections; the two formalisms are equivalent, so this makes no substantial difference.  In addition to assuming that the vector bundle $E$ admits $Z_{\oldepsilon}$-critical metrics $h_{\oldepsilon}$ for all $0 < \oldepsilon \ll 1$, we will assume that these metrics $h_{\oldepsilon}$ are uniformly bounded as tensors in the $C^2$-norm. This is again justified by our assumption that $E$ is sufficiently smooth: in \cref{sec:perturbation}, where we prove the reverse direction that asymptotic $Z$-stability implies the existence of $Z_{\oldepsilon}$-critical connections for all $0 < \oldepsilon \ll 1$, the metric $h_{\oldepsilon}$ will be constructed as perturbations of the Hermite--Einstein metric on the bundle $\Gr(E)$, and hence in this situation boundedness in $C^{\infty}$ even holds. The reason we employ the $C^2$-norm is the following.

\begin{lemma}\label{lemma:bounds} Suppose $h_{\oldepsilon}$ is a family of Hermitian metrics on a holomorphic vector bundle $E$ with uniformly bounded $C^2$-norm, and suppose $S \subset E$ is a holomorphic subbundle. Then the second fundamental forms $\gamma_{\oldepsilon}$ and the curvature forms $F_{\oldepsilon}$ are uniformly bounded in $C^1$ and $C^0$ respectively.
\end{lemma}

\begin{proof} This is an immediate consequence of the definitions. Recall that locally, the Chern connection of $h_{\oldepsilon}$ is given as $$A_{\oldepsilon}  = \partial h_{\oldepsilon}\cdot h_{\oldepsilon}^{-1},$$ while the curvature is given as $$F_{\oldepsilon} = d A_{\oldepsilon} + A_{\oldepsilon}\wedge A_{\oldepsilon}.$$ Thus the Chern connections $A_{\oldepsilon}$ are uniformly bounded in $C^1$ due to the uniform $C^2$-bound on $h_{\epsilon}$, and  the curvature forms  $F_{\oldepsilon}$ are uniformly bounded in $C^0$. 

The holomorphic subbundle $S\subset E$ induces a short exact sequence

$$0 \to S \to E \to Q \to 0.$$The Hermitian metrics $h_{\oldepsilon}$ define an $\oldepsilon$-dependent orthogonal splitting of $E$ into a direct sum of complex vector bundles $$E \cong S \oplus Q,$$ where $Q = E/S$. Via this orthogonal splitting one writes the connections $A_{\oldepsilon}$ as $$A_{\oldepsilon}=\begin{pmatrix}
	A_{S, \oldepsilon} & \gamma_{\oldepsilon}\\
	-\gamma_{\oldepsilon}^* & A_{Q,{\oldepsilon}}
	\end{pmatrix}.$$ Here $$\gamma_{\oldepsilon}\in \Omega^{0,1}(X,\Hom(Q,S))$$ is, by definition, the second fundamental form of the subbundle $F$ of $E$ \cite[page 78]{griffiths-harris}. Thus the uniform $C^2$-bound on $h_{\oldepsilon}$ induces a uniform $C^1$-bound on the second fundamental forms $\gamma_{\oldepsilon}$. \end{proof}

We are now ready to prove the main result of the Section, which uses some techniques of Leung in the study of almost Hermite--Einstein metrics   and Gieseker stability \cite[Proposition 3.1]{leung1997einstein}.

\begin{theorem}\label{thm:existencestabilitysubbundles}
	Suppose $E$ is irreducible and sufficiently smooth, and for every $0< \oldepsilon\ll 1$, $E$ admits a solution $h_{\oldepsilon}$ to the $Z$-critical equation such that the metrics $h_{\oldepsilon}$ are uniformly bounded in $C^2$. Then $E$ is asymptotically $Z$-stable with respect to holomorphic subbundles.
\end{theorem}
\begin{proof} We follow the notation introduced  in  \cref{lemma:bounds}. The  $\oldepsilon$-dependent orthogonal splittings $E \cong S \oplus Q$ defined through the Hermitian metrics $h_{\oldepsilon}$ induce block matrix decompositions
	$$A=\begin{pmatrix}
	A_{S,\oldepsilon} & \gamma_{\oldepsilon}\\
	-\gamma_{\oldepsilon}^* & A_{Q,\oldepsilon}
	\end{pmatrix},\qquad F_{\oldepsilon} = \begin{pmatrix}
	F_{S, \oldepsilon} - \gamma_{\oldepsilon} \wedge \gamma_{\oldepsilon}^* & d_{A_{\oldepsilon}} \gamma_{\oldepsilon}\\
	-d_{A_{\oldepsilon}} \gamma_{\oldepsilon}^* & F_{Q,\oldepsilon} - \gamma_{\oldepsilon}^* \wedge \gamma_{\oldepsilon}
	\end{pmatrix},$$ where as above $\gamma_{\oldepsilon}$ is the second fundamental form of $F \subset E$ and $F_{\oldepsilon}$ is the curvature of $h_{\oldepsilon}$. We fix a reference Hermitian metric with which to measure the norms of the various tensors of interest; by our assumption of uniform boundedness, any of the $h_{\oldepsilon}$ will suffice.  Since we have assumed that  $E$ is irreducible, the second fundamental form is non-trivial, thus $$\|\gamma_{\oldepsilon}\| >0.$$ By the uniform boundedness obtained in \cref{lemma:bounds} we have \begin{equation}\label{eqn:uniformoldepsilonbounds}\lim_{\oldepsilon\to 0} \oldepsilon \|F_{\oldepsilon}\|_{C^0} = 0, \quad \lim_{\oldepsilon \to 0} \oldepsilon \|\gamma_{\oldepsilon}\|_{C^1} = 0.\end{equation}

	In order to show $E$ is asymptotically $Z$-stable with respect to $S$, we must show that
	$$\varphi_{\oldepsilon}(S) < \varphi_{\oldepsilon}(E)$$
	for all $0<\oldepsilon \ll 1$, where $\varphi_{\oldepsilon}(E) = \arg Z_{\oldepsilon}(E).$ This is equivalent to the inequality
	$$\Im\left( \frac{Z_{\oldepsilon} (S)}{Z_{\oldepsilon}(E)} \right) < 0,$$
	which in turn is equivalent to the inequality
	$$\Im\left( e^{-i\varphi_\oldepsilon(E)} Z_{\oldepsilon}(S) \right) < 0.$$
	Now since $h_{\oldepsilon}$ solves the $Z_{\oldepsilon}$-critical equation, we have
	$$\Im(e^{-i\varphi_{\oldepsilon}(E)} \tilde{Z}_{\oldepsilon}(h_{\oldepsilon})) = 0$$
	and so
	\begin{equation}\label{eqn:zeqnsubbundle}
	\Im(e^{-i\varphi_{\oldepsilon}(E)} \tr_S( \tilde Z_{\oldepsilon} (h_{\oldepsilon}))) = 0.
	\end{equation} Here $\tilde Z_{\oldepsilon} (h_{\oldepsilon})$ denotes $\tilde Z_{\oldepsilon}$ of the Chern connection of $h_{\oldepsilon}$. This is an $\End(E)$-valued $(n,n)$-form, which restricts to an $\End(S)$-valued $(n,n)$-form via the splitting $E \cong S \oplus Q$, and $ \tr_S( \tilde Z_{\oldepsilon} (h_{\oldepsilon}))$ is the induced $(n,n)$-form on $X$ obtained by taking trace.	

	We will argue that for sufficiently small $\oldepsilon$, there is a positive $\oldepsilon$-dependent constant $C_{\oldepsilon}$ such that
	\begin{equation}\label{eqn:wanted}\int_X \Im(e^{-i\varphi_{\oldepsilon}(E)} \tr_S( \tilde Z_{\oldepsilon} (h_{\oldepsilon}))) = \Im(e^{-i\varphi_{\oldepsilon}(E)} Z_{\oldepsilon}(S)) + C_{\oldepsilon} \|\gamma_{\oldepsilon}\|^2,\end{equation}
	which will imply the result since the left hand side vanishes by \eqref{eqn:zeqnsubbundle}. Note that the leading order term in the $Z_{\oldepsilon}$-critical equation occurs at order $\oldepsilon$, so the constant $C_{\oldepsilon}$ will have lowest order $\oldepsilon$. 
	We begin by considering the order $\oldepsilon = \oldepsilon^1$ term. By \cref{lemma:largevolume}, to leading order the $Z_{\oldepsilon}$-critical equation is given by the weak Hermite--Einstein equation. That is, there is a positive constant $c>0$ such that the leading order term of the $\oldepsilon$-expansion of the $Z_{\oldepsilon}$-critical equation is given by	
	$$c\left([\omega]^n \rk (E) \omega^{n-1} \wedge \left( \frac{i}{2\pi} F_{A_{\oldepsilon}} + \tilde U_2 \Id_E\right) - \deg_U( E)  \omega^n\otimes \Id_E\right).$$
	Thus we see the leading order $\oldepsilon^1$-term of $\Im(e^{-i\varphi_{\oldepsilon}(E)} \tr_S( \tilde Z_{\oldepsilon} (h_{\oldepsilon})))$  is given by
	\begin{align*}
	&c\left([\omega]^n \rk (E) \cdot \left( [\omega]^{n-1}. (\ch_1(S) + \rk(S) U_2)  - [\omega]^{n-1} .\frac{i}{2\pi} \tr(\gamma_{\oldepsilon} \wedge \gamma_{\oldepsilon}^*)\right) \right. \\&\quad \left. \phantom{\frac{i}{2\pi}} - \deg_U(E) \rk(S)  [\omega]^n\right) \\
	&= \Im(e^{-i\varphi_{\oldepsilon}(E)} Z_{\oldepsilon}(S))^1 -  c\frac{i}{2\pi}  (\rk E) [\omega]^n \cdot [\omega]^{n-1} . [\tr(\gamma_{\oldepsilon} \wedge \gamma_{\oldepsilon}^*)]\\
	&= \Im(e^{-i\varphi_{\oldepsilon}(E)} Z_{\oldepsilon}(S))^1 + C_1 \|\gamma_{\oldepsilon}\|^2.
	\end{align*} Here we have used the fact that, for some positive constant $C_1$
	$$c\omega^{n-1} \wedge \tr(\gamma_{\oldepsilon} \wedge \gamma_{\oldepsilon}^*) = -2\pi iC_1 |\gamma_{\oldepsilon}|^2 \omega^n,$$ and have written $\Im(e^{-i\varphi_{\oldepsilon}(E)} Z_{\oldepsilon}(S))^1$ to denote the $\oldepsilon^1$ term in the expansion. This is precisely the desired Equation \eqref{eqn:wanted} to leading order $\oldepsilon$, where we observe that, $C_{\oldepsilon} = \oldepsilon C_1 + O(\oldepsilon^2)$.  What we have crucially used here is that  $c>0$, which from \cref{lemma:largevolume} follows from our assumption that $\Im(\rho_{n-1}/\rho_n) > 0,$ the crucial \emph{stability vector assumption} on our central charge $Z$. Curiously, at higher order in $\oldepsilon$ the lower order inequalities $\Im (\rho_{i-1}/\rho_i) > 0$ do not come into the argument, again reinforcing the observation that we only require this leading inequality in our work.
	
	We will now argue that at each higher order $\oldepsilon^{j}$, we obtain a similar expansion. Each of the terms appearing in the $Z$-critical equation at order $\oldepsilon^j$ involve differential forms of the form
	\begin{equation}\label{eqn:termform} C \oldepsilon^j\omega^{n-j} \wedge F_{A_{\oldepsilon}}^p \wedge \tilde U_{j-p}\end{equation}
	for $p$ possibly between $0$ and $j$ and $C$ some $\oldepsilon$-independent constant. 
	
	First let us note that if $p=0$, then this term is independent of the subbundle $S$ and is unaffected by our taking the $\tr_S$ in \eqref{eqn:zeqnsubbundle}, and so after integrating is absorbed by $\Im(e^{-i\varphi_{\oldepsilon}(E)} Z_{\oldepsilon}(S))^j$ on the right-hand side of \eqref{eqn:wanted}, such terms being the same whether we take $Z_{\oldepsilon}(S)$ or $Z_{\oldepsilon}(E)$ in this expression.

	Now if $p>0$, we need to understand the block matrix decomposition of a product of curvature terms
	\begin{align*}
	&\underbrace{F_{A_{\oldepsilon}} \wedge \cdots \wedge F_{A_{\oldepsilon}}}_{p \text{ times}}\\
	&= \begin{pmatrix}
	F_{S,{\oldepsilon}} - \gamma_{\oldepsilon} \wedge \gamma_{\oldepsilon}^* & d_A \gamma_{\oldepsilon}\\
	-d_A \gamma_{\oldepsilon}^* & F_{Q,{\oldepsilon}} - \gamma_{\oldepsilon}^* \wedge \gamma_{\oldepsilon}
	\end{pmatrix} \underbrace{\wedge \cdots \wedge}_{p \text{ times}} \begin{pmatrix}
	F_{S,{\oldepsilon}} - \gamma_{\oldepsilon} \wedge \gamma_{\oldepsilon}^* & d_A \gamma_{\oldepsilon}\\
	-d_A \gamma_{\oldepsilon}^* & F_{Q,{\oldepsilon}} - \gamma_{\oldepsilon}^* \wedge \gamma_{\oldepsilon}
	\end{pmatrix}.
	\end{align*}
	We will be interested in the $\tr_S$ of the terms that appear in the top left block of this matrix decomposition. This will in general involve a term of the form 
	$$F_{S_{\oldepsilon}} \underbrace{\wedge \cdots \wedge}_{p \text{ times}} F_{S,{\oldepsilon}},$$
	which after taking $\tr_S$ and wedging with the differential forms as in \eqref{eqn:termform} and integrating, gives the required factor for $\Im (e^{-i\varphi_{\oldepsilon}} Z_{\oldepsilon}(S))^j$ in \eqref{eqn:wanted}, which is
	$$ \int_X C \oldepsilon^j\omega^{n-j} \wedge \tr_S F_{S_{\oldepsilon}}^p \wedge \tilde U_{j-p}.$$
	In addition to this desired term, we will also obtain other terms all involving at least one of 
	\begin{equation} \label{eqn:list} \gamma_{\oldepsilon}^* \wedge \gamma_{\oldepsilon},\quad  \gamma_{\oldepsilon} \wedge \gamma_{\oldepsilon}^*,\quad  d_{A_{\oldepsilon}} \gamma_{\oldepsilon},\quad  d_{A_{\oldepsilon}} \gamma_{\oldepsilon}^*,\quad F_{Q,{\oldepsilon}}\end{equation} 
	and also some possible factors of $F_{S,{\oldepsilon}}$. We wish to show that no matter what product we obtain containing these terms, the corresponding form can be absorbed by the factor $C_{\oldepsilon} \|\gamma_{\oldepsilon} \|^2 \omega^n$ in \eqref{eqn:wanted} after taking a trace and integrating.
	
	At order $\oldepsilon^j$ with a differential form of the form \eqref{eqn:termform}, our curvature component consists of a product of $p$ terms in the list \eqref{eqn:list} given above, with $p$ at most $j$. Following Leung's notation, let us call such a product $\caT_p$, so our form of interest is
	$$\omega^{n-d} \wedge \caT_p \wedge \tilde U_{j-p}.$$
	We will show that provided $\oldepsilon$ is sufficiently small, any such form is much smaller in norm than $C_1 \oldepsilon \|\gamma_{\oldepsilon} \|^2$ appearing in \eqref{eqn:wanted}, and therefore can be absorbed into this term whilst preserving the positivity of $C_{\oldepsilon} \|\gamma\|^2$. There are three cases to consider.
	
	\begin{enumerate}
		\item \emph{$\caT_p$ containing $\gamma_{\oldepsilon}^* \wedge \gamma_{\oldepsilon}$ or $\gamma_{\oldepsilon} \wedge \gamma_{\oldepsilon}^*$:}
		
		First notice that by the uniform estimates \eqref{eqn:uniformoldepsilonbounds}, $\oldepsilon$ times any term in the list \eqref{eqn:list}, or a term $F_{S,\oldepsilon}, F_{Q,\oldepsilon}$ tends to zero as $\oldepsilon\to 0$. Thus if we have a term of the form $\gamma_{\oldepsilon}^* \wedge \gamma_{\oldepsilon}$ or $\gamma_{\oldepsilon} \wedge \gamma_{\oldepsilon}^*$ in our product $\caT_p$, we have an expression
		$$\pm \oldepsilon^j \omega^{n-j} \wedge \caT'_{p-1} \wedge \tilde U_{j-p} \wedge \gamma_{\oldepsilon}^* \wedge \gamma_{\oldepsilon}$$
		where $\caT'_{p-1}$ consists of the remaining $p-1$ factors in $\caT_p$, and the sign depends on the order of $\gamma_{\oldepsilon}$ or $\gamma_{\oldepsilon}^*$ in our wedge term. But we can rewrite this as
		$$\pm  \omega^{n-j} \wedge \oldepsilon^{p-1} \caT'_{p-1} \wedge  \oldepsilon^{j-p} \tilde U_{j-p} \wedge \oldepsilon  \gamma_{\oldepsilon}^* \wedge \gamma_{\oldepsilon}.$$
		Then by our initial observation, $\oldepsilon^{p-1} \caT'_{p-1}$ tends to zero as $\oldepsilon\to 0$, so for $\oldepsilon$ taken sufficiently small we can estimate (after taking trace and integrating over $X$),
		$$\|\omega^{n-j} \wedge \caT_p \wedge \tilde U_{j-p}\| \le  c_{\oldepsilon} \oldepsilon\|\gamma_{\oldepsilon}\|^2$$
		where the constant $c_{\oldepsilon}$ depending on our factor $\oldepsilon^{p-1} \caT'_{p-1}$ is as small as we like provided we take $\oldepsilon$ sufficiently small. Such a term is therefore small in norm compared to $C_1 \oldepsilon \|\gamma_{\oldepsilon}\|^2$ for $\oldepsilon$ sufficiently small, as desired.
		\item \emph{$\caT_p$ containing $d_{A_{\oldepsilon}} \gamma_{\oldepsilon}$ or $d_{A_{\oldepsilon}} \gamma_{\oldepsilon}^*$:}
		
		If there is no $\gamma_{\oldepsilon}^* \wedge \gamma_{\oldepsilon}$ or $\gamma_{\oldepsilon} \wedge \gamma_{\oldepsilon}^*$ term in the product $\caT_p$, but there is a term of the form $d_{A_{\oldepsilon}} \gamma_{\oldepsilon}$ or $d_{A_{\oldepsilon}} \gamma_{\oldepsilon}^*$, then we may integrate by parts when computing \eqref{eqn:wanted}.  This shifts $d_{A_{\oldepsilon}}$ on to the other terms appearing in the differential form
		$$\omega^{n-j} \wedge \caT_p \wedge \tilde U_{j-p}.$$
		Using the Leibniz rule for the exterior covariant derivative $d_{A_{\oldepsilon}}$, we can deal with each possibility in cases. If $d_{A_{\oldepsilon}}$ is applied to a term of the form $F_S$ or $F_Q$ after integrating parts, or a form $\tilde U_{j-p}$ or $\omega^{n-j}$, then this will vanish by the Bianchi identity $d_{A_{\oldepsilon}} F_{\oldepsilon} = 0$ or closedness of $\omega$ and $\tilde U$. 
		
		Thus the only non-vanishing possibilities occur if, after integrating by parts, the $d_{A_{\oldepsilon}}$ is applied to a term of the form  $d_{A_{\oldepsilon}} \gamma_{\oldepsilon}$ or $d_{A_{\oldepsilon}} \gamma_{\oldepsilon}^*$. Now we recall that in fact $d_{A_{\oldepsilon}} \gamma = \del_{A_{\oldepsilon}} \gamma$ and similarly $d_{A_{\oldepsilon}} \gamma^* = \db_{A_{\oldepsilon}} \gamma^*$. Thus, for example, if we started with $\del_{A_{\oldepsilon}} \gamma$ and our product $\caT_p$ contains another term $\del_{A_{\oldepsilon}} \gamma$, after integrating by parts we would obtain $\gamma \wedge \del_{A_{\oldepsilon}}^2 \gamma = 0$, and similarly for when we have $\db_{A_{\oldepsilon}} \gamma^*$. Thus we reduce just to the case where we have a factor $\gamma \wedge \del_{A_{\oldepsilon}} \db_{A_{\oldepsilon}}\gamma^*$ or $\gamma^* \wedge \db_{A_{\oldepsilon}} \del_{A_{\oldepsilon}}\gamma$ in our product after integrating by parts.
		
		Using the fact that $\db_{A_{\oldepsilon}} \gamma_{\oldepsilon} = \del_{A_{\oldepsilon}} \gamma_{\oldepsilon}^* = 0$, and that $F_{A_{\oldepsilon}} \wedge \gamma_{\oldepsilon}= d_A^2 \gamma_{\oldepsilon} = (\del_{A_{\oldepsilon}} \db_{A_{\oldepsilon}} + \db_{A_{\oldepsilon}} \del_{A_{\oldepsilon}})  \gamma_{\oldepsilon}$, we will therefore obtain more curvature terms times terms involving $\gamma_{\oldepsilon}^* \wedge \gamma_{\oldepsilon}$ or $\gamma_{\oldepsilon} \wedge \gamma_{\oldepsilon}^*$. This lands us in the previous situation above, which we have already dealt with, so we are done in this case.
		
		\item \emph{$\caT_p$ not containing a term with a $\gamma_{\oldepsilon}$:}
		
		When there is no $\gamma_{\oldepsilon}$ term appearing, from the block matrix decomposition of the curvature form we see the only non-zero terms which can appear are $p$-fold products of $F_{S,\oldepsilon}$ or $F_{Q,\oldepsilon}$. The latter have vanishing $\tr_S$ and the former terms were already accounted for, being the term which contributes to $\Im(e^{-i\varphi_\oldepsilon} Z_\oldepsilon (S))^j$. 
	\end{enumerate}
	
	This shows that every form with curvature part given by a product $\caT_p$ is very small in norm relative to $C_1 \oldepsilon \|\gamma\|^2$ provided we choose $\oldepsilon>0$ small enough. Thus \eqref{eqn:wanted} holds and since the left hand side is zero, for such small $\oldepsilon$ we have
	$$\Im(e^{-i\varphi_{\oldepsilon}(E)} Z_{\oldepsilon}(S)) < 0$$
	so $E$ is asymptotically $Z$-stable with respect to the subbundle $S$. 
\end{proof}

\begin{remark}\label{leung-errors} Our proof is a variant of a result of Leung, where he shows that the existence of ``almost Hermite--Einstein metrics'' implies Gieseker stability with respect to subbundles \cite[Proposition 3.1]{leung1997einstein}. Much like asymptotic $Z$-stability, Gieseker stability requires an inequality of topological quantities associated to subbundles and subsheaves. Leung further states that the existence of almost Hermite--Einstein metrics actually implies Gieseker stability in general, with respect to not only subbundles but also subsheaves. Leung's proof of this statement relies on the claim that if $$f: S \to Q_0$$ is a morphism between coherent sheaves such that $S$ is Gieseker stable, $Q_0$ is slope stable and $S, Q_0$ have the same slope, then $f$ must be zero or an isomorphism \cite[page 530]{leung1997einstein}. It follows from general theory that such an $f$ must be zero or surjective, but it seems likely to the authors that it could happen that $f$ is surjective but not injective. In any case, there is no analogous property for asymptotic $Z$-stability, leading us to restrict ourselves to stability with respect to subbundles, which is the situation relevant to sufficiently smooth vector bundles.
\end{remark}

\section{Existence of $Z$-critical connections on $Z$-stable bundles}\label{sec:perturbation}

We consider a holomorphic vector bundle $E$ over a compact K\"ahler manifold $(X,\omega)$, together with a polynomial central charge $Z$. We are now in a position to prove our main result.

\begin{restatable}{theorem}{intromainthm}\label{thm:intromainthm} A simple, sufficiently smooth holomorphic vector bundle $E$ admits uniformly bounded $Z_k$-critical connections for all $k \gg 0$ if and only if it is asymptotically $Z$-stable. \end{restatable}

It is equivalent to solve the $Z$-critical equation viewed as a partial differential equation on the space of Hermitian metrics, or as a partial differential equation on the space of connections inducing the fixed holomorphic structure. With this in mind, as in \cref{lemma:bounds}, uniform boundedness means that the Hermitian metrics  $h_{\oldepsilon}$ providing solutions of the $Z_k$-critical equation, with $\oldepsilon = 1/k$, are uniformly bounded in $C^2$. In \cref{sec:leunganalogue} we showed that existence of such solutions implies asymptotic $Z$-stability, so what remains is to prove the following:

\begin{theorem}
\label{thm:stabilityimpliesexistence}
	Suppose $E$ is a is asymptotically $Z$-stable and sufficiently smooth. Then for all $k\gg 0$, $E$ admits a $Z_k$-critical connection. 
\end{theorem}

The uniform boundedness will be an obvious consequence of the proof. The strategy is as follows. By \cref{cor:polystabledegen}, our asymptotically $Z$-stable bundle $E$ admits a slope polystable degeneration to $\Gr(E)$, its associated graded object, which is a vector bundle by our assumption of sufficient smoothness. Since $\Gr(E)$ is a direct sum of stable bundles of the same slope, it admits a Hermitian Yang--Mills connection \cite{donaldson1985anti, uhlenbeck1986existence}. Our goal is to perturb this connections to $Z_k$-critical connections on $E$ for $k \gg 0$, and the main challenge is that this is only possible under the assumption of asymptotically $Z$-stability. The approach we take is to first employ a family of gradient flows in an associated finite-dimensional setting, which is the step at which we use our assumption of asymptotic $Z$-stability. We then construct approximate solutions to the $Z_k$-critical equation and use a version of the inverse function theorem to construct genuine $Z_k$-critical connections, for which the main step will be to understand the linearisation of the $Z_k$-critical operator.

\subsection{The case when $E$ is slope stable}\label{sec:stable}

We begin by proving the result in the considerably simpler situation that $E$ is actually slope stable. 

\begin{theorem}
\label{thm:stabilityimpliesexistencestable}
	Suppose $E\to (X,[\omega])$ is slope stable. Then for all $k\gg 0$, $E$ admits $Z_k$-critical connections with respect to the metric $k\omega$. 
\end{theorem}

First we observe that by \cref{lemma:largevolume}, the leading order term in the $Z$-critical equation is the \emph{weak} Hermite--Einstein condition. It is a standard result that the existence of a weak Hermite--Einstein metric is equivalent to the existence of a Hermite--Einstein metric \cite[\S 4.2]{kobayashi2014differential}, requiring only a conformal change of the Hermitian metric. Thus when a bundle is stable, the Hitchin--Kobayashi correspondence affords us a solution to the weak Hermite--Einstein equation for any function $f$ having the appropriate integral  \cite{donaldson1985anti, uhlenbeck1986existence}. In the case of $Z$-critical connections, we will apply the Hitchin--Kobayashi correspondence to obtain a solution to the weak Hermite--Einstein equation $$\Lambda_{\omega} F_{h_0} = if \Id_E,$$ with the function $f\in C^{\infty}(X,\R)$ defined by
$$f = -2\pi \left( \frac{\deg_U(E)}{(n-1)! [\omega]^n\rk(E)} - \contr_{\omega} \tilde U_2 \right),$$ arising through \cref{lemma:largevolume}.
The function $f$ will be fixed throughout, so we will refer to a connection solving this equation as simply a weak Hermite--Einstein connection.

It will be convenient to consider powers $\oldepsilon = 1/k$, so that we are interested in the range $0 < \oldepsilon \ll 1$. We recall then that our notation for the $Z$-critical operator is $D_{\oldepsilon}$, so that the equation  of interest is
$$D_{\oldepsilon}(h) = 0$$
with argument $h$ a Hermitian metric, or equivalently the associated Chern connection $A$ when this is more convenient. The linearisation of the operator $D_{\oldepsilon}$  is denoted $P_{\oldepsilon}$.

\begin{remark} Throughout this section, for notational convenience we rescale the $Z$-critical operator $D_{\oldepsilon}$ by a factor of $\frac{\oldepsilon^{2n-1}}{c \rk (E) [\omega]^n}$, so that by \cref{lemma:largevolume} it satisfies $$D_{\oldepsilon} = D_0 + O(\oldepsilon),$$ where $D_0$ is the weak Hermite--Einstein operator. \end{remark}

\begin{proof}[Proof of \cref{thm:stabilityimpliesexistencestable}]
	Since $E$ is slope stable, it admits a weak Hermite--Einstein metric $h_0$ through the Hitchin--Kobayashi correspondence. By \cref{lemma:largevolume} the leading order term of the $Z$-critical equation is the  weak Hermite--Einstein equation, so we immediately have
	$$D_0 (h_0) = 0,\qquad \|D_{\oldepsilon} (h_0) \| \le C\oldepsilon$$
	for some constant $C$, where $D_0$ denotes the weak Hermite--Einstein operator. As used in \cref{lemma:linearisation}, when taking the point of view of changing the $\db$-operator on $E$ through the action of the gauge group, the linearisation of the weak Hermite--Einstein operator $D_0$ at $h_0$ is given by $$P_0 = \Lap_{\db_{h_0}},$$ the bundle Laplacian on $\End E$ with respect to $h_0$. Since $E$ is stable, it is simple, and so the kernel of $P_0$ consists only of the constant endomorphisms, and $P_0$ is invertible orthogonal to this kernel. 
	
	We  pass to Banach spaces, and view our linearisation is an invertible operator
	$$P_{\oldepsilon} : L_{d+2,0}^2 (\End E) \to L_{d}^2(\End E)$$
	where $d\in \Z_{\ge0}$ is some non-negative integer and $L_{d,0}^2$ denotes the Sobolev of space trace-integral zero endomorphisms of $E$. Our previous discussion produces the estimate
	$$\|P_{\oldepsilon} - P_0\| \le C\oldepsilon$$
	for some $C$ independent of $\oldepsilon$, and since $P_0=\Lap_{\db_{h_0}}$ is invertible modulo constant endomorphisms and invertibility is an open condition, for $\oldepsilon$ sufficiently small, $P_{\oldepsilon}$ is also invertible on this space. If $G$ denotes the inverse of $P_0$ and $Q_{\oldepsilon}$ denotes the inverse of $P_{\oldepsilon}$, then we also obtain a bound
	$$\frac{1}{C'} \|G \| \le \|Q_{\oldepsilon}\| \le C'\|G\|$$
	for some constant $C'$. By the inverse function theorem for Banach spaces applied to the point $h_0$, there exists a neighbourhood of $D_{\oldepsilon}(h_0)$ with size independent of $\oldepsilon$ which is in the image of the operator $D_{\oldepsilon}$. In particular, since $D_{\oldepsilon}(h_0) \to 0$ as $\oldepsilon\to 0$, there exists some $\oldepsilon_0>0$ such that for all $0<\oldepsilon<\oldepsilon_0$, there exists a solution
	$$D_{\oldepsilon}(h_{\oldepsilon}) = 0$$
	for $h_{\oldepsilon} = \exp (V_{\oldepsilon}) h_0$ for some $V_{\oldepsilon}$ in the Sobolev space $L_{d+2,0}^2(\End E).$ But as was proven in \cref{lemma:linearisation}, the $Z$-critical equation is elliptic for all $\oldepsilon$ sufficiently small, so by elliptic regularity this $h_{\oldepsilon}$ is actually smooth and hence is a genuine solution of the $Z_{\oldepsilon}$-equation. 
\end{proof}

\subsection{The general case: finite-dimensional results}

\subsubsection{Kuranishi theory and moment map flows}\label{sec:choosecplxstr}

We next consider a simple, sufficiently smooth slope semistable vector bundle $E$, such that its graded object $\Gr(E)$ is a slope polystable vector bundle. We assume that $E$ is asymptotically $Z$-stable and aim to construct $Z_{\oldepsilon}$-critical connections on $E$ for all $\oldepsilon \ll 1$. In the previous section, which considered the case when $E$ is slope stable, the inverse function theorem was used to perturb the Hermitian Yang--Mills connection on $E$ (whose existence follows from slope stability) to $Z_{\oldepsilon}$-critical connections. The key step in this argument was the invertibility of the Laplacian associated to the  Hermitian Yang--Mills connection, orthogonal to constant endomorphisms. 

In our setting the relevant linear operator is $ \Delta_{\Gr(E)}$, where analytically the key new difficulty is the nontriviality of the kernel $\ker \Delta_{\Gr(E)}$ of the Laplacian on $\Gr(E)$. Geometrically, this kernel is the Lie algebra of the Lie group $K$ of isometries of $(\Gr(E),h)$. The complexification $G = K^{\C}$ acts on the infinite-dimensional space $\scA$ of unitary connections on the smooth Hermitian vector bundle underlying $\Gr(E,h)$, fixing the Hermitian Yang--Mills connection $A_0$ inducing $\Gr(E)$, and we next consider a finite-dimensional space where we can further examine this action.

The space is the \emph{Kuranishi space} of $\Gr(E)$, which is an open neighbourhood $U$ of the origin in the vector space $H^1(X, \End \Gr(E))$, which we take to be of dimension $m$, and which comes with a holomorphic embedding $\Phi: U \to \scA$ with $\Phi(0) = A_0$; we refer to \cite[Section 2]{BS} for a discussion of the Kuranishi space and details of all properties we require. The image of the  Kuranishi space under the embedding $\Phi$ parametrises both integrable and non-integrable connections; the integrable locus is a complex subspace of $U$. By versality of the Kuranishi space, the holomorphic vector bundle $E$ is associated to an integrable unitary connection  $\Phi(z_0)\in\scA$ for $z_0 \in U$. This point $z_0$ is not unique, and our next step is to choose an $\oldepsilon$-dependent sequence of such points $z_{\oldepsilon}$ whose image is ``closest'' to being $Z_{\oldepsilon}$-critical.

For this we move the initial point $z_0$ inside its $G$-orbit. The group $K$ acts naturally on $U$ induced from the linear action on $H^1(X, \End \Gr(E))$, in such a way that $\Phi$ is $K$-equivariant; the only obstruction to extending this to a $G$-equivariant embedding is the fact that $U$ is only an open neighbourhood of the origin in $H^1(X, \End \Gr(E))$. To choose the points $z_{\oldepsilon}$, we use moment maps and begin by endowing $U$ with a sequence of $K$-invariant K\"ahler metrics. As the map $\Phi$ is a holomorphic embedding, we may pull back the forms $\Omega_{Z_{\oldepsilon}}$ on $\scA$ to $K$-invariant closed $(1,1)$-forms on $\scA$ which we denote $$\omega_{\oldepsilon} = \Phi^*\Omega_{Z_{\oldepsilon}}.$$ Since $\Omega_{Z_{0}}$ is the usual Atiyah--Bott K\"ahler metric  on $\scA$, the closed $(1,1)$-forms  $\omega_{\oldepsilon}$ are K\"ahler on $U$ for all $\oldepsilon$ sufficiently small. 

The Lie algebra of the subgroup $K \subset \G$ (where $\G=\Gamma(X,\Uni(E,\C))$ is the gauge group) corresponds to a finite-dimensional Lie subalgebra $\mfk \subset \scA^0(\End E) = \Lie \G$, which is given as the kernel $\ker \Delta_{\Gr(E)}$. The $Z_{\oldepsilon}$-critical operator $$\Ima(e^{-i\phi_{\oldepsilon}}\tilde Z_{\oldepsilon}(-)): \scA \to \scA^0(\End E)$$ is a moment map for the $\G$-action on $\scA$ with respect to the $\Omega_{Z_{\oldepsilon}}$ (where we have identified the Lie algebra with its dual via the $L^2$-inner product), so the induced moment map for the action of $K\subset \G$ is the composition of the $Z$-critical operator with the $L^2$-projection $$p: \scA^0(\End E) \to \mfk.$$ Here we note that while for the majority of our work we only consider integrable unitary connections, the moment map property was established for arbitrary unitary connections. Thus the moment map for the $K$-action on $(U,\omega_{\oldepsilon})$ is $$z \to  p(\Ima(e^{-i\phi_{\oldepsilon}}\tilde Z_{\oldepsilon}(\Phi(z)));$$ we denote this operator by $\mu_{\oldepsilon}: U \to \mfk$. A zero of the moment map $z_{\oldepsilon}$ thus satisfies $$p(\Ima(e^{-i\phi_{\oldepsilon}}\tilde Z_{\oldepsilon}(\Phi(z_{\oldepsilon}))) = 0.$$

We next use our hypothesis that $E$ is asymptotically $Z$-stable to construct zeroes of the moment maps $\mu_{\oldepsilon}$ inside the $G$-orbit of $z_0$. These will be constructed for $\oldepsilon \ll 1$, but this condition will enter only into the final step of the proof. Thus for the moment we fix $\oldepsilon>0$. Our problem fits into a class of problems relating stability to moment maps, originating with the Kempf--Ness theorem. The new difficulty in our situation is that the problem is a local one as $U$ is not compact, which leads us to employ an associated geometric flow. The flow is the gradient flow of the norm-squared of the moment map, which explicitly takes for form 
\begin{align*}\dot z(t) =& Jv_{\mu_{\oldepsilon}(z(t))}, \\
z(0) =& z_0,
\end{align*}
where $\mu_{\oldepsilon}(z(t)) \in \mfk$, $v_{\mu_{\oldepsilon}(z(t))} \in T_{z(t)} U$ is the associated tangent vector and $J$ is the almost complex structure on $U$. We call this flow the \emph{moment map flow}. Our main reference for the technical results we use is Georgoulas--Robbin--Salamon \cite{moment-weight} (some of which, used here, is based on Chen--Sun \cite[Section 4.1]{CS}).

\begin{proposition}\label{flow-bounded-region}  There is a constant $C>0$ independent of $\oldepsilon$ such that the moment map flow remains in the ball of radius $C\sqrt{\oldepsilon}$ around the origin in $U \cap G.z_0$ for all $t>0$ and all sufficiently small $\oldepsilon >0$.

\end{proposition}

\begin{proof}
It suffices to show that there is a $C>0$ such that $Jv_{\mu_{\oldepsilon}(z(t))}$ lies in the ``right halfplane'' centred at $z(t)$ whenever $z(t)$ is close to the boundary of the ball of radius $C\sqrt{\oldepsilon}$, namely that $Jv_{\mu_{\oldepsilon}(z(t))}\cdot z(t) > 0$ whenever $z(t)$ lies in such a region. Indeed, differentiating the norm squared of $z(t)$ and using that $z(t)$ is a solution of the moment map flow, this implies that the derivative of the norm of $z(t)$ is \emph{negative} wherever this holds. We will show that condition  $Jv_{\mu_{\oldepsilon}(z)}\cdot z > 0$ holds for $z \in\partial B_{C\sqrt{\oldepsilon}} (0) \cap \overline{G.z_0}$ for a suitably chosen constant $C>0$, from which it follows that the flow remains in $B_{C\sqrt{\oldepsilon}} (0)$ for all $t$. 

We prove this first for the flat K\"ahler metric $\omega_{\Euc}$ arising from the Hermitian inner product defined by $\omega_0$ on $T_0U$, and choose coordinates so that 
$$ \omega_{\Euc} = i\sum_{j=1}^{m} dz_j\wedge d\bar z_j,$$ where we recall $m = \dim U$.
Associated to any $z \in U$ is a natural element of $\mfk$ defined by 
$$p(izz^*) + \mu_{\oldepsilon}(0);$$ here $izz^*$ is a skew-Hermitian matrix and we view $p$ as the projection, defined through the Hermitian inner product on $H^1(X, \End \Gr(E))$, from skew-Hermitian matrices onto $\mfk \subset \mfu(m)$ (where this inclusion is the induced by the inclusion $K \subset U(m)$). The element $p(izz^*) + \mu_{\oldepsilon}(0) \in \mfk$ then induces a tangent vector $\xi_{\oldepsilon}(z) \in T_z U$  defined as $$\xi_{\oldepsilon}(z) = p(izz^*)z + \mu_{\oldepsilon}(0) z.$$ Note that $\xi_{\oldepsilon}(z)$ is tangent to $G\cdot z$, and in fact, $\dot z(t) = -J\color{black} \xi_{0}(z(t))$ is the moment map flow with respect to the Euclidean metric, where $\xi_0 = p(izz^*)z$. We next claim that we can find  a $C>0$ so that $\mu_{\oldepsilon}(z)$ always lies in the desired halfplane at $z$. We begin by showing this for a fixed direction of $z$, and then we show we can find a $C$ that works for all $z$.

First note that the $G$-action generates a linear subspace at every tangent space $T_z V$ of $V$. Let $q = q_z : T_z V \to T_z V$ be the orthogonal projection to this subspace, using the Euclidean metric, so that $$
\xi_0(z) \cdot \bar z = |q(z)|^4.
$$
Note that $q(z) = 0$ if and only if $z$ is a fixed point of the $G$-action. One direction is clear, since if $z$ is a fixed point, then the image of $q$ is $\{ 0 \}$. In the reverse direction, suppose $q(z) =0$. Note that $\xi.z = \xi. q(z)$, where we have identified $\xi \in \mfg$ with a matrix, the action being a linear one on a vector space. Thus if $q(z) = 0$, then $\xi.z=0$ for all $\xi$. It follows that 
\begin{align*}
 \exp(t.\xi) z - z = \sum_{i=1}^{\infty} \frac{t^i \xi^{i}.z}{i!} = 0.
\end{align*}
Since $\xi \in \mfg$ was arbitrary, this implies that the $G$-action on $z$ is trivial, which precisely means that $z$ is a fixed point. 

Next, 
\begin{align*}
\langle \xi_{\oldepsilon}(z) ,  z \rangle =& \xi_0(z) \cdot \bar z + \mu_{\oldepsilon}(0) z \cdot \bar z
\end{align*}
and since $\mu_{0}(0) = 0$, we have $$
|\mu_{\oldepsilon}(0) z \cdot \bar z| \leq c \oldepsilon  |z|^2,
$$
for some $c>0$, independent of $\oldepsilon$. If $q(z) \neq 0$, then, by choosing $C>0$ sufficiently large, we can ensure that 
\begin{align}
\label{eq:Cbound}
c |z|^4 < \frac{1}{2} C |q(z)|^4,
\end{align} 
which we note is scale invariant in $z$. For such a choice of $C$, we have that if $z \in B_{C\sqrt{\oldepsilon}} (0)$ so that $|z|^2 = C \epsilon$, then
\begin{align*}
|\mu_{\oldepsilon}(0) z \cdot \bar z| &\leq c \epsilon |z|^2 \\
&= \frac{c |z|^4}{C} \\
&< \frac{1}{2} |q(z)|^4 \\
&= \frac{1}{2} \xi_0(z) \cdot \bar z .
\end{align*}
This in turn ensures that $\xi_{\oldepsilon}(z) \cdot \bar z =  \xi_0(z) \cdot \bar z + \mu_{\oldepsilon}(0) z \cdot \bar z >0$ for such $z$ for all $\oldepsilon$, since $\xi_0(z) \cdot \bar z >0$. 

While the particular value of the constant $C$ depends on the direction of $z$, we will next show that $C$ can be chosen uniformly in any compact subset outside the locus where $q(z) = 0$. This will in particular imply that we can choose it uniformly for $z$ in the closure of the $G$-orbit of $z_0$, since the only fixed point in the closure of the $G$-orbit of $z_0$ is the origin (as there is a unique closed orbit in the closure of $G.z_0$), hence for any non-zero point in the closure of the $G$-orbit of $z_0$, the projection $q(z) \neq 0$ is nonvanishing. By the choice of $C$, this is equivalent to showing that $\frac{|z|}{|q(z)|}$ can be uniformly bounded from below on any such set. Note that this will be implied by the lower semicontinuity of $|q(z)|$. In turn, this will be a consequence of the fact that the dimension of the image of $q$ can only increase near $z$. 

Indeed, fix a $z$ with $q(z) \neq 0$. Then, in particular, the image of $q$ at $z$ is non-trivial, so there exists $\xi_1, \ldots, \xi_d \in \mfg$ generating a basis for the image of $q$ in $T_z U$. We claim that there is an open neighbourhood $W$ of $z$ such that for all $w \in W$, the tangent vectors at $T_w U$ induced by the $\xi_j$ are linearly independent. This is a consequence of upper semicontinuity of the dimension of the Lie algebra of the stabiliser $\mfk_z$ of $z$. In more detail, if  $g \in G$ fixes an element in $B_{\delta} (z)$ for any sufficiently small $\delta >0$, then it also fixes $z$, since the elements of $U$ fixed by a given element of $G$ is closed. Now, since the image $A_w$ of the $\xi_j$ in $T_w U \cong \C^n$ varies continuously with $w$ near $z$, the norm of the projection $q'_w : T_w U \to A_w$ to the subspace $A_w$ \emph{is} continuous. Since $|q(w)| \geq |q'(w)|$ near $z$, we have that $|q(w)|$ is bounded below by a continuous function near $z$ and so this shows that $|q(w)|$ is lower semicontinuous at $z$. 

At this stage, we have shown that there is a $C>0$ such that $\xi_{\oldepsilon}(z) \cdot \bar z < 0$ whenever $z \in \partial B_{C \sqrt{\oldepsilon}} \cap \overline{G.z_0}$. We next consider the general case, where we replace $\xi_{\oldepsilon}(z)$ by the value of the moment map with respect to our actual K\"ahler metric. Here the tangent vector at $z(t)$ is associated to ($-J$ times) the value of the moment map ${\mu_{\oldepsilon}(z(t))}$. Since the K\"ahler metric $\omega_0$ is Euclidean to leading order, we have $$\omega_0 = \omega_{\Euc} + O(|z|).$$ Similarly comparing the closed $(1,1)$-forms $\omega_{\oldepsilon}$ to $\omega_0$, we have $\omega_{\oldepsilon} = \omega_0 + O(\oldepsilon)$. It follows that the moment maps $\mu_{\oldepsilon}$ satisfy 
$$\mu_{\oldepsilon}(z)= p(zz^*) + \mu_{\oldepsilon}(0) + O(|z|^3) + O(|z|\oldepsilon).$$ 
Thus 
\begin{align*}
\left(\mu_{\oldepsilon}(z)z - \xi_{\oldepsilon}(z)\right)\cdot \bar z = O(\oldepsilon^2 |z|^2)
\end{align*}
on $ \partial B_{C \sqrt{\oldepsilon}}$, and so
$$
\left| \left( \mu_{\oldepsilon}(z)z - \xi_{\oldepsilon}(z) \right)\cdot \bar z \right| < \frac{1}{2} \xi_{\oldepsilon}(z) \cdot \bar z 
$$
for all sufficiently small $\oldepsilon>0$, since the right hand side is an $O(\oldepsilon |z|^2)$-term and again, we are working on the region $\partial B_{C \sqrt{\oldepsilon}} \cap \overline{G.z_0}$, where $|z|$ is uniformly bounded by a multiple of $|q(z)|$. Since $\xi_{\oldepsilon}(z) \cdot \bar z >0$ here, it follows that $ \mu_{\oldepsilon}(z)z \cdot \bar z > 0$, too, which is what we wanted to show.
\end{proof}

We summarise the main consequence of this result.

\begin{corollary}
\label{cor:flowconvergence} The moment map flow  beginning at $z_0$ exists for all time and converges to a unique point $z_{\oldepsilon,\infty}\in \overline {G.z_0}$ satisfying $$\mu_{\oldepsilon}(z_{\oldepsilon,\infty}) \in \mfk_{z_{\oldepsilon,\infty}} \in U,$$ where $\mfk_{z_{\oldepsilon,\infty}}$ denotes the Lie algebra of the stabiliser of $z_{\oldepsilon,\infty}$ under the $K$-action. Furthermore, the points  $z_{\oldepsilon, \infty}$ lie in a ball of radius $C \sqrt{\oldepsilon}$, for some $C>0$ independent of $\oldepsilon$. \end{corollary}

\begin{proof}
The obstruction to the flow existing for all time is that in principle the flow could leave the region $U$ in which it is actually defined, which is ruled out by Proposition \ref{flow-bounded-region}. It is then a standard result that the flow exists for all time, while a $\L{}$ojasiewicz inequality argument implies that the flow converges to a unique point $z_{\oldepsilon, \infty}$ fixed by the flow (see \cite[Theorem 3.2]{moment-weight} for proofs of both of these claims). Since $z(t) \in G.z_0$ for all $t>0$ as $-Jv_{\mu_{\oldepsilon}(z(t))}$ is always tangent to $G.z(t)$, it then follows that  $z_{\oldepsilon,\infty}\in \overline {G.z_0}$. As $z_{\oldepsilon, \infty}$ is a fixed point of the flow, it follows that $\mu_{\oldepsilon}(z_{\oldepsilon, \infty}) \in \mfk_{z_{\oldepsilon}}$ by definition of the flow, as desired.\end{proof}

\subsubsection{Asymptotics of the moment map flow}\label{sec:asymptotics} We next consider the asymptotics of the moment map flow. The results of this section are independent of $\oldepsilon$, so we fix $\oldepsilon$ and omit it from our notation. Thus we have an open ball $U$ in a vector space $V = H^1(X,\End E)$ of dimension $m$, along with a K\"ahler metric $\omega$ on $U$. The vector space $V$ admits a linear action of a reductive group $G$, which is the complexification of $K$. The Lie algebras of $G$ and $K$ are denoted $\mfg$ and $\mfk$ respectively. The K\"ahler metric $\omega$ and the open ball $B$ are $K$-invariant and we fix a moment map $$\mu: U \to \mfk$$ for the $K$-action on $(U,\omega)$, where we have identified $\mfk$ with its dual via an invariant inner product. We fix $z_0 \in U$ with finite stabiliser under $G$ such that the moment map flow starting at $z$ converges to a point $z_{\infty} \in U$, as justified in our application by \cref{flow-bounded-region}. The main goal of the present section is to show that one of the following must hold, perhaps after moving $z_0$ to another point in its $G$-orbit inside $U$:
\begin{enumerate}
\item the point $z_{\infty}$ satisfies $\mu(z_{\infty})=0$, and there is an element $\xi \in \mfk$ such that $\exp(-it\xi)(z_0) \in U$ for all $t\geq 0$ and $$\lim_{t\to\infty} \exp(-it\xi)(z_0) = 
 z_{\infty} \in U;$$ 
\item there is a $\tilde z \in U$ and an element $\xi \in \mfk$ such that  such that $\exp(-it\xi)(z_0) \in U$ for all $t\geq 0$ and
  \begin{align*}\lim_{t\to\infty} \exp(-it\xi)(z_0) = \tilde z,\end{align*}with $$\langle \mu(\tilde z),\xi \rangle > 0.$$ 
\end{enumerate}

In the statement we allow ourselves to replace the initial point $z_0$ with another in its $G$-orbit such that the statement holds. We will then show in \cref{global-moment-map-geometry} that asymptotic $Z$-stability rules out the second possibility provided $0<\oldepsilon \ll 1$, while again by asymptotic $Z$-stability the first possibility forces $z_{\infty} \in G.z_0$, meaning the moment map flow converges to a zero of the moment map in $G.z_0 \cap U$. Note that once we show $z_{\infty} \in G.z_0$, it is automatic that $\mu(z_{\infty})=0$ as $z_0$ (hence $z_{\infty}$) has finite stabiliser and limit of the moment map flow satisfies $\mu(z_{\infty}) \in \mfk_{z_{\infty}}$.

The results we require around the asymptotics of the moment map flow have been established only in the projective setting, so we wish to view $(U,\omega)$ as a submanifold of a projective manifold in order to appeal to these results. Consider the natural embedding $$f: V \hookrightarrow \pr^m$$ as the complement of the hyperplane $x_0 =0$, where we use homogeneous coordinates $[x_0:x_1:\ldots:x_m]$ on $\pr^m$ associated to linear coordinates $x_1,\ldots,x_m$ on $V$. The $G$-action on $V$ extends naturally to an action on $\pr^m$, acting trivially on the $x_0$ coordinate, which is well defined since the $G$ action on $V$ is linear; this action makes $f$ a $G$-equivariant embedding, and hence makes $f$ a $K$-equivariant embedding when restricted to $U$. The following result states we may $K$-invariantly extend the K\"ahler metrics $\omega$ and the moment map $\mu$ to $\pr^m$.

\begin{lemma}
There is an open ball $U' \subset U$, a $K$-invariant K\"ahler metric $\omega_{\pr^m}$ on $\pr^m$ and a moment map $$\mu_{\pr^m}: \pr^m \to \mfk$$ such that viewing $U' \subset \pr^m$ as a submanifold via $f$:

\begin{enumerate}
\item $\omega_{\pr^m}|_{U'} = \omega$;
\item $\mu_{\pr^m}|_{U'} = \mu$.
\end{enumerate}
\end{lemma}

\begin{proof}We begin by constructing the K\"ahler metric $\omega_{\pr^m}$. In local holomorphic coordinates around the origin in $U$, we may write \begin{align*}\omega_{\FS} &= \ddb(|z|^2 + \psi_1(z)), \\ \omega &= \ddb(|z|^2 + \psi_2(z)),\end{align*} where $\psi_1(z)$ and $\psi_2(z)$ are $O(|z|^3)$ and $\omega_{FS} = \ddb(\log(1+|z|^2))$ is the Fubini--Study metric (which is $K$-invariant). We construct $\omega_{\pr^m}$ by gluing $\psi_1(z)$ to $ \psi_2(z)$ using $K$-invariant bump functions. 

Let $\chi : \R \to \R$ be a bump function such that $\chi(t) = 0$ if $t \geq 2$ and $\chi(t) = 1$ if $t \leq 1$. Let $r \in \R_{>0}$ be a sufficiently small positive parameter, and define 
\begin{align*}
\omega_r =&\begin{cases}
			\omega & \text{if } z \in B_r(0) \subset U \\
			\ddb( |z|^2 + \chi(|z|/r) \psi_1(z) + (1 - \chi(|z|/r)) \psi_2(z))  & \text{if } z \in \bar B_{2r}(0) \setminus B_{r}(0) \\
            \omega_{\FS} & \text{if } z \in \pr^m \setminus \bar B_{2r}(0)
		 \end{cases}
\end{align*}
which is a smooth closed two-form on $\mathbb{P}^m$ extending $\omega$ from $B_r(0)$ to $\pr^m$. We check that for a suitable choice of $r$, $\omega_r$ is K\"ahler. This amounts to checking that $$\eta_r = \ddb( |z|^2 + \chi(|z|/r) \psi_1(z) + (1 - \chi(|z|/r))\psi_2(z))$$ is positive on $ B_{2r}(0) \setminus B_{r}(0)$. Now,  \begin{align*}
\eta_r  =& \omega_{\Euc} \\ 
&+ i \partial \chi (|z|/r) \wedge \bar \partial ( \psi_1 (z) - \psi_2(z)) \\
&+ i \bar \partial \chi (|z|/r) \wedge \partial ( \psi_1 (z) - \psi_2(z)) \\
&+ ( \psi_1 (z) - \psi_2(z)) \ddb \chi (|z|/r) .
\end{align*}
By the Leibniz rule, $ \partial \chi (|z|/r)$ and  $\bar \partial \chi (|z|/r)$ are $O(r^{-1})$, whereas  $\partial (\psi_1 (z) -\psi_2(z))$ is $O(r^2)$ on $  B_{2r}(0) \setminus B_{r}(0)$, since the $\psi_i$ are $O(|z|^3)$. Thus the second and third summands are $O(r)$. Similarly, the final term is $O(r)$, since $\ddb \chi (|z|/r)$ is $O(r^{-2})$ and the $\psi_i$ are $O(r^3)$. In particular, since $\omega_{\Euc}$ is a fixed metric independent of $r$, we can choose $r$ sufficiently small such that $\eta_r \geq \frac{1}{2} \omega_{\Euc}$, which therefore is positive on $\bar B_{2r}(0) \setminus B_{r}(0)$. Choosing $r$ such that this holds, we obtain the first statement with $U' = B_r(0)$. 

We next extend the moment map. There is a canonical choice of moment map for the $U(m+1)$-action on $\pr^m$ with respect to the Fubini--Study metric, defined by $$\mu_{\FS}(x) = i\frac{xx^*}{|x|^2} \in \mfu(m+1),$$ which induces a moment map $\mu_{\FS}: \pr^m \to \mfk$ through the projection $\mfu(m+1) \to \mfk$ defined as the projection induced by the trace inner product on $\mfu(m+1)$. 
Our K\"ahler metric $\omega_{\pr^m}$ satisfies $$\omega_{\pr^m} = \omega_{\FS} + \ddb \psi$$ for a $K$-invariant function $\psi$ on $\pr^m$. We claim that there is a canonical moment map $\mu_{\pr^m}$ with respect to  $\omega_{\pr^m}$  defined by setting for $v \in \mfk$ (and denoting $v$ also the induced vector field on $\pr^m$) $$\langle \mu_{\pr^m}, v\rangle = \langle \mu_{\FS}, v\rangle - 2(Jv)(\psi),$$ where $J$ here is the complex structure on projective space, so that $$\ddb \psi = 2dJd\psi.$$ To prove this, first note that $$\iota_v \ddb \psi =  2 \iota_v dJd\psi = 2 d\iota_v Jd\psi = 2 d\iota_{Jv} d\psi = 2d((Jv)(\psi)),$$ where we used Cartan's magic formula, meaning $$\iota_v \omega_{\pr^m} = -d\langle \mu_{\pr^m}, v\rangle.$$ 

Thus we must only show that $\mu_{\pr^m}$ is an equivariant map from $\pr^m$ to $\mfk$, which means proving that for all $v\in \mfk$ and $k \in K$ $$\langle \mu_{\pr^m}, k\cdot v\rangle = k^*\langle \mu_{\pr^m}, k\cdot v\rangle,$$ where $k\cdot v$ denotes the adjoint action. Since $\mu_{\FS}$ is $K$-equivariant, this condition is equivalent to asking $$k^* ((Jv)(\psi)) = (Jk\cdot v)(\psi)$$ which (using that when we identify $v$ with a vector field on $\pr^m$, $iv$ is identified with $Jv$), follows from the equality \begin{align*}k^*\left( \frac{d}{dt}\bigg|_{t=0} \exp(itv)^*\psi\right) &=   \frac{d}{dt}\bigg|_{t=0}  k^*\exp(itv)^*k^{-1*}\psi, \\ &=\frac{d}{dt}\bigg|_{t=0} \exp(itk\cdot v)^*\psi, \end{align*} where we have used $K$-invariance of $\psi$. 

Finally the moment map $\mu_{\pr^m}$ is determined uniquely up to the addition of an element of $\mfk^K$ (namely an element invariant under the adjoint action), as is the moment map $\mu: U \to \mfk$; so by adding an element of $\mfk^K$ we may ensure that $\mu_{\pr^m}|_{U'} = \mu$, as required.  \end{proof}

We thus replace $U$ with $U'$ and assume we have such an embedding for $U$ itself. That the moment map and K\"ahler metric extend $K$-invariantly means that the two moment map flows---beginning at $z_0$ considered as either a point of $U$ or $\pr^m$---are identical, and so the limiting point $z_{\infty} \in U$ of the moment map flow is also the limit point of the flow considered on $\pr^m$. We now divide the argument into two cases, depending on whether or not  $z_{\infty}$ is a zero of the moment map. We ultimately aim to show that $z_{\infty}\in G.z_0$, meaning that the moment map flow produces a zero of the moment map in the orbit of interest, so we assume this is not the case.

\begin{proposition}\label{prop:zero}
If $\mu(z_{\infty})=0$ and $z_{\infty} \notin G.z_0$,  there is a $\xi\in \mfk$ such that $\exp(-it\xi)(z_0) \in U$ for all $t\geq 0$ and $$\lim_{t\to\infty} \exp(-it\xi)(z_0) = z_{\infty}.$$\end{proposition}

\begin{proof}

This is essentially a standard part of geometric invariant theory, related to the Hilbert--Mumford criterion, as in this case $z_{\infty}$ is a zero of the moment map $\mu$ and hence is polystable in the sense of geometric invariant theory. Without employing the full force of geometric invariant theory, one can argue as follows. Since $\mu_{\infty}(z_{\infty})=0$, the Lie algebra  $\mfg_{z_{\infty}}$ is reductive \cite[Lemma 2.3 (i)]{moment-weight}, which by taking a slice for the action allows us to construct such a $\xi$ \cite[Proposition 1 (ii)]{donaldson-b-stability}.
\end{proof}

One can in fact take the $\xi_{\infty}$ to be rational in this case  \cite[Proposition 1 (ii)]{donaldson-b-stability}, in the sense that it generates a $\C^*$-action on $\pr^m$ and hence $V$, though we will not need this. The remaining case is more involved, and will use further ideas around the relationship between moment maps and geometric invariant theory, for which our general reference is again Georgoulas--Robbin--Salamon \cite{moment-weight}.

Fix a point $z_0 \in \pr^m$. Firstly for any $\xi \in i\mfk$ we can take the limit $$z' = \lim_{t \to \infty} \exp(-it\xi)(z_0);$$ this limit exists by \cite[Lemma 5.4]{moment-weight}. We then define  the \emph{weight} by $$w_{\mu}(z_0,\xi) = -\langle \mu(z'), \xi\rangle.$$ Note the expression  has the opposite sign of that given \cite{moment-weight}, since we have the opposite sign convention for the definition of a moment map. Thus our notion of weight agrees with that of \cite{moment-weight}.

The starting point is the following result.


\begin{theorem}\cite[Theorem 10.4]{moment-weight}\label{lowerbound}
Suppose $\mu(z_{\infty}) \neq 0$, and set $m = |\mu(z_{\infty})|.$ Then there is a $\xi \in \mfk$ with $$w_{\mu}(z_0,\xi) = -m^2, \quad |\xi| = m.$$ Furthermore, $$-m = \frac{w_{\mu}(z_0,\xi)}{|\xi|} = \inf_{0 \neq \zeta \in \mfk} \frac{w_{\mu}(z_0, \zeta)}{|\zeta|}.$$ 
\end{theorem}

The idea of the proof originates with Chen--Sun \cite[Section 3]{CS}, and proves that the moment map flow produces a ``most destabilising element of $\mfk$''; the element $\xi$ is constructed to be asymptotically parallel to the moment map flow, in a suitable sense, and in particular $\xi$ is actually conjugate to $\mu(z_{\infty})$ in $\mfk$. So, either the flow converges to a zero of the moment map \emph{inside $U$}, or there is a ``destabilising'' $\xi \in \mfk$. In general, the difficulty is that $$z' = \lim_{t \to \infty} \exp(-it\xi)(z_0) $$ may not equal $z_{\infty}$, and in particular we must show that  $z' $ actually lies in $U$, as this is where we will have control through stability in applications. We will ultimately show that $z_{\infty} \in \overline{G.z'}$, by appealing to the second Ness uniqueness theorem \cite[Theorem 6.5]{moment-weight}, and this will be sufficient in our general argument to a reduce to a problem involving only $U$.

For this, we must understand the geometry of $z'$ further; our running assumption is thus that $\mu(z_{\infty})$ (and hence $\xi$) is non-zero. Denote by $\mfk_{\xi}$ the commutator of $\xi$ in $\mfk$, so the elements of $\mfk$ commuting with $\xi$.  We begin with the following.

\begin{lemma}\label{adding} Let $\zeta \in \mfk_{\xi}$ be such that 
$$y = \lim_{t \to \infty} \exp(-it\zeta)(z').$$ Then for all $k\gg 0$ $$y = \lim_{t \to \infty} \exp(-it(\zeta+k\xi))(z_{0}).$$
\end{lemma}

\begin{proof} We first determine the limits $z'$ and $y$ in suitable homogeneous coordinates. As the two actions commute, we can diagonalise both actions simultaneously, and so we can choose homogeneous coordinates  $[x_0, \ldots, x_n]$ so that for a general $z\in\pr^m$
\begin{align*}
\exp(-t i \xi)(x) =& [e^{-t \xi_0} x_0, \ldots, e^{-t\xi_n} x_n] \\
\exp(-t i \zeta)(x) =& [e^{-t \zeta_0} x_0, \ldots, e^{-t\zeta_n} x_n],
\end{align*}
for some real numbers $\xi_0, \ldots, \xi_n$ and $\zeta_0, \ldots, \zeta_n$. Moreover, we can assume that the coordinates of $z_0$ satisfy that their $j$\textsuperscript{th}-component $x_j = 0$ vanishes if and only if $j < d$. The coordinates are homogeneous and so replacing $\xi_j$ by $\xi_j + c$ for a constant $c$ independent of $j$ leaves the action unchanged. We can therefore assume that  $\xi_j \geq 0$ for all $j \geq d$, that the $\xi_j$ are non-decreasing for $j \geq d$, and that $\xi_d =0$. Thus there is an $r \geq 0$ such that $\xi_{d+r} = 0$ and $\xi_{d+r+1} > 0$. In these coordinates, the limit $z'$ is then given by 
$$
z' = [0, \ldots, 0, x_d , \ldots, x_{d+r}, 0, \ldots, 0].
$$

As with $\xi$, we can rescale the $\zeta_j$ and reorder the coordinates so that $\zeta_d = 0$ and   $\zeta_j \leq \zeta_{j+1}$ for all $j=d, \ldots, d+r-1$. Thus there is $q \leq r$ such that $\zeta_q = 0$ and, if $q<r$, $\zeta_{q+1}>0$. The limit $y$ of $z'$ under the action induced by $\zeta$ is then 
$$
y = [0, \ldots, 0, x_d , \ldots, x_{d+q}, 0, \ldots, 0].
$$

For all $k\gg 0$, we then have that $k\xi_{d+q+j}+ \zeta_{d+q+j}>0$ for all $j>0$, since $\zeta_{d+q+j}>0$ and $\xi_{d+q+j} =0$ if $q+j\leq r$, and $\xi_{d+q+j}> 0$ if $q+j>r$. Moreover, $k\xi_{d+j} + \zeta_{d+j} = 0$ for each $j =0, \ldots, q$. Thus for all sufficiently large $k$, the limit of $z_0$ under the action of $\exp(-i(k\xi+\zeta))$ is $y$, as claimed. \end{proof}

The point of this result is that it allows us to prove the following. An analogue of part of  this result for K-stability of polarised varieties is proven as \cite[Theorem 1.1]{optimal}.

\begin{proposition}

We have $$\inf_{0 \neq \zeta \in \mfk_{\xi}} \frac{w_{\mu}(z', \zeta)}{|\zeta|} = \inf_{0 \neq \zeta \in \mfk} \frac{w_{\mu}(z_0, \zeta)}{|\zeta|}.$$

\end{proposition}

\begin{proof}

It is obvious that 
$$\inf_{0 \neq \zeta \in \mfk_{\xi}} \frac{w_{\mu}(z', \zeta)}{|\zeta|} \leq \inf_{0 \neq \zeta \in \mfk} \frac{w_{\mu}(z_0, \zeta)}{|\zeta|};$$ 
the infimum of the right hand side is computed by $\xi$ itself, and the limit of $z_0$ under $\exp(-it\xi)$ is by definition $z'$. Thus 
$$\inf_{0 \neq \zeta \in \mfk} \frac{w_{\mu}(z_0, \zeta)}{|\zeta|} = - \frac{\langle \mu(z'), \xi\rangle}{|\xi|} =  \frac{w_{\mu}(z', \xi)}{|\xi|}.$$ Note that this implies \begin{equation}\label{negativeeq}\langle \mu(z'), \xi\rangle>0.\end{equation}

For the reverse direction, we begin by proving that for all $ \zeta \in \mfk_{\xi}$ we have \begin{equation}\label{relsemistable}w_{\mu}(z', \zeta) \geq - \langle \mu(z'), \xi\rangle\frac{\langle\xi, \zeta\rangle}{|\xi|^2};\end{equation} although we will not need the general theory, this means that $z'$ is ``GIT semistable relative to $\xi$'' \cite[Chapter 13]{moment-weight}. Fix such a $\zeta$ and let $$y = \lim_{t\to\infty} \exp(-it\zeta)(z'),$$ so that by  \cref{adding}  for $k \gg 0$ $$y = \lim_{t\to\infty} \exp(-it(\zeta + k\xi))(z_0).$$ The inequality of Equation \eqref{relsemistable} is invariant under adding a multiple of $\xi$ to $\zeta$ by a direct calculation, so we may assume that $$\langle \zeta, \xi\rangle = 0;$$ that is, we may assume $\zeta$ is orthogonal to $\xi$. 

Thus we must show that $w_{\mu}(z', \zeta) \geq 0$, when  $\zeta$ is orthogonal to $\xi$. We follow the proof of  \cite[Theorem 1.1]{optimal}. Let $y$ be the specialisation of $z'$ under $\zeta$.  Suppose for contradiction that $w_{\mu}(z', \zeta)<-\epsilon$ for some $\epsilon>0$. It follows from Lemma \ref{adding} that for any $k \gg 0$ 
$$\frac{w_{\mu}(z_0, \zeta+k\xi)}{|\zeta+k\xi|} = -\frac{\langle \mu(y), \zeta\rangle + k\langle \mu(y), \xi\rangle  }{|\zeta+k\xi|} < \frac{-\epsilon - k\langle \mu(y), \xi\rangle  }{|\zeta+k\xi|}.$$ 

Note that $\langle \mu(z'), \xi\rangle = \langle \mu(y), \xi\rangle.$ Indeed,  conjugation invariance (e.g. \cite[Proposition 6]{XW}) gives 
 $$\langle \mu(\exp(-it\zeta)(y)), \exp(-it\zeta) \cdot \xi\rangle = \langle \mu(y), \xi\rangle$$ 
 for all $t \geq 0$, where $\exp(-it\zeta) \cdot \xi$ denotes the adjoint action, and $\exp(-it\zeta) \cdot \xi = \xi$ since $\zeta \in \mfk_{\xi}$. Thus $\langle \mu(\exp(-it\zeta)(y)), \xi\rangle = \langle \mu(y), \xi\rangle$ for all $t$, and by continuity, taking the limit as $t\to \infty$ gives $\langle \mu(z'), \xi\rangle = \langle \mu(y), \xi\rangle ,$ as claimed. In particular, $\langle \mu(y), \xi\rangle >0$.

Since $\langle \mu(y), \xi\rangle >0$,  we can scale $\xi$ so that $|\xi|^2 = \langle \mu(y), \xi\rangle$. Then 
$$-|\xi| = -\frac{\langle \mu(y), \xi\rangle}{|\xi|} \leq \frac{w_{\mu}(z_0, \zeta+k\xi)}{|\zeta+k\xi|}< \frac{-\epsilon - k\langle \mu(y), \xi\rangle  }{|\zeta+k\xi|} =  \frac{-\epsilon - k|\xi|^2  }{|\zeta+k\xi|},$$ 
where we use  \cref{lowerbound} and the fact that $\langle \mu(y), \xi\rangle = -w_{\mu}(z_0,\xi)$ to obtain the first inequality; it follows that $$-|\xi| |\zeta+k\xi| < -\epsilon - k|\xi|^2.$$ Squaring both sides implies $$|\zeta+k\xi|^2 > k^2 |\xi|^2+2k\epsilon + \frac{\epsilon^2}{|\xi|^2}.$$ Since $\langle \zeta, \xi\rangle = 0$, it follows that $|\zeta+k\xi|^2 = k^2|\xi|^2 + |\zeta|^2,$ so $$|\zeta|^2 \geq 2k\epsilon |\xi|^2 + \frac{\epsilon^2}{|\xi|^2},$$ which is a contradiction for $k \gg 0$. Thus Equation \eqref{relsemistable} holds for all $\zeta \in \mfk_{\xi}$.

Now take a general $0 \neq \zeta \in \mfk_{\xi}$; we claim 
\begin{equation}\label{more-negative-eq} \frac{w_{\mu}(z', \zeta)  }{|\zeta|} \geq -\frac{ \langle\mu(z'), \xi\rangle}{|\xi|}.\end{equation} 
This will prove the result, since the right hand side satisfies  
$$-\frac{ \langle\mu(z'), \xi\rangle}{|\xi|} = \frac{w_{\mu}(z_0, \xi)}{|\xi|} = \inf_{0 \neq \zeta \in \mfk} \frac{w_{\mu}(z_0, \zeta)}{|\zeta|} .$$ 
To establish Equation \eqref{more-negative-eq}, we may assume $w_{\mu}(z', \zeta)<0$, since otherwise there is nothing to prove as  $ \langle \mu(z'), \xi\rangle>0$ by Equation \eqref{negativeeq}. Cauchy--Schwarz implies $$\langle\zeta,\xi\rangle  \leq |\xi|.|\zeta|,$$ while comparing signs of both sides of the inequality of Equation \eqref{relsemistable} further implies that $\langle\zeta,\xi\rangle \geq 0$. Thus  
$$w_{\mu}(z', \zeta)  \geq  -\langle \mu(z'), \xi\rangle\frac{\langle \xi, \zeta\rangle}{|\xi|^2} \geq -\langle\mu(z'), \xi\rangle\frac{ |\xi|.|\zeta|}{|\xi|^2} = -\langle \mu(z'), \xi\rangle \frac{ |\zeta|}{|\xi|}, $$implying in turn$$ \frac{w_{\mu}(z', \zeta)  }{|\zeta|} \geq -\frac{ \langle\mu(z'), \xi\rangle}{|\xi|}$$ and completing the proof. \end{proof}

This is the key result that allows us to understand the geometry of $z'$.

\begin{proposition}

Suppose $\mu(z_{\infty}) \neq 0$. Then there is an element $\xi' \in \mfk$ and an element $g' \in G$ such that $\exp(-it\xi'.g'(z_0))$ lies in $U$ for all $t\geq 0$ and the limit $$\lim_{t\to\infty} \exp(-it\xi')(g'(z_0)) = \tilde z \in U$$ satisfies $$w_{\mu}(g'(z_0), \xi')  = -\langle \mu(  z'),  \xi\rangle < 0.$$ 

\end{proposition}

\begin{proof}

We firstly claim that $z_{\infty} \in \overline{G.z'}$; as $z_{\infty} \in U$ and $U$ is open, this will allow us to replace $z'$ with another point in its $G$-orbit inside $U$. It follows from \cite[Theorem 10.4]{moment-weight} that  $$\inf_{0 \neq \zeta \in \mfk} \frac{w_{\mu}(z_0, \zeta)}{|\zeta|} = |\mu(z_{\infty})|,$$ and by the same result $$\inf_{0 \neq \zeta \in \mfk} \frac{w_{\mu}(z', \zeta)}{|\zeta|} = |\mu(z_{\infty}')|,$$ where $z_{\infty}'$ is the limit  as $t\to \infty$ of the moment map flow starting from $z'$, namely the flow
\begin{align*}\dot z(t), =& Jv_{\mu(z(t))} \\
z(0) =& z'.
\end{align*}  
This infinimum can actually computed by an element commuting with any maximal torus of automorphisms of $z'$, so $$\inf_{0 \neq \zeta \in \mfk_{\xi}} \frac{w_{\mu}(z', \zeta)}{|\zeta|} = \inf_{0 \neq \zeta \in \mfk} \frac{w_{\mu}(z', \zeta)}{|\zeta|},$$ by \cite[Theorem 13.3]{moment-weight}: the infimum is computed by an element of $\mfk$ asymptotic to the moment map flow, and the essential point is that the flow preserves symmetry \cite[Lemma 13.14]{moment-weight}. Thus it follows that $$ |\mu(z_{\infty})| =  |\mu(z_{\infty}')| = \inf_{g\in G} |\mu(g(z_0))|,$$ where the final equality follows again from \cite[Theorem 10.4]{moment-weight}. 

We next claim that $z_{\infty}' \in \overline{G.z_0}$. Indeed, certainly $z_{\infty}' \in \overline{G.z'}$ since $z_{\infty}' $ is the limit along the moment map flow, and the flow for all finite time remains in $G.z'$. Thus there is a sequence $h_i \in G$ such that $$\lim_{i \to \infty} h_i(z') = z_{\infty}'.$$ On the other hand, take a sequence $g_t\in G$ with $$\lim_{t \to \infty} g_t(z_0) = z'.$$ Then $$h_i(z') =h_i\left( \lim_{t \to \infty} g_t(z_0)\right) = \lim_{t \to \infty} h_i(g_t(z_0)),$$ since $h_i$ is continuous. Thus $h_i(z') \in \overline{G.z_0}$ and since $\overline{G.z_0}$  is closed it follows that also $$ z_{\infty}' = \lim_{i\to \infty} h_i(z') \in \overline{G.z_0}.$$

Thus we have two points $z_{\infty}' $ and $z_{\infty}$ which satisfy $z_{\infty}, z_{\infty}' \in \overline{G.z_0}$ and $$ |\mu(z_{\infty})| =  |\mu(z_{\infty}')| = \inf_{g\in G} |\mu(g(z_0))|.$$ It then follows from the second Ness uniqueness theorem  \cite[Theorem 6.5]{moment-weight} that $$ z_{\infty}' \in K.z_{\infty}.$$ Since  $U$ is $K$-invariant, it follows that $ z_{\infty}' \in U$.

Since $ z_{\infty}' \in U$ and $z_{\infty}' \in \overline{G.z'}$, we can choose  a  $g\in G$ with $g(z') \in U$ since $U$ is open. Then $$\lim_{t \to \infty} \exp(-it g\cdot \xi)(g(x_0)) = g(z') \in U.$$ It may not be the case that $ \exp(-it g\cdot \xi)(g(x_0))$ lies in $U$ for all $t\geq 0$, but since their limit as $t \to \infty$, namely $g(z')$, does lie $U$, we may replace $g(z_0)$ with some $g'(z_0)$ with $g' = \exp(-it g\cdot \xi) \circ g $ for $t\gg 0$ so that $\exp(-itg\cdot \xi)(g'(z_0))$ does lie in $U$ for all $t\geq 0$. With this choice we also still have $$w_{\mu}(g'(z_0), g\cdot\xi)> 0,$$ through the conjugation invariance property (e.g. \cite[Proposition 6]{XW}) $$w_{\mu}(g(z_0), g\cdot\xi) = w_{\mu}(z_0, \xi)$$ and the equality $w_{\mu}(g(z_0), g\cdot\xi) = w_{\mu}(g'(z_0), g\cdot\xi)$ (the latter holding since they both equal $\langle \mu(g(z')), g\cdot \xi\rangle = \langle \mu(z'), \xi\rangle$). This completes the proof by setting $\xi' = g\cdot \xi$. \end{proof}

\begin{remark}
Although we will not need the following strengthening, under the same hypotheses there is a $ \xi''$ which satisfies, in the notation of the proof $$\lim_{t\to\infty} \exp(-it\xi'')(z_0) = g(z')$$ with $\xi''$ \emph{rational} (in the sense that it generates a $\C^*$-action) and still satisfies $w_{\mu}(g(z'), \xi'')<0$ (analogously to \cite[Proof of Proposition 3.5]{CSW}; the proof follows by perturbing $\xi'$). We also note that Chen--Sun--Wang give an approach to the asymptotics to gradient flows satisfying general hypotheses they impose, related to some of the results we have proven here \cite{CSW}. In particular, under their hypotheses, one also obtains $z_{\infty} \in \overline{G.z_0}$, and another approach to the results presented here would be to show that their hypotheses apply to our (more classical) situation.
\end{remark}

As we expect these results to be useful for other problems, we summarise what has been proven here along with  \cref{sec:choosecplxstr}, which is essentially a local version of the Kempf--Ness theorem for families of moment maps. In what follows, $ B_{C\sqrt{\epsilon}}$ denotes the ball of radius $C\sqrt{\epsilon}$ around the origin.

\begin{corollary}\label{end-result}Suppose $U$ is an open ball around the origin in a vector space $V$ given a linear $K$-action, and let $\omega_{\epsilon}$ be a family  of $K$-invariant K\"ahler metrics on $U$ such that $|\omega_{\epsilon} - \omega_0| = O(\epsilon)$. Fix $z_0 \in U$ with finite stabiliser such that $0 \in \overline{G.z_0}$, where $G$ is the complexification of $K$. Let $\mu_{\epsilon}$ be a family of moment maps for the $K$-action on $(U,\omega_{\epsilon})$ with $\mu_0(0)=0$. Then there is a $C>0$ such that for all $0 < \epsilon \ll 1$ either there is a point $z_{\epsilon,\infty} \in G.z_0 \cap B_{C\sqrt{\epsilon}}$ with $\mu_{\epsilon}(z_{\epsilon,\infty}) = 0$ or there is a $g\in G$ and $0 \neq \xi_{\epsilon} \in \mfk$ such that $$w_{\mu_{\epsilon}}(g(z_0), \xi_{\epsilon})\leq 0.$$ Furthermore, for all $t \geq 0$ each of $\exp(-it\xi_{\epsilon})(g(z_0))$ and $$\lim_{t\to \infty} \exp(-it\xi_{\epsilon})(g(z_0))$$  lie in $U$. Finally, a $\xi_{\epsilon}$ with this property can be taken to be rational. 
\end{corollary}

\subsubsection{Returning to the space of unitary connections}\label{global-moment-map-geometry} We return to the setting of \cref{sec:choosecplxstr} and hence to a \emph{sequence} of moment maps $\mu_{\oldepsilon}$ on $U$. We will show  that asymptotic $Z$-stability ensures that $z_{\oldepsilon,\infty} \in G.z_0$, meaning we have constructed zeroes of the moment map $\mu_{\oldepsilon}$  in the orbit $G.z_0$ itself. \cref{end-result} implies that for $0<\epsilon \ll 1$ either there is a zero of the moment map $\mu_{\epsilon}$ in $G.z_0$ itself, or there is a $0\neq \xi_{\epsilon} \in \mfk$ with limit which we denote (replacing $z_0$ with another point in its $G$-orbit inside $U$ as in \cref{sec:choosecplxstr}) $$y_{\epsilon} = \lim_{t\to\infty}\exp(-it\xi_{\epsilon})(z_0)$$ with $\langle \mu_{\epsilon} (y_{\epsilon}), \xi_{\epsilon}\rangle \leq 0$. Note that $\xi_{\epsilon} \in \mfk_{y_{\epsilon}}$ and also that by definition of the moment map $\mu_{\epsilon}$ (and crucially using that $y_{\epsilon} \in U$) $$\langle \mu_{\oldepsilon}( y_{\oldepsilon}), \xi_{\oldepsilon}\rangle  = \int_X \left\langle \Ima(e^{-i\phi_{\oldepsilon}(E)} \tilde Z_{\oldepsilon}(E,\Phi( y_{\oldepsilon}))), \xi_{\oldepsilon}\right\rangle,$$ where $\Ima(e^{-i\phi_{\oldepsilon}(E)} \tilde Z_{\oldepsilon}(E,\Phi(y_{\oldepsilon})))$ is an $\End(E)$-valued $(n,n)$-form, so that the pairing means the trace product on the $\End(E)$ component, ensuring the integrand is an $(n,n)$-form. Here we have included $E$ in the notation of the $Z$-critical operator $\tilde Z_{\oldepsilon}$ as we will now incorporate more than one bundle in our discussion.

We begin by showing that $\xi_{\oldepsilon}$ induces a filtration of $E$, and subsequently show that the numerical invariants associated to the filtrations of the sort thus produced may be related to asymptotic $Z$-stability. Let $E'$ denote the underlying smooth bundle of $E$ given the holomorphic structure $\db_{E'}$ induced from the $(0,1)$-part of connection $\Phi ( y_{\oldepsilon})$. This connection is automatically integrable, as integrability is a closed condition. Note that $\xi_{\oldepsilon}$ acts on $E'$ holomorphically, so that $\xi_{\oldepsilon}$ is a holomorphic section of $\End E'$. 
\begin{lemma}
For each $\lambda \in \R$, 
$$
V_{\lambda} = \ker(i\xi_{\oldepsilon}- \lambda \Id_E )
$$
is a holomorphic subbundle of $E'$.
\end{lemma}
\begin{proof}
The smooth section $i\xi_{\oldepsilon} - \lambda \Id_E$ of the underlying smooth vector bundle of $\End E'$ defines a holomorphic section of $\End E'$ since, as remarked above, $\xi_{\oldepsilon}$ is a holomorphic section of $\End E'$ and the section $\Id_E = \Id_{E'}$ is always holomorphic. Thus its kernel is a holomorphic subsheaf of $E'$. Since $\xi_{\oldepsilon}$ also defines a holomorphic section of the graded object of $E$, which is assumed to be locally free, the kernel is in fact a subbundle.
\end{proof}

This produces a natural splitting of $E'$.

\begin{lemma}
$E'$ splits holomorphically as a direct sum of the non-trivial $V_{\lambda}$. There are also at least two such components provided $E'$ is not isomorphic to $E$.
\end{lemma}
\begin{proof}
When viewed as an action of $G$ on connections on $E$, $K$ acts via unitary endomorphism of $E$. Thus $i\xi_{\oldepsilon}$, viewed as an endomorphism of $E$, is Hermitian, in the sense that it satisfies $(i\xi_{\oldepsilon})^* = i\xi_{\oldepsilon}$. In particular, it is diagonalisable on each fibre. Hence each fibre is a union of the fibres of the $V_{\lambda}$, proving the first part of the statement. The remaining part follows because the only endomorphism of $E$ with kernel the entirety of $E$ is the zero section---thus if there were only one component $V_{\lambda}$, $\xi_{\oldepsilon}$ would be a constant multiple of the identity section $\Id_E$, but $\Id_E$ fixes the holomorphic structure of $E$, and $E'$ is not isomorphic to $E$ by assumption. Thus there must be at least two components in the splitting by eigenbundles of $\xi_{\oldepsilon}$.
\end{proof}

The above demonstrates that we obtain a splitting of $E'$ induced by the vector field $\xi_{\oldepsilon}$. We now further show that we  obtain a filtration of $E$ from this data. Write $\db_{E} = \db_{E'} + \gamma$ for some $\gamma \in \Omega^{0,1}(\End E)$. Let $\lambda_1 > \ldots > \lambda_l$ be the values of $\lambda$ such that $V_{\lambda} \neq 0$.  We decompose $\gamma$ as $\gamma = \sum_{i,j} \gamma_{ij}$ with
$$
\gamma_{ij} \in \Omega^{0,1} \left( V_{\lambda_j}^* \otimes V_{\lambda_i}\right).
$$
\begin{lemma}
The components $\gamma_{ij}$ of $\gamma$ satisfy  $\gamma_{ij} = 0$ for each $i \geq j$.
\end{lemma}
\begin{proof}
Write 
$$
\exp(ti\xi_{\oldepsilon}) \cdot \db_E = \db_{E'} + \gamma_t.
$$
As $v=i\xi_{\oldepsilon}$ acts on $V_{\lambda}$ by multiplication by $\lambda$ and $\gamma_{ij} \in \Omega^{0,1} \left( V_{\lambda_j}^* \otimes V_{\lambda_i}\right),$ it follows that
\begin{align*}
(\gamma_t)_{ij} =& \exp(-tv) \circ \gamma_{ij} \circ \exp(tv) \\
=& e^{(\lambda_j - \lambda_i )t} \gamma_{ij}.
\end{align*}
Moreover, $\db_{E'}$ is the limit of $\exp(t\xi_{\oldepsilon}) \cdot \db_E$ as $t \to \infty$. Thus $\gamma_t \to 0$ as $t \to \infty$. From the above formula, we see that $\gamma_t$ converges to zero if and only if for each $i,j$, either $\lambda_j < \lambda_i$, or $\gamma_{ij} = 0$. Since the $\lambda_i$ are ordered in such a way as to be strictly decreasing, this implies that $\gamma_{ij} =0$ for each $i \geq j$, as desired.
\end{proof}

The above implies that we obtain a filtration of $E$ from the data of the degeneration of $E$ to $E'$.
\begin{corollary}
\label{cor:degfiltration}
Let $S_i = \bigoplus_{j \leq i} V_{\lambda_j}$. Then each $S_i$ is a holomorphic subbundle of $E$ and in particular
$$
0 \subset S_1 \subset \ldots \subset S_l = E
$$
is a filtration of $E$ by holomorphic subbundles.
\end{corollary}

We next explain how to relate this filtration to asymptotic $Z$-stability.

\begin{lemma}
\label{lem:moment-map-filtration} The pairing of the moment map $\mu_{\epsilon}$ with $\xi_{\oldepsilon}$ is given at $ y_{\oldepsilon}$  by
\begin{align*}
\langle \mu_{\oldepsilon}( y_{\oldepsilon}), \xi_{\oldepsilon}\rangle =& (\lambda_{l} - \lambda_{l-1})\Ima\left(\frac{Z_{\oldepsilon}(S_{l-1})}{Z_{\oldepsilon}(E)}\right) +  (\lambda_{l-1} - \lambda_{l-2})\Ima\left(\frac{Z_{\oldepsilon}(S_{d-2})}{Z_{\oldepsilon}(E)}\right)\\
&+ \ldots +(\lambda_{2} - \lambda_{1})\Ima\left(\frac{Z_{\oldepsilon}(S_{1})}{Z_{\oldepsilon}(E)}\right).
\end{align*} 

\end{lemma}

\begin{proof}

We prove the statement for an arbitrary $\oldepsilon$, considered  fixed. The $Z_{\oldepsilon}$-critical operator is clearly linear in $\xi_{\oldepsilon}$, so if we write $\xi_j \in \scA^0(\End E)$ for the diagonal matrix with entries equal to one on the diagonal entries corresponding to $V_{\lambda_j}$, then 
$$
\langle \mu_{\oldepsilon}( y_{\oldepsilon}), \xi_{\oldepsilon} \rangle = \sum_{j=1}^{l} \lambda_j \langle \mu_{\oldepsilon}( y_{\oldepsilon}), \xi_j\rangle,
$$ 
where by definition $$ \langle \mu_{\oldepsilon}(y_{\oldepsilon}), \xi_j\rangle = \int_X \langle \xi_j, \Ima(e^{-i\phi_{\oldepsilon}(E)} \tilde Z_{\oldepsilon}(E,\Phi(y_{\oldepsilon})))\rangle.$$ 

Since the connection $\Phi(y_{\oldepsilon})$ induces the reducible holomorphic structure $$E'  = V_{\lambda_1} \oplus V_{\lambda_2}  \oplus  \ldots \oplus V_{\lambda_l},$$ the curvature $F_{\Phi(y_{\oldepsilon})}$ also splits as a direct sum of the curvatures of $\Phi(y_{\oldepsilon})$ restricted to each subbundle. Since $\xi_j$ has non-zero entries only on the diagonal associated to $V_{\lambda_j}$, taking trace against $\xi_j$ corresponds to replacing the term $\Ima \left(e^{-i\phi_{\oldepsilon}(E)} \tilde Z_{\oldepsilon}(E,\Phi(y_{\oldepsilon}))\right)$ with $\Ima(e^{-i\phi_{\oldepsilon}(E)} \tilde Z_{\oldepsilon}(V_{\lambda_j}),\Phi(y_{\oldepsilon}))$. In particular 
\begin{align*} 
\langle \mu_{\oldepsilon}(y_{\oldepsilon}), \xi_j\rangle &= \int_X \langle \xi_j, \Ima(e^{-i\phi_{\oldepsilon}(E)} \tilde Z_{\oldepsilon}(E,\Phi(y_{\oldepsilon})))\rangle \\ 
&= \int_X \tr \Ima(e^{-i\phi_{\oldepsilon}(E)} \tilde Z_{\oldepsilon}(V_{\lambda_j},\Phi(y_{\oldepsilon})) \\ 
&= \Ima(e^{-i\phi_{\oldepsilon}(E)}  Z_{\oldepsilon}(V_{\lambda_j}))) \\ 
&= \Ima\left(\frac{Z_{\oldepsilon}(V_{\lambda_j})}{Z_{\oldepsilon}(E)}\right).\end{align*}

It follows from this that 
$$
\langle \mu_{\oldepsilon}(y_{\oldepsilon}), \xi_{\oldepsilon} \rangle =  \sum_{j=1}^{l} \lambda_j \Ima\left(\frac{Z_{\oldepsilon}(V_{\lambda_j})}{Z_{\oldepsilon}(E)}\right).
$$ 
We next use additivity of the central charge in short exact sequences, namely that $$Z_{\oldepsilon}(S_j) = Z_{\oldepsilon}(S_{j-1}) + Z_{\oldepsilon}(V_{\lambda_j})$$ which holds due to the existence of the short exact sequence $$ 0 \to S_{j-1} \to S_j \to V_{\lambda_{j}} \to 0$$ to replace each $Z_{\oldepsilon}(V_{\lambda_j})$ with $Z_{\oldepsilon}(S_j) - Z_{\oldepsilon}(S_{j-1})$. Through this process one calculates 
\begin{align*} 
\sum_{j=1}^d \lambda_j \Ima\left(\frac{Z(V_{\lambda_j})}{Z(E)}\right) =& (\lambda_{l-1} - \lambda_{l})\Ima\left(\frac{Z_{\oldepsilon}(S_{l-1})}{Z_{\oldepsilon}(E)}\right) +  (\lambda_{l-2} - \lambda_{l-1})\Ima\left(\frac{Z_{\oldepsilon}(S_{d-2})}{Z_{\oldepsilon}(E)}\right) \\
&+ \ldots +(\lambda_{1} - \lambda_{2})\Ima\left(\frac{Z_{\oldepsilon}(S_{1})}{Z_{\oldepsilon}(E)}\right),
\end{align*} proving the result.
 \end{proof}

Combining the results of this section, we  summarise:
\begin{corollary}
\label{prop:solfindim}
For each $0<\oldepsilon \ll 1$, the moment map flow converges to $z_{\oldepsilon, \infty} \in G.z_0$ such that the corresponding integrable unitary connections $\Phi(z_{\oldepsilon,\infty})$ satisfy
$$p(\Ima(e^{-i\phi_{\oldepsilon}}\tilde Z_{\oldepsilon}(\Phi(z_{\oldepsilon,\infty}))) = 0.$$
Moreover, the $z_{\oldepsilon, \infty}$ lie in a ball of radius $C \sqrt{\oldepsilon}$, for some $C>0$ independent of $\oldepsilon$.
\end{corollary}
\begin{proof}
 \cref{flow-bounded-region} and \cref{cor:flowconvergence} imply that the moment map flow converges to points $z_{\oldepsilon, \infty} \in \overline{G.z_0}$ which lie in the ball of radius $C \sqrt{\oldepsilon}$ around the origin in $U$, for some $C>0$ independent of $\oldepsilon$. The statement of the result we wish to prove is that $z_{\oldepsilon, \infty} \in G.z_0$ is actually a zero of the moment map for each $\oldepsilon$ sufficiently small. 

By  \cref{end-result}, if $z_{\oldepsilon, \infty} \notin G.z_0$, there is an element $\xi_{\oldepsilon,\infty}$ such that $$\lim_{t\to\infty} \exp(-it\xi_{\oldepsilon,\infty})(z_0) =  y_{\oldepsilon} \in U,$$ where $ y_{\oldepsilon}$ satisfies \begin{equation}\label{eq:contra}\langle \mu(y_{\oldepsilon}),  \xi_{\oldepsilon}\rangle \geq 0.\end{equation} 

\cref{cor:degfiltration} associates a filtration of $E$ to $\xi_{\oldepsilon}$ in such a way that by \cref{lem:moment-map-filtration}\begin{align*}
\langle \mu_{\oldepsilon}(y_{\oldepsilon}), \xi_{\oldepsilon}\rangle =& (\lambda_{l-1} - \lambda_{l})\Ima\left(\frac{Z_{\oldepsilon}(S_{l-1})}{Z_{\oldepsilon}(E)}\right) +  (\lambda_{l-2} - \lambda_{l-1})\Ima\left(\frac{Z_{\oldepsilon}(S_{l-2})}{Z_{\oldepsilon}(E)}\right)\\
&+ \ldots +(\lambda_{1} - \lambda_{2})\Ima\left(\frac{Z_{\oldepsilon}(S_{1})}{Z_{\oldepsilon}(E)}\right).
\end{align*} Since the $\lambda_j$ are constructed to be strictly decreasing, asymptotic $Z$-stability implies $\Ima\left(\frac{Z_{\oldepsilon}(S_{j})}{Z_{\oldepsilon}(E)}\right)<0$ for each $j$ provided $S_j$ is a non-trivial subbundle and $0<\epsilon\ll 1$ is sufficiently small by \cref{lem:equivalentstabilityconditions}, and in particular $$\langle \mu_{\oldepsilon}(y_{\oldepsilon}), \xi_{\oldepsilon}\rangle <0$$ which is a contradiction. Thus Equation \eqref{eq:contra} forces the filtration to be trivial, meaning that $z_{\oldepsilon, \infty} \in G.z_0$, proving the result. \end{proof}

\subsection{The general case: perturbation to a solution}

\subsubsection{Deformations of complex structures}\label{sect:cplxdeformations}

We next collect various technical results about variation of complex structure and automorphism groups of holomorphic vector bundles. These results will be used to construct approximate solutions to the $Z$-critical equation. We thus return to an asymptotically $Z$-stable, sufficiently smooth holomorphic vector bundle $E$.  Thus $E$ is slope semistable and, recalling the notation from \cref{sec:stabconds}, has a Jordan--H\"older filtration 
$$0 = S_0 \subset S_1 \subset \hdots \subset S_l = E$$
such that each $Q_i = S_i/S_{i-1}$ is slope stable with $\mu(Q_i) = \mu(E).$ The graded object 
$$
\Gr (E) = \oplus_{i=1}^l Q_i,
$$ 
which is smooth. 

As complex vector bundles, $E$ and $\Gr (E)$ are equal, but their holomorphic structures differ. If we let $\db_E$ and $\db_0$ denote the holomorphic structures of $E$ and $\Gr(E)$, respectively, then they are related by 
$$
\db_E = \db_0 + \gamma
$$
for some $\gamma \in \Omega^{0,1} ( \End E).$ This is upper-triangular in matrix form, using the ordering of the $Q_i$ induced from an increasing filtration of $E$. In other words, we can write
$$
\gamma = \sum_{i<j} \gamma_{ij}
$$
where $\gamma_{ij} \in \Omega^{0,1} ( Q_j^* \otimes Q_i).$

The Maurer--Cartan equation resulting from the integrability requirement $\db_E^2 =0$ is 
$$
\db_0 \gamma + \gamma \wedge \gamma = 0,
$$
which is equivalent to 
$$
\db_0 \gamma_{ij} + \sum_{i < k <j} \gamma_{i k} \wedge \gamma_{k j} =0
$$
for all $i<j$. The complex gauge group $\G^{\C} = \Gamma(X,\GL(E,\C))$ acts on the space holomorphic structures via conjugation, similarly to Equation \eqref{gaugegroupaction}. After fixing a Hermitian metric, $\G^{\C}$ hence acts on the space of induced Chern connections. As each component $Q_i$ of $\Gr(E)$ is slope stable, it admits a Hermitian metric $h_i$ solving the weak Hermite--Einstein equation $$\Lambda_{\omega} F_{h_i} = if\Id_{Q_i},$$ with the function $f\in C^{\infty}(X,\R)$ defined by
$$f = -2\pi \left( \frac{\deg_U(E)}{(n-1)! [\omega]^n\rk(E)} - \contr_{\omega} \tilde U_2 \right),$$ which is the same function used in \cref{lemma:largevolume} and \cref{thm:stabilityimpliesexistencestable}. Here  we have used that the slope of $Q_i$ equals that of $E$ to ensure that both sides of the equation have the same integral. Thus we obtain a Hermitian metric on $\Gr(E)$, and hence on $E$.

We next give an explicit description of the Lie algebra of the automorphism group of $\Gr(E)$. To begin with, for each $c \in \C$, the endomorphism $c\Id_{Q_i}$ lies in $\aut(\Gr(E))$.  If some of the $Q_i$ are biholomorphic, however, there are additional symmetries of the graded object.  Any isomorphic $Q_i$ and $Q_j$ are gauge equivalent, hence after a gauge transformation we may assume they are equal. The following result, which is a general result about automorphisms of slope polystable vector bundles, completely characterises the Lie algebra in question.

\begin{lemma}\cite[Lemma 5.6]{sektnan2020hermitian}
\label{lem:automsgraded}
 Let $$\psi_i^j: \Gr(E) \to \Gr(E)$$ be the identity isomorphism between $Q_i$ and $Q_j$ if $Q_i$ and $Q_j$ are equal, and zero otherwise.  Then the Lie algebra $\mathfrak{aut} (\Gr (E) )$  is given by 
$$
\mathfrak{aut} ( \Gr (E) ) = \oplus_{i,j=1}^l \langle \psi_i^j \rangle.
$$
\end{lemma}

\subsubsection{Behaviour of the linearised operator under extensions} In this section, it will be convenient to introduce the parameter $\epsilon$ defined by $$\epsilon^2 = \oldepsilon = 1/k,$$ essentially doubling our powers. This simplifies the notation considerably---avoiding the appearance of square roots throughout---but has no other substantial impact. We then use $\epsilon$-dependent notation throughout, so that  $Z_{\epsilon}$ denotes the expansion of the central charge  in powers of $\epsilon$, and similarly for other quantities.

Our construction of $Z$-critical connections will employ a version of the inverse function theorem.  As such, it is crucial to the analysis to understand the properties of the linearisation $P_{\epsilon}$ of the $Z$-critical operator $D_{\epsilon}$ associated to the solutions of the finite dimensional problem in \cref{prop:solfindim}. The main result of this Section is \cref{prop:bound-on-linearised}, which provides a bound of the form $$\| P_{\epsilon} ( s) \| \geq C \epsilon^{2q} \| s \|,$$ where $C>0$ is constant, $\epsilon$ is taken sufficiently small and $q$ is the \emph{discrepancy order} defined in \cref{def:discreporder}; this is precisely the kind of bound that will allow us to apply the inverse function theorem, after forming appropriate approximate solutions to the $Z$-critical equation (which will be done in \cref{prop:generalinductive}). The strategy used here is similar to that used in \cite[Section 5.4.1]{sektnan2020hermitian}, where an analogous strategy is developed when the automorphism group $\Aut(\Gr(E))$ of the graded object $\Gr(E)$ of $E$ is a complex torus.

The results here apply to quite general $\epsilon$-dependent families of deformations from $E$ to $\Gr(E)$. We will therefore first prove the results for such deformations, and then show our results apply to connections $\Phi(z_{\epsilon,\infty})$ constructed in \cref{sec:choosecplxstr}, see \cref{prop:solnasymptotics} and \cref{rem:discrepancy}. 

We work with an $\epsilon$-dependent family of $\db_{\epsilon} = f_{\epsilon} \cdot \db_E$ of holomorphic structures biholomorphic to that of $E$. We thus have
\begin{align}
\label{eq:def}
\db_{\epsilon} = \db_{0} + \gamma_{\epsilon},
\end{align}
where $\gamma_{\epsilon} \in \Omega^{0,1} ( \End E)$ is upper-diagonal in terms of the decomposition of $\End E$ coming from the decomposition $\oplus_i Q_i$ of $E$ as a \emph{complex} vector bundle. We assume 
that each component of $\gamma_{\epsilon}$ is $O(\epsilon)$. This is automatic in our situation, as the $z_{\oldepsilon, \infty}$ in \cref{prop:solfindim} are in a ball of radius $C \epsilon$ for some $C$ that is independent of $\epsilon$. Note that if we let $$a_{\epsilon} = \gamma_{\epsilon} - \gamma_{\epsilon}^*,$$ the corresponding Chern connection takes the form 
$$
A_0 + a_{\epsilon},
$$
where $A_0$ is the Chern connection of $\db_0$, and the induced curvature is
$$
F_{\epsilon} = F_{\db_0} + d_{A_0} (a_{\epsilon}) + a_{\epsilon} \wedge a_{\epsilon}.
$$

We begin by considering properties of the Laplacian under extensions, in a general situation. We suppose $E$ is holomorphic vector bundle which is an extension of $S$ by $Q$ induced by an extension form $\gamma \in \Omega^{0,1} (Q^* \otimes S)$. Denote by $A_t$ the Chern connection with respect to the holomorphic structure $\bar\partial_t = \db_{S \oplus Q} + t \gamma$. This induces a connection $A_t^{\End E}$ on $\End E$ and a corresponding Laplacian operator
$$
\Lap_t = i\contr_{\omega} \left(\del_{A_t^{\End E}} \db_{A_t^{\End E}} - \db_{A_t^{\End E}} \del_{A_t^{\End E}}\right).$$ 
\begin{lemma}
\label{lemma:tlaplacian}
	The Laplacian with respect to $A_t$ on $\End E$ admits an expansion $$\Lap_t = \Lap_0 + tL_1 + t^2L_2,$$ where for $u$ a section of $\End E$ \begin{align*} L_1(u) &= i\contr_{\omega}\left( \del_0([\gamma,u]) -[\gamma^*, \db_0(u)] +\db_0([\gamma^*, u]) - [\gamma, \del_0(u)]\right), \\ L_2(u) &=   i\contr_{\omega} \left( [\gamma, [\gamma^*, u]] - [\gamma^*, [\gamma, u]] \right).\end{align*}
\end{lemma}
\begin{proof}
	Recall that on the endomorphism bundle we have
	$$A_t = A_0 + t[\gamma - \gamma^*,  \underline{\hspace{0.3cm}}]$$
	where $\gamma$ has type $(0,1)$. Then 
	\begin{align*}
	\Lap_t(u) &= i\contr_{\omega} (\del_t \db_t - \db_t \del_t) (u)\\
	&= i\contr_{\omega} \left((\del_0 -t[\gamma^*,  \underline{\hspace{0.3cm}}])(\db_0 + t[\gamma,  \underline{\hspace{0.3cm}}]) - (\db_0 + t[\gamma,  \underline{\hspace{0.3cm}}])(\del_0 -t[\gamma^*,  \underline{\hspace{0.3cm}}])\right)(u)\\
	&=\Lap_0(u) + ti\contr_{\omega}\left( \del_0([\gamma,u]) -[\gamma^*, \db_0(u)] +\db_0([\gamma^*, u]) - [\gamma, \del_0(u)]\right) \\
	&\, + t^2 i\contr_{\omega} \left( [\gamma, [\gamma^*, u]] - [\gamma^*, [\gamma, u]] \right),
	\end{align*} which proves the result. \end{proof}

We now apply this to the actual deformation of $E$ given in Equation \eqref{eq:def} that we consider. We let  $$P_{\varepsilon} : L^2_{d+2,0}(\End E) \to L^2_{d,0}(\End E)$$ be the linearisation of the $Z$-critical operator $D_{\varepsilon}$ at a connection $\db_{\epsilon}$ as above.

\begin{corollary}
\label{cor:linearisationfirstterm}
The linearised operator $P_{\epsilon}$ admits an expansion as $\epsilon\to 0$ of the form
$$
P_{\epsilon} = \Delta_{\Gr(E)} + O (\epsilon).
$$
\end{corollary}
\begin{proof}
By an identical proof to \cref{lemma:linearisation}, we see that $$P_{\epsilon} = \Delta_{\epsilon} + O(\epsilon).$$ The result then follows from \cref{lemma:tlaplacian}.\end{proof}

As $E$ is simple by assumption, the linearised operator $P_{\epsilon}$ is invertible modulo constants, despite the leading order operator $ \Delta_{\Gr(E)}$ having a larger kernel. The main goal of the remainder of this section is to give a bound on the operator norm of $P_{\epsilon}$ in terms of $\epsilon$---this will necessarily decay with $\epsilon$, due to the larger kernel of  $ \Delta_{\Gr(E)}$.

In understanding the operator norm of $P_{\epsilon}$, we will use following quantity associated to the family of extension forms $\gamma_{\epsilon}$.
\begin{definition}
\label{def:bdldiscreporder} Denote by $(\gamma_{\epsilon})_{ij}$ the $\epsilon$-dependent coefficients of the extension forms. We define an integer $q(i,j)$ for each entry as follows. If $i<j$, we define $q(i,j)$ to be the minimal order of $\epsilon$ for which a non-zero term appears in the expansion of $(\gamma_{\epsilon})_{ij}$ in powers of $\epsilon$. If  $j<i$ we define $q(i,j) = q(j,i).$ If $(\gamma_{\epsilon})_{ij} = 0$, we define $q(i,j) = +\infty$. We finally set  $q_i = \min_j q(i,j)$.
\end{definition}
 Rewriting the definition, the $q(i,j)$ are defined such that
$$
(\gamma_{\epsilon})_{ij} = \epsilon^{q(i,j)} \varsigma_{ij} + O(\epsilon^{q(i,j) +1})
$$
for some $\varsigma_{ij} \in \Omega^{0,1} (Q_j^* \otimes Q_i)$.
  Note that from the way we perform the deformation, $q(i,j)$ is an integer---this is the essential reason why we introduced the variable $\epsilon$. It follows from the definition that $q(i,j) \geq q_i$, and for all $i$ there exists a $j$ such that $q(i,j) = q_i$.
\begin{definition}
\label{def:discreporder} 
Let $\db_{\epsilon} = \db_0+\gamma_{\epsilon}$ be as above. Then $q = \max_i q_i$ is called the \emph{discrepancy order of the deformation} $\db_{\epsilon}$. 
\end{definition}
\begin{remark}
For the particular sequence $\gamma_{\epsilon}$ constructed by our solution of the finite dimensional problem in \cref{sec:choosecplxstr}, we will show in \cref{prop:solnasymptotics} that the discrepancy order of the deformation $\gamma_{\epsilon}$ has an algebro-geometric interpretation as the largest exponent such that there is a subbundle $F$ coming from a Jordan--H\"older filtration of $E$ whose $Z_{\epsilon}$-phase agrees with that of $E$ up to order $\epsilon^{2q}$. However, the results we present here apply to more general deformations.
\end{remark}

The following result will be used at several points in the remaining of this section.
\begin{lemma}\cite[Lemma 5.9]{sektnan2020hermitian}
\label{lem:extensioncontribution}
Let $\db_{\epsilon}$ be an $\epsilon$-dependent family of holomorphic structures as above. Then
\begin{enumerate}
\item $\tr ( a_{\epsilon} \wedge a_{\epsilon} ) = 0$;
\item  if $\gamma_{ij} = (\gamma_{\epsilon})_{ij} \neq 0$, then 
$$
\int_X \Lambda_{\omega} \left( \tr_{Q_j} (  i \gamma_{ij}^* \wedge \gamma_{ij} ) \right) \omega^n > 0 .
$$
\item and 
$$
\Lambda_{\omega} (d_0 a_{\epsilon}) = 0.
$$
\end{enumerate}
\end{lemma}

The following relates the discrepancy order to the mapping properties of the linearised operator. The statement involves an $L^2$-projection on sections of $\End E$ defined using the volume form $\omega^n$ and the Hermitian metric induced by $h$.

\begin{lemma}
\label{lemma:linearisationexpansion}
Let $\Pi_{\oplus_j \langle \Id_{Q_j} \rangle}$ denote the orthogonal projection onto $\oplus_j \langle \Id_{Q_j} \rangle$. For each $i$ and $j$, there exists constants $c_{i,j}>0$ such that 
\begin{align*}
\Pi_{\oplus_j \langle \Id_{Q_j} \rangle}  \left( P_{\epsilon} (\Id_{Q_i}) \right) =&  - \sum_{j} \left(c_{i,j} \epsilon^{2 q(i,j)} \left( \frac{1}{\rk (Q_i)} \Id_{Q_i}- \frac{1}{\rk (Q_j)} \Id_{Q_j}  \right) + O(\epsilon^{2q(i,j)+1})\right).
\end{align*}
Moreover, the $c_{i,j}$ are symmetric: $c_{i,j} = c_{j,i}$.
\end{lemma}

\begin{proof}
As in \cref{lemma:linearisation}, each term of the expansion of the linearisation will be a wedge product of three types of terms:
\begin{enumerate}
\item $\partial_{\epsilon} \db_{\epsilon} - \db_{\epsilon} \partial_{\epsilon}  $;
\item $\omega$;
\item $\tilde U$.
\end{enumerate}
The former arises from linearising the curvature operator, while the latter two terms are not bundle-valued, and so do not change upon varying the connection.

We will compare the first term with the same type of term obtained from $\db_{0,\epsilon}$. On endomorphisms, the differentials are related by \begin{align*}
d_{\epsilon} = d_{0} + [\gamma_{\epsilon} - \gamma_{\epsilon}^*,\underline{\hspace{0.3cm}}]. 
\end{align*}
It follows that
\begin{align*}
\partial_{\epsilon} \db_{\epsilon} - \db_{\epsilon} \partial_{\epsilon}  =&( \partial_{0} \db_{0} - \db_{0} \partial_{0}) + (\db_{0} ( [\gamma_{\epsilon}^*,\underline{\hspace{0.3cm}} ]) - [\gamma_{\epsilon}^*, \db_{0}] )\\
& + (\partial_{0} ( [\gamma_{\epsilon},\underline{\hspace{0.3cm}}] )  - [\gamma_{\epsilon}, \partial_{0}]) + [\gamma_{\epsilon}, [\gamma^*_{\epsilon}, \underline{\hspace{0.3cm}}] ] - [\gamma_{\epsilon}^*, [ \gamma_{\epsilon},\underline{\hspace{0.3cm}}] ].
\end{align*}

When we apply this to $\Id_{Q_i}$ for some $i$, which is in the kernel of $\partial_{0} \db_{0} - \db_{0} \partial_{0}$, the contribution from the two middle pairs of components is off-diagonal in the sense that it lies in $\Omega^2 ( \oplus_{j \neq k} Q_k^* \otimes Q_j )$. This remains true upon wedging with powers of $\omega$ and $\tilde U$. Thus after projecting to $\oplus_j \langle \Id_{Q_j} \rangle$, these terms vanish and hence play no role.

This implies that the leading order contribution of $P_{\epsilon} ( \Id_{Q_i})$, in the subspace $\Omega^{2n}(\oplus_j  \langle \Id_{Q_j} \rangle) \subset \Omega^{2n}(\End E)$ of interest to us, comes from a positive constant multiple of 
$$
\Lambda_{\omega} i \left( [\gamma_{\epsilon}, [\gamma^*_{\epsilon}, \Id_{Q_i} ] ] - [\gamma_{\epsilon}^*, [ \gamma_{\epsilon}, \Id_{Q_i}] ] \right) \otimes \omega^n.
$$
For ease of notation in what follows, we will identify top degree forms with degree zero forms through the volume form $\omega^n$.

The projection to $\langle \Id_{Q_j} \rangle $ of $ [\gamma_{\epsilon}, [\gamma^*_{\epsilon}, \Id_{Q_i} ] ] - [\gamma_{\epsilon}^*, [ \gamma_{\epsilon}, \Id_{Q_i}] ] $ is $\frac{1}{\rk (Q_j)}$ times a universal positive constant multiple of its trace with respect to $Q_j$. While they differ in their off-diagonal contributions, the diagonal contribution of $ [\gamma_{\epsilon}, [\gamma^*_{\epsilon}, \Id_{Q_i} ] ]$ is
$$
\sum_{j< i} \left( (\gamma_{\epsilon})_{j i} \wedge(\gamma_{\epsilon}^*)_{ij} - (\gamma_{\epsilon}^*)_{ij} \wedge (\gamma_{\epsilon})_{j i}  \right) + \sum_{j>i} \left(  (\gamma_{\epsilon})_{ij} \wedge (\gamma_{\epsilon}^*)_{ji} - (\gamma_{\epsilon})^*_{j i} \wedge (\gamma_{\epsilon})_{ij}\right).
$$
and the diagonal contribution of $ [\gamma_{\epsilon}^*, [ \gamma_{\epsilon}, \Id_{Q_i}] ] $ is the negative of this.

It follows that 
\begin{align*}
\tr_{Q_i} \left( [\gamma_{\epsilon}, [\gamma^*_{\epsilon}, \Id_{Q_i} ] ] \right) =&  -  \sum_{ j < i } \left( \epsilon^{2q(i,j)} \tr_{Q_i} \varsigma_{ij}^* \wedge \varsigma_{ji}  + O(\epsilon^{2q(i,j)+1}) \right) \\
&+ \sum_{j > i  } \left(\epsilon^{2q(i,j)} \tr_{Q_i} \varsigma_{ij} \wedge \varsigma_{ji}^* + O(\epsilon^{2q(i,j)+1})\right) 
\end{align*}
and, for $j \neq i$,
\begin{align*}
\tr_{Q_j} \left( [\gamma_{\epsilon}, [\gamma^*_{\epsilon}, \Id_{Q_i} ] ] \right) =& 
\begin{cases}
0 &\textnormal{ if $(\gamma_{\epsilon})_{ij} =0$}\\
 \epsilon^{2q(i,j)} \tr_{Q_j} \varsigma_{ji} \wedge \varsigma^*_{ij} + O(\epsilon^{2q(i,j)+1}) &\textnormal{ if $i>j$} \\
-\epsilon^{2q(i,j)} \tr_{Q_j} \varsigma^*_{ji} \wedge \varsigma_{ij} + O(\epsilon^{2q(i,j)+1}) &\textnormal{ if $i<j$} .
\end{cases}
\end{align*}
The  same then also holds for the term $[\gamma_{\epsilon}, [\gamma^*_{\epsilon}, \Id_{Q_i} ] ]$.

We now invoke part (i) and (ii) of \cref{lem:extensioncontribution}, which implies that if $i<j$, then $\int_X i\Lambda_{\omega} \tr_{Q_i} \varsigma^*_{ji} \wedge \varsigma_{ij} \omega^n > 0$, and  $\int_X i\Lambda_{\omega} \tr_{Q_j} \varsigma_{ij} \wedge \varsigma_{ji}^* \omega^n = - \int_X i\Lambda_{\omega} \tr_{Q_i} \varsigma^*_{ji} \wedge \varsigma_{ij} \omega^n .$ Combining this with the fact that we are computing the expansion of a \emph{positive} constant multiple of 
$$
\Lambda_{\omega} \left( [\gamma_{\epsilon}, [\gamma^*_{\epsilon}, \Id_{Q_i} ] ] - [\gamma_{\epsilon}^*, [ \gamma_{\epsilon}, \Id_{Q_i}] ] \right) \otimes \omega^n,
$$
we obtain the result.
\end{proof}

Next, we consider the off-diagonal elements of $\aut (\Gr(E))$.
\begin{lemma}
\label{lemma:off-diag-bound}
Suppose that $\Psi$ is an isomorphism between two of the components $Q_i$ and $Q_k$, i.e. one of the off-diagonal elements of \cref{lem:automsgraded}. Then $\Psi$ satisfies 
$$
 -\langle P_{\epsilon} (\Psi), \Psi \rangle \geq C \epsilon^{2q} \| \Psi \|^2.
$$
\end{lemma}
\begin{proof}
Our argument is similar to the above, but specialised to the case of off-diagonal isomorphisms. Thus suppose we have one such $\Psi : Q_k \to Q_i$ (where $k$ labels the component $Q_k$ and is unrelated to $\epsilon$). We will assume $i<k$, as the bound for the inverse is similar. Again, the crucial point is to consider $ [\gamma_{\epsilon}, [\gamma^*_{\epsilon}, \Psi  ] ]$ and  $ [\gamma_{\epsilon}^*, [ \gamma_{\epsilon}, \Psi ] ] $.

We thus calculate $[\gamma_{\epsilon}, [\gamma^*_{\epsilon}, \Psi  ] ] $ and $[\gamma_{\epsilon}^*, [ \gamma_{\epsilon}, \Psi ] ]$. First,
\begin{align*}
[ \gamma_{\epsilon}, \Psi ] =& \sum_{a,b \mid a<b} \left( (\gamma_{\epsilon})_{a,b} \circ \Psi - \Psi \circ  (\gamma_{\epsilon})_{a,b} \right),\\
=& \sum_{a \mid a <i} (\gamma_{\epsilon})_{a,i} \circ \Psi  - \sum_{b \mid b>k} \Psi \circ  (\gamma_{\epsilon})_{k,b},
\end{align*}
therefore
\begin{align*}
[ \gamma_{\epsilon}^*, [ \gamma_{\epsilon}, \Psi ] ]=& \sum_{u,v \mid u>v} \left[ (\gamma_{\epsilon}^*)_{u,v} , \left( \sum_{a \mid a <i} (\gamma_{\epsilon})_{a,i} \circ \Psi  - \sum_{b \mid b>k} \Psi \circ  (\gamma_{\epsilon})_{k,b} \right) \right].
\end{align*}
We deal with each term separately.

Note
\begin{align*}
&\sum_{u,v \mid u>v} \left[ (\gamma_{\epsilon}^*)_{u,v} ,  \sum_{a \mid a <i} (\gamma_{\epsilon})_{a,i} \circ \Psi \right]\\
=& \sum_{a \mid a< i} \sum_{u \mid u>a} (\gamma_{\epsilon}^*)_{u,a} \wedge (\gamma_{\epsilon})_{a,i} \circ \Psi + \sum_{a \mid a < i} \sum_{v \mid v<k }  (\gamma_{\epsilon})_{a,i} \circ \Psi  \wedge (\gamma_{\epsilon}^*)_{k,v}.
\end{align*}
The former term can contribute with $Q_k^* \otimes Q_i$-components (which is where $\Psi$ lives), while the latter cannot. In fact, the contribution to the $Q_k^* \otimes Q_i$ component of the above is 
\begin{align*}
\sum_{a \mid a< i}  (\gamma_{\epsilon}^*)_{i,a} \wedge (\gamma_{\epsilon})_{a,i} \circ \Psi .
\end{align*}
On the other hand,
\begin{align*}
&\sum_{u,v \mid u>v} \left[ (\gamma_{\epsilon}^*)_{u,v} , \sum_{b \mid b>k} \Psi \circ  (\gamma_{\epsilon})_{k,b}  \right] \\
=& \sum_{ u \mid u> i}  \sum_{b \mid b>k}  (\gamma_{\epsilon}^*)_{u,v} \wedge  \Psi \circ  (\gamma_{\epsilon})_{k,b} + \sum_{b \mid b>k } \sum_{v \mid v<b} \Psi \circ  (\gamma_{\epsilon})_{k,b} \wedge (\gamma_{\epsilon}^*)_{b,v} .
\end{align*}
Thus, here the contribution to the $Q_k^* \otimes Q_i$ component is
\begin{align*}
\sum_{b \mid b>k } \Psi \circ  (\gamma_{\epsilon})_{k,b} \wedge (\gamma_{\epsilon}^*)_{b,k}  .
\end{align*}
It follows that the overall contribution of $[ \gamma_{\epsilon}^*, [ \gamma_{\epsilon}, \Psi ] ]$ to the $Q_k^* \otimes Q_i$ component is
$$
\sum_{a \mid a< i}  (\gamma_{\epsilon}^*)_{i,a} \wedge (\gamma_{\epsilon})_{a,i} \circ \Psi  - \sum_{b \mid b>k } \Psi \circ  (\gamma_{\epsilon})_{k,b} \wedge (\gamma_{\epsilon}^*)_{b,k}.
$$

We now repeat the above procedure for $[\gamma_{\epsilon}, [\gamma^*_{\epsilon}, \Psi  ] ] $. We have
\begin{align*}
[ \gamma_{\epsilon}^*, \Psi ] =& \sum_{a,b \mid a<b} \left( (\gamma_{\epsilon}^*)_{b,a} \circ \Psi - \Psi \circ  (\gamma_{\epsilon}^*)_{b,a} \right) \\
=& \sum_{b \mid b >i} (\gamma_{\epsilon}^*)_{b,i} \circ \Psi  - \sum_{a \mid a<k} \Psi \circ  (\gamma^*_{\epsilon})_{k,a}
\end{align*}
and so
\begin{align*}
[ \gamma_{\epsilon}, [ \gamma_{\epsilon}^*, \Psi ] ]=& \sum_{u,v \mid u<v} \left[ (\gamma_{\epsilon})_{u,v} , \left( \sum_{b \mid b >i} (\gamma_{\epsilon}^*)_{b,i} \circ \Psi  - \sum_{a \mid a<k} \Psi \circ  (\gamma^*_{\epsilon})_{k,a}\right) \right].
\end{align*}
Similarly to the first calculation,
\begin{align*}
&\sum_{u,v \mid u<v} \left[ (\gamma_{\epsilon})_{u,v} , \sum_{b \mid b >i} (\gamma_{\epsilon}^*)_{b,i} \circ \Psi  \right] \\
=& \sum_{b \mid b >i} \sum_{u \mid u<b} (\gamma_{\epsilon})_{u,b} \wedge (\gamma_{\epsilon}^*)_{b,i} \circ \Psi + \sum_{b \mid b >i} \sum_{v \mid v> k} (\gamma_{\epsilon}^*)_{b,i} \circ \Psi \wedge (\gamma_{\epsilon})_{k,v}
\end{align*}
whose  contribution to the  $Q_k^* \otimes Q_i$ component is
$$
 \sum_{b \mid b >i} (\gamma_{\epsilon})_{i,b} \wedge (\gamma_{\epsilon}^*)_{b,i} \circ \Psi ,
$$
and 
\begin{align*}
& \sum_{u,v \mid u<v} \left[ (\gamma_{\epsilon})_{u,v} ,  \sum_{a \mid a<k} \Psi \circ  (\gamma^*_{\epsilon})_{k,a} \right] \\
=& \sum_{u \mid u< i} \sum_{a \mid a<k}  (\gamma_{\epsilon})_{u,i} \wedge  \Psi \circ  (\gamma^*_{\epsilon})_{k,a}  + \sum_{a \mid a<k} \sum_{v \mid v>a} \Psi \circ  (\gamma^*_{\epsilon})_{k,a} \wedge (\gamma_{\epsilon})_{a,v},
\end{align*}
whose  contribution to the $Q_k^* \otimes Q_i$ component is
$$
 \sum_{a \mid a<k}  \Psi \circ  (\gamma^*_{\epsilon})_{k,a} \wedge (\gamma_{\epsilon})_{a,k}.
$$

The upshot is that the leading order contribution to the $ Q_k^* \otimes Q_i$ component when applying $\Psi$ arises from
\begin{align*}
& - \sum_{a \mid a< i}  (\gamma_{\epsilon}^*)_{i,a} \wedge (\gamma_{\epsilon})_{a,i} \circ \Psi  + \sum_{b \mid b>k } \Psi \circ  (\gamma_{\epsilon})_{k,b} \wedge (\gamma_{\epsilon}^*)_{b,k} \\
& +  \sum_{b \mid b >i} (\gamma_{\epsilon})_{i,b} \wedge (\gamma_{\epsilon}^*)_{b,i} \circ \Psi - \sum_{a \mid a<k}  \Psi \circ  (\gamma^*_{\epsilon})_{k,a} \wedge (\gamma_{\epsilon})_{a,k}.
\end{align*}

It follows that $-\langle P_{\epsilon} (\Psi), \Psi \rangle$ is a positive multiple of the integral over $X$ of
\begin{align*}
& i\Lambda_{\omega} \tr_{Q_i}  \left( \sum_{a \mid a< i}  (\gamma_{\epsilon}^*)_{i,a} \wedge (\gamma_{\epsilon})_{a,i} - \sum_{b \mid b >i}  (\gamma_{\epsilon})_{i,b} \wedge (\gamma_{\epsilon}^*)_{b,i}   \right) \\
+&i\Lambda_{\omega} \tr_{Q_k} \left( \sum_{a \mid a<k}   (\gamma^*_{\epsilon})_{k,a} \wedge (\gamma_{\epsilon})_{a,k}- \sum_{b \mid b>k }  (\gamma_{\epsilon})_{k,b} \wedge (\gamma_{\epsilon}^*)_{b,k} \right) .
\end{align*}
By \cref{lem:extensioncontribution}, these terms are all \emph{positive}. Moreover, there exists at least one index which occurs at order $\epsilon^{q'}$, where $q' \leq q$, since $q(i,j) \leq q$ for some $j$, and similarly for $k$. 
\end{proof}

Next, we consider the diagonal elements of the automorphism group of the graded object.

\begin{proposition}\label{prop:bound-on-linearised-diagonal}
Let $\db_{\epsilon}$ be a deformation of $(\oplus_i Q_i, \db_{0})$ as above such that the holomorphic vector bundle $(\oplus_i Q_i, \db_{\epsilon})$ is irreducible. Let $P_{\epsilon}$ be the linearisation of the $Z_{\epsilon}$-critical operator at $\db_{\epsilon}$. Then there exists a constant $C>0$ such that for all $0 < \epsilon \ll 1$ and all $s = \sum_{i=1}^{l} \frac{t_i}{\rk (Q_i)} \Id_{Q_i}$ orthogonal to $\Id_E$,
$$
\| P_{\epsilon} ( s) \| \geq C \sum_i \epsilon^{2q_i} | t_i |
$$
In particular,
$$
\| P_{\epsilon} ( s) \| \geq C \epsilon^{2q} \| s \|
$$
for all $s \in \oplus_i \langle \Id_{Q_i} \rangle$ orthogonal to $\langle \Id_E \rangle$. 
\end{proposition}
Note that by self-adjointness and since $P_{\epsilon}$ has only non-positive eigenvalues, the estimate is equivalent to the estimate

\begin{align}
\label{eq:keybound}
-\langle P_{\epsilon} (s), s \rangle \geq
C \sum_i \epsilon^{2q_i} | t_i |^2.
\end{align}

\begin{remark}
\label{rem:zerogammas} 
If $i$ is an index such that $(\gamma_{\epsilon})_{ik} = 0$ for all $k$, then $Q_i$ is a holomorphic quotient of $E$; we temporarily denote by $\mathcal{S}_i$ the bundle for which $E/\mathcal{S}_i = Q_i$. By intersecting a given filtration with $\mathcal{S}_i$, we can then reorder the $Q_j$ so that $Q_i$ is the final quotient. In particular, we assume without loss of generality that the maximal $i$ such that  $(\gamma_{\epsilon})_{i, l} \neq 0$ is maximal among all columns in the matrix representation of $\gamma_{\epsilon}$, i.e. $(\gamma_{\epsilon})_{k p} =0$ for all $k> i$ and for all $p$. 
\end{remark}

For each $i$, let $J_i = \{ j : q(i,j) = q_i \}$. Note that for each $j \in J_i$, $q_j \leq q_i$. Also, it may be that $j \in J_i$, but $i \notin J_j$. In fact, for each $j \in J_i$, $i \in J_j$ if and only if $q_i = q_j$. The next goal is to prove the following estimate.
\begin{proposition}
\label{prop:linop}
There exists a $C>0$ such that for all $0< \epsilon \ll 1$ 
$$
-\langle P_{\epsilon} (\sigma), \sigma \rangle \geq C \sum_{i=1}^{l} \epsilon^{2q_i} \sum_{j \in J_i} (c_i - c_j)^2
$$
for any $\sigma = \sum_i c_i \Id_{Q_i}.$
\end{proposition}
\begin{proof}
By \cref{lemma:linearisationexpansion}
\begin{align*}
-\langle P_{\epsilon}(\sigma) , \Id_{Q_i} \rangle =&  \sum_{j \neq i} \left( c_{i,j}\epsilon^{2q(i,j)} (c_i - c_j) + O(\epsilon^{2q(i,j)+1}|c_i -c_j|) \right),
\end{align*}
from which it follows that
\begin{align*}
-\langle P_{\epsilon}(\sigma) , \sigma \rangle =& \sum_i c_i  \sum_{j \neq i} \left( c_{ij} \epsilon^{2q(i,j)} (c_i - c_j) + O(\epsilon^{2q(i,j)+1}|c_i -c_j|) \right).
\end{align*}
Thus for each $i$ and each $j \neq i$, we obtain a term $c_i c_{ij} \epsilon^{2q(i,j)} (c_i - c_j)$ from the first summand at $i$ and $c_j c_{ji} \epsilon^{2q(i,j)} (c_j - c_i)$ from the first summand at $j$. These sum to 
$$
c_{ij} \epsilon^{2q(i,j)} (c_i - c_j)^2.
$$
Since $c_{ij}> 0$, this gives precisely the leading order term claimed in the estimate. Note that if $q_i = q_j$, we overcount this term, but this can be accounted for by adjusting $C$ (just dividing it by $2$). Since we have a similar pairing also for the higher order terms, we see that the error terms of $\langle P_{\epsilon}(\sigma) , \sigma \rangle$ are $O(\epsilon^{2q(i,j)} (c_i-c_j)^2)$, which can then be absorbed in the leading order term for sufficiently small $\epsilon$.
\end{proof}

Next, we show how the above implies the bound we seek. For this we need the following lemma. 
\begin{lemma}
\label{lem:quadraticform}
Suppose $E$ is irreducible. Then the quadratic form
$$\sum_{i=1}^{l} c_i e_i \mapsto \sum_{i=1}^{l} \epsilon^{2q_i} \sum_{k \in J_i} (c_i - c_k)^2$$
on $\C^{l}$ with basis $e_1, \ldots, e_{l}$, has signature $(l-1,0)$. 
\end{lemma}
Note that $\sum_{i=1}^l e_i$ is in the kernel of this quadratic form, so the lemma will imply that this is precisely its kernel.
\begin{proof}
Let $V(E)$ be the vector subspace of $\mathbb{C}^{l}$ spanned by 
$$
\{e_i - e_k : i \in \{1, \ldots, l \} \textnormal{ and } k \in J_i \}.
$$
Then the quadratic form is positive on each of the lines spanned by the $e_i -e_k$, and so the signature is $(l-1,0)$ if and only if $V(E)$ has dimension $l-1$ (note that the quadratic form has at least a one-dimensional kernel). The proof is via induction on the number of components $Q_k$. In the base case $k=0$, there is nothing to prove. For subbundles $S$ of $E$ defined via only some of the $Q_i$, we make the analogous definition $V(S)$ of the vector space spanned by the relevant $e_i - e_k$. 

The goal is to reduce to an irreducible subbundle of $E$ made out of only some of the $Q_i$. The proof is divided into three cases, depending on the precise form of the extension parameter $\gamma$. 

\textbf{Case 1:} $S_{l-1}$ is reducible.

In this case, there are multiple $i$ with $q(i, l) = q_{l},$ i.e. $|J_{l}| >1$. Moreover, $S_{l-1}$ splits as $\oplus_{i\in J_{l}} \check{S}_i$, where each of the $\check{S}_i$ is a direct sum of some of the $Q_k$ as a smooth vector bundle, but with holomorphic structure such that $\check{S}_i$ is irreducible. For each $i \in J_{l}$, let $K_i$ denote the set of indices $k$ such that $Q_k$ is a component of $\check{S}_i$. 

First note that $\check{S}_i \oplus Q_{l}$ deforms to an irreducible bundle $E_i$. On $\check{S}_i$, the holomorphic structure is obtained by restricting  $\db_{\epsilon}$ to $\left( \oplus_{k \in K_i} Q_k \right)$. We can then construct the holomorphic structure of $E_i$ by using only the $(\gamma_{\epsilon})_{k, l}$ with $k \in K_i$. This still satisfies the Maurer--Cartan equation, because of the splitting of $S_{l-1}$. The only term involving $Q_{l}$ at the order $\epsilon^{2q_{l}}$ is that coming from $(\gamma_{\epsilon})_{i,l}$. 

We use induction on the number of components, and we use the assumption on the $E_i$, i.e. that $\dim(V(E_i)) = |K_i|$ for each $i \in J_{l}$. Note also that the $V(E_i)$ are linearly independent, since the only common component is the final component. Moreover, as every $k < l$ is in $K_i$ for some $i$, we have that $\sum_{i\in J_{j+1}} |K_i|= l-1$ and $V(E) = \oplus_{i \in J_{l}} V(E_i)$. Thus $\dim V(E) = l-1$, which is what we wanted to show.

\textbf{Case 2:} $S_{l-1}$ is irreducible and $(\gamma_{\epsilon})_{l-1,l} \neq 0$. 

In this case, $J_{l} = \{ l-1 \}$. Thus, the span of 
$$
\{e_i - e_k : i \in \{1, \ldots, l \} \textnormal{ and } k \in J_i \} 
$$
is the sum of the span of 
$$
\{e_i - e_k : i \in \{1, \ldots, l-1 \} \textnormal{ and } k \in J_i \} 
$$
and the span of $e_{l}-e_{l-1}$. Note that these two vector spaces do not intersect non-trivially and so this is a direct sum. Also, the former vector space is the one associated to $S_{l-1}$. Thus the dimension of $V(E)$ is $1+ \dim V(S_{l-1})$. By the inductive assumption, $\dim (V(S_{l-1}))=l-2$ and so $\dim (V(E)) = l-1$, as required.

\textbf{Case 3:}  $S_{l-1}$ is irreducible and $(\gamma_{\epsilon})_{l-1,l} = 0$.

Choose the largest $p$ such that $(\gamma_{\epsilon})_{p,l} \neq 0$. Note that since by \cref{rem:zerogammas} we can assume that $(\gamma_{\epsilon})_{k i}=0$ for all $k>p$ and all $i$, we have that $Q_{p+1}, \ldots, Q_{l}$ are holomorphic quotients of $E$. In particular, we can pass to another filtration that reorders the quotients after $Q_p$ so that $Q_k$ becomes $Q_{l}$. If we can apply Case 1 to any such $Q_k$, we have made the reduction to a smaller bundle as we wanted. 

We may therefore assume that this is not possible, which means that for each $k>p$, there is exactly one $i=i_k$ (which is necessarily at most $p$) such that $q(i,k) = q_k$. Again, after reordering we may assume that all the $k$ with $(\gamma_{\epsilon})_{pk} \neq 0$ are last. That is, we may assume that there is a $b$ such that $(\gamma_{\epsilon})_{p, p+1}, \ldots, (\gamma_{\epsilon})_{p, b}$ are all zero and the $(\gamma_{\epsilon})_{p,k}$ with $k>b$ are non-zero and $O(\epsilon^{q_k})$. Note that $b \geq p+1$ and equality occurs when none of these terms are vanishing. Also, $b \leq l-1$, since $(\gamma_{\epsilon})_{p, l} \neq 0$. Note that all of the bundles $S_{p+1}, \ldots , S_{l-1}$ are irreducible. 

The upshot is that $J_{b+1} = J_{b+2} = \ldots = J_{l} = \{ p \}$. Thus, the span of 
$$
\{e_i - e_k : i \in \{1, \ldots, l \} \textnormal{ and } k \in J_i \} 
$$
is the sum of the span of
$$
\{e_i - e_k : i \in \{1, \ldots, b \} \textnormal{ and } k \in J_i  \}
$$
and the span of the $l-b$ vector spaces $\langle e_{b+i}- e_p \rangle$ for $i=1, \ldots, l-b$. Note that these vector spaces do not intersect non-trivially and so this is a direct sum. Also, the former vector space is the vector space $V(S_{b})$ associated to $S_{b}$. Thus the dimension of $V$ is $l-b+ \dim V(S_{b})$. Since $S_b$ is irreducible, we can apply the inductive assumption, and $\dim (V(S_{b}))=b-1$. Thus  $\dim (V(E)) = (l-b)+(b-1) = l-1$, as required.

Since one of the three cases above has to occur, the proof is complete, by induction.
\end{proof}

The above allows us to deduce that $P_{\epsilon}$ has the claimed bound on   the Lie algebra of the automorphism group of the graded object. 
\begin{corollary}
There exists a $C>0$ such that for all $0< \epsilon \ll 1$
$$
-\langle P_{\epsilon} (\sigma), \sigma \rangle \geq C \epsilon^{2q} \| \sigma \|^2
$$
for any $\sigma \in \langle \Id_E \rangle^{\perp} \cap \aut(\Gr(E)).$
\end{corollary}
\begin{proof}
By \cref{prop:linop} and the fact that $q_i \leq q$ for all $i$, there is a $C_1 >0$ such that
$$
-\langle P_{\epsilon} (\sigma), \sigma \rangle \geq C_1 \epsilon^{2q} \sum_{i=1}^{l} \sum_{j \in J_i} (c_i - c_j)^2
$$
By \cref{lem:quadraticform}, the quadratic form $\sum_{i=1}^{l} \sum_{j \in J_i} (c_i - c_j)^2$ is positive definite on $\langle \Id_E \rangle^{\perp}$. In particular, there is a $C_2>0$ such that 
$$
\sum_{i=1}^{l} \sum_{j \in J_i} (c_i - c_j)^2 \geq C_2 \sum_{i=1}^{l} c_i^2
$$ 
on $\langle \Id_E \rangle^{\perp}.$ The result follows.
\end{proof}

At this stage, we have shown the bound we want on the off-diagonal elements of $\aut(\Gr(E))$, the diagonal elements of $\aut(\Gr(E))$ orthogonal to $\Id_E$ and $\aut(\Gr(E))^{\perp}$, individually. We next show that this bound holds on all of $\langle \Id_E \rangle^{\perp}$, which requires showing that the  mixed terms do not contribute to the bound.  We begin by showing that  the bound holds on $\aut(\Gr(E)) \cap \langle \Id_E \rangle^{\perp}.$
\begin{lemma}
\label{lemma:nomixedterms}
If $\Psi$ is an off-diagonal and $\Xi$ is a diagonal element of $\aut ( \Gr(E))$, then

there exists a $C>0$ such that for all $0 < \epsilon \ll 1$
$$
-\langle P_{\epsilon} (\Psi + \Xi) , \Psi + \Xi \rangle \geq -C\left( \langle P_{\epsilon} (\Psi ) , \Psi  \rangle + \langle P_{\epsilon} ( \Xi) ,  \Xi \rangle \right).
$$
\end{lemma}
\begin{proof}
It suffices to consider the case when $\Psi$ is an isomorphism $Q_k \to Q_i$ and we then continue with the notation from \cref{lemma:off-diag-bound}. We first consider the terms $ \sum_{a \mid a< i} \sum_{u \mid u>a} (\gamma_{\epsilon}^*)_{u,a} \wedge (\gamma_{\epsilon})_{a,i} \circ \Psi $. If one of these terms is diagonal, then since $\Xi : Q_k \to Q_i$, we need $u=k$. So we may assume $\Xi = c \Id_{Q_k}$ for some constant $c$ as any other component of $\Xi$ will be orthogonal to this term, since the Hermitian metric is a product with respect to the decomposition $\oplus_i Q_i$. We will then show that
\begin{align}
\label{eq:phipsimixed}
\langle P_{\epsilon} (\Psi + \Xi) , \Psi + \Xi \rangle = \langle P_{\epsilon} (\Psi ) , \Psi  \rangle +\langle P_{\epsilon} ( \Xi) ,  \Xi \rangle + O(\epsilon^{q_i  + q_k +1}).
\end{align}
This suffices to prove the result for this part of the of the expansion of $P_{\epsilon}(\Psi)$ because by the proof of \cref{lemma:off-diag-bound}, $\langle P_{\epsilon} (\Psi), \Psi \rangle$ is a negative multiple of $\epsilon^{2\min \{ q_i, q_k \}}$. Thus the result follows by completing the square.

We are interested in the pairing $\langle P_{\epsilon} (\Psi + c \Id_{Q_k} ) , \Psi + c \Id_{Q_k}\rangle$ for a constant $c$. If we then consider the contribution in this direction for $\Id_{Q_k},$ there is  a corresponding component $(\gamma_{\epsilon}^*)_{i,a} \wedge (\gamma_{\epsilon})_{a, k}$ coming from $[\gamma_{\epsilon}^* , [\gamma_{\epsilon}, \Id_{Q_k} ] ]$. This also has a positive sign, hence when computing the mixed terms $\langle P_{\epsilon} ( \Psi ) ,  \Id_{Q_k} \rangle + \langle P_{\epsilon} ( \Id_{Q_k} ) , \Psi  \rangle$, the leading order term is a positive constant multiple of 
$$
\sum_a \int_X i\Lambda_{\omega} \left( \tr_{Q_k}((\gamma_{\epsilon}^*)_{k,a} \wedge (\gamma_{\epsilon})_{a,i} \circ \Psi) - \tr_{Q_i}( (\gamma_{\epsilon}^*)_{i,a} \wedge (\gamma_{\epsilon})_{a, k} ) \circ \Psi^{-1}  \right) \omega^n.
$$
But, $\tr_{Q_i}(( (\gamma_{\epsilon}^*)_{i,a} \wedge (\gamma_{\epsilon})_{a, k} ) \circ \Psi^{-1}) = \tr_{Q_k}( \Psi \circ (\gamma_{\epsilon}^*)_{i,a} \wedge (\gamma_{\epsilon})_{a, k} )$, and so the above terms are traces of commutators, hence $0$. Moreover, each of the terms $(\gamma_{\epsilon}^*)_{k,a} \wedge (\gamma_{\epsilon})_{a,i}$ and $(\gamma_{\epsilon}^*)_{i,a} \wedge (\gamma_{\epsilon})_{a k}$ are at least $O(\epsilon^{q_i+q_k})$. Since the above vanishes, we therefore get that the mixed terms are $O(\epsilon^{q_i+q_k+1})$, giving us \cref{eq:phipsimixed} and hence the result.

The argument for the  remaining terms is similar. Note that we do not need to consider the terms coming from e.g. $\partial_{0} ( [\gamma_{\epsilon}, \Psi])$, since, by the Nakano identities,
\begin{align*}
\langle i\Lambda_{\omega} (\partial_{0} ( [\gamma_{\epsilon}, \Psi] )), \Xi \rangle =& - \langle \db_0^* ( [\gamma_{\epsilon}, \Psi] ), \Xi \rangle \\
=&\langle  ( [\gamma_{\epsilon}, \Psi] ), \db_0 \Xi \rangle \\
=& 0,
\end{align*}
since $\Xi$ is in the kernel of $\db_0$.
\end{proof}

With all this in place, we can now prove the main result of this section, which gives a lower bound on the operator norm of the linearised operator $P_{\epsilon}$. 
\begin{proposition}\label{prop:bound-on-linearised}
Let $\db_{\epsilon}$ be a deformation of $(\oplus_i Q_i, \db_{0})$ as above. Let $P_{\epsilon}$ be the linearisation of the $Z$-critical operator at $\db_{\epsilon}$. Then there exists a constant $C>0$ such that for all $0 < \epsilon \ll 1$, 
$$
\| P_{\epsilon} ( s) \| \geq C \epsilon^{2q} \| s \|
$$
for all $s$ orthogonal to $\langle \Id_E \rangle$. Moreover, this bound also holds for any other $\epsilon$-dependent family $\db_{\epsilon}' = \db_0 + \gamma_{\epsilon}'$ of holomorphic structures such that $(\gamma_{\epsilon}')_{ij} = (\gamma_{\epsilon})_{ij} + O(\epsilon^{q_i + 1})$ for every $(i,j)$ with $q(i,j) = q_i$. 
\end{proposition}
\begin{proof}
It remains to prove that if $s \in \left( \oplus_i \langle \Id_{Q_i} \rangle \right)^{\perp}$ and $t \in \left(\oplus_i \langle \Id_{Q_i} \rangle \right) \cap \langle \Id_E \rangle^{\perp}$, then the estimate holds for $s+t$. Thus we need to estimate the mixed terms $\langle P_{\epsilon} (s) , t \rangle $ and $\langle P_{\epsilon} (t) , s \rangle $. In fact, the orthogonality of $t$ to $\Id_E$ will only feature in order to obtain the required bound for $\langle P_{\epsilon} (t) , t \rangle $ and we do not need this assumption for the estimate of the mixed terms.

We next prove the required bound for each $t= \Id_{Q_i}$ separately. Note that $P_{\epsilon}(\Id_{Q_i})$ is $O(\epsilon^{2q_i})$. We will now show that the mixed terms are $O(\epsilon^{q_i +1} \|s\| \|t \|)$. This will give the required result, since by completing the square, any such term (regardless of sign) can be bounded below by a negative multiple of $\epsilon  \|s\|^2 + \epsilon^{2q_i+1} \|t\|^2$.  

We first consider $\langle P_{\epsilon}(t), s \rangle$. From the expansion in \cref{lemma:tlaplacian}, similarly to the proof of \cref{lemma:linearisationexpansion}, we see that the leading order term in the expansion of $\langle P_{\epsilon}(t), s \rangle$ is the $O(\epsilon^{q_i})$-term
$$\left\langle i\Lambda_{\omega} \left( (\db_{0} ( [\gamma_{\epsilon}^*, t ]) - [\gamma_{\epsilon}^*, \db_{0}](t) ) + (\partial_{0} ( [\gamma_{\epsilon}, t] ) - [\gamma_{\epsilon}, \partial_{0}] (t)) \right), s \right\rangle,$$ 
since $t$ is in the kernel of $\Delta_{\Gr(E)}$. But since $t$ is in the kernel of $\db_0$ and $\partial_0$, the term in the left slot only involves $\Lambda_{\omega} (\db_0 a_{\epsilon})$ and $\Lambda_{\omega} (\partial_0 a_{\epsilon})$. Thus part (iii) of \cref{lem:extensioncontribution} guarantees that this vanishes and so $\langle P_{\epsilon}(t), s \rangle$ is an $O(\epsilon^{q_i+1})$-term, which is what we wanted to show.

We next consider $\langle P_{\epsilon}(s), t \rangle$. In this case, the leading order term in the expansion of $P_{\epsilon}(s)$ is $\Delta_0 (s)$, whose inner product with $t$ vanishes since $t$ lies in the kernel of the self-adjoint operator $\Delta_0$. We thus see that the leading order term in $\langle P_{\epsilon} (s), t\rangle$ is given by the $O(\epsilon^{q_i})$-term
$$
\left\langle i\Lambda_{\omega} \left( (\db_{0} ( [\gamma_{\epsilon}^*, s ]) - [\gamma_{\epsilon}^*, \db_{0}](s) ) + (\partial_{0} ( [\gamma_{\epsilon}, s] ) - [\gamma_{\epsilon}, \partial_{0}] (s)) \right), t \right\rangle.
$$ 
As above, the terms $(\partial_{0} ( [\gamma_{\epsilon}, s] )$ and $ (\partial_{0} ( [\gamma_{\epsilon}, s] )$ vanish by part (iii) of \cref{lem:extensioncontribution}. However, we cannot use this for the remaining two terms, since $s$ does not lie in the kernel of $\db_0$ and $\partial_0$. However, by the Nakano identities
\begin{align*}
\Lambda_{\omega} \left( \db_{0} ([\gamma_{\epsilon}^*, s ]) \right) =& [\Lambda_{\omega},\db_{0}]  \left(  [\gamma_{\epsilon}^*, s ] \right)  \\
=& -i (\partial_0)^* \left(  [\gamma_{\epsilon}^*, s ] \right)
\end{align*}
and so, since $t$ lies in the kernel of $\Delta_0$ ,
$$
\left\langle i\Lambda_{\omega} \left( (\db_{0} ( [\gamma_{\epsilon}^*, s ]) \right), t \right\rangle = 0.
$$ 
Similarly
$$
\left\langle i\Lambda_{\omega} \left( \partial_{0} ( [\gamma_{\epsilon}, s] ) \right), t \right\rangle = 0.
$$ 
and so $\langle P_{\epsilon}(s), t \rangle$ is also $O(\epsilon^{q_i+1})$, as required.

The final claim follows because in the proof of the above, the only components that were used were, for each $i$, the components $(i,j)$ achieving the minimum value of $q(i,j)$ over all $j$. \end{proof}

\begin{remark}
\label{rem:linop}
A consequence of all the above is that for any endomorphism of the form $\epsilon^{2q} \tau$ which is orthogonal to $\Id_E$, there exists a $\sigma$ (which may depend in a polynomial fashion on $\epsilon$), such that 
$$
P_{\epsilon} (\sigma) = \epsilon^{2q} \tau + O(\epsilon^{2q+1}).
$$
The fact that the linearised operator is surjective at order $\epsilon^{2q}$ will be important when we construct approximate solutions to the Z-critical equation in the next section.
\end{remark}

The above produces a bound for the linearised operator of the Chern connection of $\db$-operator of the form given in \cref{eq:def}. We next show that the solutions $z_{\oldepsilon, \infty}$ to our finite-dimensional problem in the Kuranishi space behaves in the same manner.
\begin{proposition}
\label{prop:solnasymptotics}
Let $A_{\epsilon}=\Phi(z_{\oldepsilon, \infty})$ be the Chern connection whose $\db$-operator is given by the image of the solutions $z_{\oldepsilon, \infty}$ to the finite dimensional problem in \cref{prop:solfindim}, with $\oldepsilon = \epsilon^2$. Then there exists a constant $C>0$ such that the linearised operators $P_{\epsilon}$ at $A_{\epsilon}$ satisfies for all $0 < \epsilon \ll 1$ that
$$
\| P_{\epsilon} ( s) \| \geq C \epsilon^{2q} \| s \|
$$
for all $s$ orthogonal to $\langle \Id_E \rangle$, where $q$ is the maximal discrepancy order of $E$.
\end{proposition}
We first prove that to leading order, the connections $A_{\epsilon}$ agree with the weak Hermitian Yang--Mills connection on the graded object $\Gr(E)$.
\begin{lemma}
\label{lem:solnleadingasymptotics}
The connections $A_{\epsilon}$ satisfy 
$$
A_{\epsilon}^{0,1} - \db_0 = O(\epsilon).
$$
\end{lemma}
\begin{proof}
Since the map $\Phi$ from the Kuranishi space to $\scA$ is holomorphic and $\Phi(0)$ is the Chern connection of $\db_0$, this is a direct consequence of \cref{flow-bounded-region}, which implies an $O(\epsilon)$ bound for the solutions $z_{\oldepsilon, \infty}$.
\end{proof}
With this in place, we  prove \cref{prop:solnasymptotics}.
\begin{proof}
From \cref{lem:solnleadingasymptotics}, we know that we can write
$$
A_{\epsilon}^{0,1} = \db_0 + \gamma_{\epsilon},
$$
where $\gamma_{\epsilon} = O(\epsilon).$ Moreover, the $A_{\epsilon}$ solve
$$
p(\Ima(e^{-i\phi_{\epsilon}}\tilde Z(A_{\epsilon})) = 0,
$$
where $p$ is the projection to $\mfk$ of the $Z$-critical term of $A_{\epsilon}$. Now, if $S$ is a holomorphic subbundle of $E$ arising from some Jordan--H\"older filtration of $E$, then we can consider $\Id_S \in \aut (\Gr(E))$. Then, in particular, the projection of $\Ima(e^{-i\phi_{\epsilon}}\tilde Z(A_{\epsilon}))$ to $\langle \Id_S \rangle$ vanishes. On the other hand, the projection of $\Ima(e^{-i\phi_{\epsilon}}\tilde Z(A_{0}))$ to $\langle \Id_S \rangle$ is, as in the proof of \cref{lemma:right-constant}, a non-zero multiple of $\Im(e^{-i\varphi_{\varepsilon}(E)}Z_{\epsilon}(S))$. This is $O(\epsilon^{2q_S})$ for some $q_S \leq q$, the maximal discrepancy order $q$ of $E$.
Moreover, the contribution from $\gamma_{\epsilon}^* \wedge \gamma_{\epsilon} + \gamma_{\epsilon} \wedge \gamma_{\epsilon}^*$ gives the difference between $\Ima(e^{-i\phi_{\epsilon}}\tilde Z(A_{\epsilon}))$ and $\Ima(e^{-i\phi_{\epsilon}}\tilde Z(A_{0}))$, and so the projection of $\gamma_{\epsilon}^* \wedge \gamma_{\epsilon} + \gamma_{\epsilon} \wedge \gamma_{\epsilon}^*$ to $\langle \Id_S \rangle$ has to be $O(\epsilon^{2q_S})$. In particular, because this holds for all such $S$, this implies that in the expression of $\gamma_{\epsilon}$, the corresponding collection of $q_i$ (as in \cref{def:bdldiscreporder}) are all bounded above by $q$. Moreover, since there is an $S$ such that $q_S = q$, at least one of the $q_i$ equals $q$. 
The bound is therefore a direct consequence of the last statement of \cref{prop:bound-on-linearised}.
\end{proof}
\begin{remark}
\label{rem:discrepancy}
The above shows that the order of the deformation associated to the $z_{\epsilon, \infty}$ equals the maximal discrepancy order of $E$, since the maximal $q_i$ equals the maximal $q_S$ such that $\Ima(e^{-i\phi_{\epsilon}}\tilde Z(A_{\epsilon})) - \Ima(e^{-i\phi_{\epsilon}}\tilde Z(A_{0})) = O(\epsilon^{2q_S})$ for some subbundle $S$ arising from a Jordan--H\"older filtration of $E$.
\end{remark}

\subsubsection{Constructing solutions}

We now apply the above to construct solutions to the $Z$-critical equation. We begin by constructing approximate solutions.

\begin{proposition}
\label{prop:generalinductive} 

	For each $r$, there exist Hermitian endomorphisms $f_1,\dots, f_{r}$ such that if we define $f_{\epsilon,r}$ by
	$$f_{\epsilon,r}= \prod_{i=1}^r \exp (f_i\epsilon^i),$$
	then 
	$$D_{\epsilon}(f_{\epsilon,r}\cdot  A_{\epsilon}) =  O(\varepsilon^{r+1}).$$ 
Moreover, the linearised operator $P_{\epsilon,r} $ of $f_{\epsilon,r}\cdot  A_{\epsilon}$ satisfies
	\begin{align}
	\label{eqn:Prbound} \| P_{\epsilon,r} (s) \| \geq C \epsilon^{2q} \| s \|, 
	\end{align}
on $\langle \Id_E \rangle^{\perp}$, for some $C>0$ independent of $\epsilon$.
\end{proposition}
\begin{proof}
The proof is by induction on $r$. For $r=0$, we use the initial connection $A_{\epsilon} = \Phi (z_{\oldepsilon, \infty})$ given by \cref{prop:solfindim}. The bounds on the linearised operator hold in this case by \cref{prop:bound-on-linearised}.

We next assume we have proven the result for some fixed $r$, and aim to prove the result for $r+1$. We first consider when $r < 2q$. In this case, we inductively assume that we can write 
$$
D_{\epsilon}(f_{\epsilon,r}\cdot A_{\epsilon}) =  \epsilon^{r+1} s_{r+1} + O(\epsilon^{r+2}),
$$
where $s_{r+1} \in  \aut(\Gr(E))^{\perp}$ and that the projection $\Pi_{\aut(\Gr(E))}\left( D_{\epsilon}(f_{\epsilon,r}\cdot A_{\epsilon}) \right)$ of $D_{\epsilon}(f_{\epsilon,r}\cdot A_{\epsilon})$ orthogonally to $\aut(\Gr(E))$ is $O(\epsilon^{2q+1})$. This holds at $r=0$ because $A_{\epsilon}$ has \emph{vanishing} projection to $\aut(\Gr(E))$. 

The endomorphism $s_{r+1}$ lies in the image of $\Delta_{\Gr(E)}$, the Laplacian of the graded object, and so there exists a $\tilde \sigma_{r+1} \in \aut(\Gr(E))^{\perp}$ such that $\Delta_{\Gr(E)} (\tilde \sigma_{r+1})=-s_{r+1}$. By \cref{cor:linearisationfirstterm}, we therefore obtain that 
$$
D_{\epsilon}(\exp(\tilde \sigma_{r+1} \epsilon^{r+1}) \cdot f_{\epsilon,r}\cdot A_{\epsilon}) = O(\epsilon^{r+2}).
$$

It remains to ensure that the orthogonal projection $\Pi_{\aut(\Gr(E))}$ to $\aut(\Gr(E))$ is $O(\epsilon^{2q+1})$, which is part of our inductive hypothesis in the situation $r < 2q$ under consideration. This may not be the case if we choose $\exp(\tilde \sigma_{r+1} \epsilon^{r+1}) \cdot f_{\epsilon,r}\cdot A_{\epsilon}$ as our connection, as $P_{\epsilon}(\tilde \sigma_{r+1})$ does not necessarily remain orthogonal to $\aut(\Gr(E))$ to all orders. We thus modify the connection further to ensure this holds. Write 
\begin{align*}
\Pi_{\aut(\Gr(E))} \left( D_{\epsilon}(\exp(\tilde \sigma_{r+1} \epsilon^{r+1})  \cdot f_{\epsilon,r}\cdot A_{\epsilon}) \right) =  \sum_{i=1}^{2q-r-1} \epsilon^{r+1+i} \tau_{r+1+i} + O(\epsilon^{2q+1}),
\end{align*}
where $\tau_{r+1+i}  \in \aut(\Gr(E)) \cap \langle \Id_E \rangle^{\perp}$. Choose the minimal $i$ such that $\tau_{r+1+i} \neq 0$. From the proof of \cref{prop:bound-on-linearised}, we see that if such a $\tau_{r+1+i}$ appears  in the above, then $\langle P_{\epsilon}(\sigma_{r+1}), \tau_{r+1+i} \rangle$ is $O(\epsilon^{i}).$ The proof moreover shows that this can only happen if $ \tau_{r+1+i}$ itself is in the image of $P_{\epsilon}$ at order $\epsilon^{i}$. Thus we can alter $\tilde \sigma_{r+1}$ by applying another gauge transformation which is at least $O(\epsilon^{r+1})$ and that removes the $\epsilon^{r+1+i} \tau_{r+1+i}$ term. This may alter the $\tau_{r+1+j}$ with $j>i$, but again, if this happens, then the altered $\tau_{r+1+j}$ will still lie in the image of $P_{\epsilon}$ at the required order. Continuing like this, we eventually obtain a $\sigma_{r+1}$ that may depend on $\epsilon$ in a polynomial fashion, such that $f_{\epsilon,r+1}\coloneqq  \exp(\sigma_{r+1} \epsilon^{r+1}) \cdot f_{\epsilon,r}$ satisfies the required conditions. This settles the case when $r< 2q$.

We finally consider the case $r\geq 2q$, which is simpler. In this case, the $Z_{\epsilon}$-critical operator satisfies
$$
D_{\epsilon}(f_{\epsilon,r}\cdot A_{\epsilon}) = \epsilon^{r+1} s_{r+1} + O(\epsilon^{r+2}),
$$
where $s_{r+1} \in  \langle \Id_E \rangle^{\perp}.$ As explained in \cref{rem:linop}, as by assumption $r\geq 2q$,  there exists a $\sigma_{r+1}$, which may depend in a polynomial fashion on $\epsilon$, such that $P_{\epsilon, r} (\sigma) = -\epsilon^{2r} s_{r+1} + O(\epsilon^{2r+1}).$ Setting $f_{\epsilon,r+1}= \exp(\epsilon \sigma_{r+1} ) \cdot f_{\epsilon,r},$ we obtain $$
D_{\epsilon}(f_{\epsilon,r+1}\cdot A_{\epsilon}) = O(\epsilon^{r+2}),
$$
as required.

Finally, the mapping properties of $P_{\epsilon}$ were established by considering the leading order contributions of the operator on $\aut(\Gr(E))^{\perp}$ and $\aut(\Gr(E))$. These mapping properties continue to hold  for the operator $P_{\epsilon, r}$, since we have changed any of these components by higher order contributions in $\epsilon$ at each stage. Thus by \cref{prop:bound-on-linearised} we obtain for each $r$ the required $O(\epsilon^{2q})$-bound for the linearised operator $P_{\epsilon,r} $ at $f_{\epsilon,r}\cdot A_{\epsilon}$.
\end{proof}
\begin{remark}
The constant $C$ above for the operator norm of $P_{\epsilon,r}$ may depend on $r$. This is sufficient for our applications, the key being that the order in $\epsilon$ in the above does not depend on $r$. This will later allow us to apply a version of the inverse function theorem, after constructing approximate solutions to the $Z_{\epsilon}$-critical equation to sufficiently high order.

\end{remark}

We now show that we can solve the $Z$-critical equation. 
\begin{theorem}
\label{thm:solns}

For each $\epsilon \ll 1$, there are connections $\nabla_{\epsilon}$ on $E$ such that
$$
D_{E, \epsilon}  (\nabla_{\epsilon}) = 0.
$$
\end{theorem}

We will use the following consequence of the mean value theorem. Let $\mathcal{M}_{\epsilon,r}$ denote the non-linear part $D_{\epsilon, r} - P_{\epsilon, r}$ of the $Z$-critical operator at the connection $f_{r,\epsilon}\cdot \check \nabla_{\epsilon}$ of \cref{prop:generalinductive}. Analogously to e.g. \cite[Lemma 7.1.]{fine}, we then have the following.
\begin{lemma}
\label{lem:lipschitzbound}
There are constants $c,C>0$ such that for all $\epsilon >0$ sufficiently small, we have that for $s_0, s_1 \in L^2_{d+2,0}$ such that $\| s_i\| \leq c$, 
$$
\| \mathcal{M}_{\epsilon,r} (s_0 ) - \mathcal{M}_{\epsilon,r} (s_1) \| \leq C( \| s_0\| + \|s_1\| ) \| s_0 - s_1\|.
$$
\end{lemma}
We will use the above bound in the following version of the quantitative version of the inverse function theorem to prove \cref{thm:solns}; see, for example, \cite[Theorem 4.1]{fine}.
\begin{theorem}
\label{thm:inversefnthm}
Let $F : V \to W$ be a differentiable map of Banach spaces $V, W$, whose linearisation $P = D F$ at $0$ is invertible with inverse $Q$. Let 
\begin{itemize}
\item $\delta'$ be the radius of the closed ball in $V$ such that $F - P$ is Lipschitz of constant $\frac{1}{2 \| Q \|}$;
\item $\delta = \frac{\delta'}{2 \| Q \|}.$
\end{itemize}
Then for all $w \in W$ with $\| w- F (0) \| < \delta$, there exists $v \in V$ such that $F (v) = w.$
\end{theorem}
We can now prove \cref{thm:solns}.

\begin{proof}
We seek to find a root of $D_{\epsilon,r}$ for all $\varepsilon > 0$ sufficiently small, for some $r$, where $D_{\epsilon,r}$ is the $Z$-critical operator at the connection $f_{\epsilon,r}\cdot A_{\epsilon}$ of \cref{prop:generalinductive}.  \cref{lem:lipschitzbound} produces a constant $c>0$ such that for all $\varrho > 0$ sufficiently small, the non-linear part $\mathcal{M}_{\epsilon,r}$ of $D_{\epsilon,r}$ is Lipschitz with Lipschitz constant $c \varrho$ on the ball of radius $\varrho.$ Moreover, $\frac{1}{2 \| Q_{\epsilon, r } \| }$ is bounded below by $C_r \epsilon^{2q}$ for some constant $C_r >0$, by the bound in Equation \eqref{eqn:Prbound}. Thus there is a constant $C'_r$ such that the radius $\delta'$ of the ball on which $D_{\epsilon, r} - P_{\epsilon, r}$ is Lipschitz with Lipschitz constant $\frac{1}{2 \| Q_{\epsilon, r} \|}$ is bounded below by 
	$$
	\delta' \geq C'_r \epsilon^{2q}.
	$$
In turn, using the same bound for $\| Q_{\epsilon, r} \|$, this implies that $\delta = \frac{\delta'}{2 \| Q_{\epsilon,r} \|}$ satisfies 
	$$
	\delta \geq C''_r \epsilon^{4q},
	$$
for a constant $C''_r>0$.

We next take $r=4q$. Then 
$\| D_{\varepsilon,r} (0)\| \leq C'''_r \varepsilon^{4q+1}$ for some $C'''_r$, and so in particular, when $\varepsilon >0$ is sufficently small, $D_{\varepsilon,r} (0) $ is in the ball of radius $C''_r \varepsilon^{4q},$ and hence in the ball of radius $\delta$. By \cref{thm:inversefnthm} applied to $F = D_{\varepsilon,r} $, we can therefore find a root of $D_{\varepsilon, r}$ in $L^2_{d+2,0}$ when $\varepsilon$ is sufficiently small. Elliptic regularity theory then implies that the solution is smooth if $d$ was chosen large enough at the beginning. This completes the proof.
\end{proof}

\bibliography{large-dhym}
\bibliographystyle{alpha}

\vspace{4mm}

\end{document}